\documentclass[12pt]{amsart}%
\usepackage{color}
\usepackage{placeins}
\usepackage{graphicx}
\usepackage{subcaption}
\usepackage{amssymb}
\usepackage{amsmath,amsxtra}
\usepackage{multirow}
\usepackage{enumerate}
\usepackage{epstopdf}
\usepackage{cite,stmaryrd,txfonts}
\usepackage{tikz}
\usepackage{pgf}
\usetikzlibrary{snakes}
\usetikzlibrary{angles,calc,patterns}
\usetikzlibrary{shapes,arrows,automata}
\usepackage{bbm}
\usepackage{dsfont}
\usepackage{accents}
 \usepackage{float}

\usepackage{epic}
\usepackage{empheq}
\usepackage{cases}
\usepackage{diagbox}
\newcommand{\AppendixRef}[1]{Appendix~\ref{#1}}
\newcommand{\AppendicesRef}[2]{Appendices~\ref{#1} and~\ref{#2}}

\def\dx{\,{\rm dx}}

\newtheorem{theorem}{Theorem}[section]
\newtheorem{remark}[theorem]{Remark}
\newtheorem{proposition}[theorem]{Proposition}
\newtheorem{lemma}[theorem]{Lemma}

\usepackage{footnote}
\usepackage{booktabs}
\usepackage{rotating}
\numberwithin{equation}{section}

\usepackage{cite}

\renewcommand{\dfrac}[2]{{
\renewcommand{\arraystretch}{1.375}
\begingroup\displaystyle
\rule[0pt]{0pt}{11pt}#1\endgroup
\over\displaystyle\rule[-3pt]{0pt}{11pt}#2
}}

\usepackage[colorlinks,linkcolor=black,
anchorcolor=black,
citecolor=black]{hyperref}
\usepackage[shortlabels]{enumitem}
\setlist[enumerate]{nosep}
\usepackage{color, soul}
\soulregister\cite7
\usepackage{setspace}

\def\uV{\undertilde{V}}

\def\utau{\undertilde{\tau}}
\def\ualpha{\undertilde{\alpha}}

\def\ux{\undertilde{x}}
\def\rot{{\rm rot}}
\def\curl{{\rm curl}}
\def\grad{{\rm grad}}
\def\dv{{\rm div}}

\def\omp{\ominus^\perp}
\def\opp{\oplus^\perp}

\def\od{\mathbf{d}}
\def\oT{\mathbf{T}}

\def\odelta{\boldsymbol{\delta}}
\def\okappa{\boldsymbol{\kappa}}

\def\icr{{\sf pic}}

\def\xA{\boldsymbol{\mathbf{A}}}
\def\xB{\boldsymbol{\mathbf{B}}}
\def\xD{\boldsymbol{\mathbf{D}}}
\def\xH{{\boldsymbol{\mathbf{H}}}}
\def\xX{{\boldsymbol{\mathbf{X}}}}
\def\xv{\boldsymbol{\mathbf{v}}}
\def\xw{\boldsymbol{\mathbf{w}}}
\def\yT{\boldsymbol{\mathbbm{T}}}
\def\yY{{\boldsymbol{\mathbbm{Y}}}}

\def\aoT{\yT}

\def\hf{\boldsymbol{\mathfrak{H}}}

\def\R{\mathcal{R}}
\def\N{\mathcal{N}}

\def\fomega{\boldsymbol{\omega}}
\def\fmu{\boldsymbol{\mu}}

\def\fzeta{\boldsymbol{\zeta}}
\def\feta{\boldsymbol{\eta}}

\def\fsigma{\boldsymbol{\sigma}}
\def\ftau{\boldsymbol{\tau}}

\def\fvartheta{\boldsymbol{\vartheta}}
\def\fvarsigma{\boldsymbol{\varsigma}}

\def\fW{\boldsymbol{W}}
\def\fV{\boldsymbol{V}}

\def\ff{\boldsymbol{f}}

\def\ixalpha{\boldsymbol{\alpha}} 
\def\ixve{\boldsymbol{\varepsilon}}

\def\fvarrho{\boldsymbol{\varrho}}
\def\fvarsigma{\boldsymbol{\varsigma}}

\def\pddeL{\mathbf{P}_{\od\cap\odelta}\Lambda}

\def\ppdd{\mathbf{P}_{\od\cap\odelta}\Lambda^k(\mathcal{G}_h)}

\def\rppdd{\mathring{\mathbf{P}}_{\odelta\times\od}^k(\mathcal{G}_h)}

\usetikzlibrary{arrows.meta, calc, positioning, shapes.geometric}
\usepackage{pgfplots}

\begin{document}

\title{Primal finite element scheme of the Hodge-Laplace problem}

\author{Wenyu Dong}
\address{Institute of Computational Mathematics and Scientific/Engineering Computing, Academy of Mathematics and System Sciences, Chinese Academy of Sciences, Beijing 100190, China}
\email{dongwenyu@lsec.cc.ac.cn}

\author{Shuo Zhang$^{\ast}$}
\thanks{*Corresponding author.}
\address{State Key Laboratory of Mathematical Sciences (SKLMS) and State Key Laboratory of Scientific and Engineering Computing (LSEC), Institute of
Computational Mathematics and Scientific/Engineering Computing, Academy of Mathematics and Systems Science, Chinese Academy of Sciences,
100190, Beijing, China}
\address{School of Mathematical Sciences, University of Chinese Academy of Sciences, 100049, Beijing, China}
\email{szhang@lsec.cc.ac.cn}

\thanks{The research is partially supported by NSFC (12271512, 11871465).}


%
%
%
%
%
%
%


\begin{abstract}
In this paper, we construct nonconforming finite element spaces $\fV^{\od\cap\mathring{\odelta}}_h\Lambda^k$ for the approximation of $H\Lambda^k\cap H^*_0\Lambda^k$ on simplicial meshes, for $n\geqslant 2$ and $1\leqslant k\leqslant n-1$, by enforcing adjoint continuity against piecewise Whitney spaces rather than trace matching. It holds, with $\od^k_h$ and $\odelta_{k,h}$ denoting respectively the piecewise action of differential and codifferential operators, and $\hf_h\Lambda^k$ being the discrete harmonic forms in the FEEC sense, that
$$
\hf_h\Lambda^k=\{\fmu_h\in \fV^{\od\cap\mathring{\odelta}}_h\Lambda^k:\od^k_h\fmu_h=0,\ \odelta_{k,h}\fmu_h=0\},
$$
which mirrors the continuous Hodge--Laplace kernel on domains with nontrivial topology. The space is not a classical Ciarlet-type finite element space; though, a uniform discrete Poincar\'e inequality and locally supported basis functions (supported on at most two cells) are guaranteed.  The resulting primal scheme yields an $\mathcal{O}(h)$ error bound for smooth data and $\mathcal{O}(h^s)$ on $s$-regular domains ($0<s\leqslant 1$), nontrivial topology admitted. Two- and three-dimensional eigenvalue tests agree with the mixed method on perforated domains, which are given to verify the validity of the scheme. 

\end{abstract}

\maketitle

\tableofcontents

\section{Introduction}
\label{sec:intro}


Let $\Omega\subset\mathbb{R}^n$ be a bounded Lipschitz domain. The Hodge-Laplace problem in its primal weak formulation reads: find $\fomega\in H\Lambda^k(\Omega)\cap H^*_0\Lambda^k(\Omega)$, orthogonal to the harmonic forms $\hf\Lambda^k$, such that
\begin{equation}\label{eq:modelhlori}
\langle\od^k\fomega,\od^k\fmu\rangle_{L^2\Lambda^{k+1}}+\langle\odelta_k\fomega,\odelta_k\fmu\rangle_{L^2\Lambda^{k-1}}=\langle\ff-\mathbf{P}_{\boldsymbol{\mathfrak{H}}}\ff,\fmu\rangle_{L^2\Lambda^k},
\end{equation}
for all $\fmu\in H\Lambda^k(\Omega)\cap H^*_0\Lambda^k(\Omega)$, where $\mathbf{P}_{\boldsymbol{\mathfrak{H}}}$ denotes the $L^2\Lambda^k$ projection onto $\hf\Lambda^k$. The model problem \eqref{eq:modelhlori} corresponds to the strong form that 
\begin{equation}
\od^k\fomega\in H^*_0\Lambda^{k+1}(\Omega),\ \ \ \odelta_k\fomega\in H\Lambda^{k-1}(\Omega),
\end{equation}
and
\begin{equation}
\fomega\perp\hf\Lambda^k(\Omega),\quad\mbox{and}\quad
\odelta_{k+1}\od^k\fomega+\od^{k-1}\odelta_k\fomega=\ff-\mathbf{P}_{\boldsymbol{\mathfrak{H}}}\ff. 
\end{equation}
The Hodge-Laplace problem arises in many applied sciences, including electromagnetics\cite{Monk.P2003mono,Hiptmair.R2002acta}, fluid-structure interaction \cite{Bathe.K;Nitikitpaiboon.C;Wang.X1995cs,Bermudez.A;Rodriguez.R1994cmame,Hamdi.M;Ousset.Y;Verchery.G1978ijnme}, and others. 

The standard remedy for the Hodge-Laplace problem in the FEEC literature is to adopt a mixed method. Particularly, the numerical solution of the Hodge-Laplace problem is a central subject of the theory of finite element exterior calculus (FEEC), and we refer to \cite{Arnold.D;Falk.R;Winther.R2006acta,Arnold.D;Falk.R;Winther.R2010bams,Arnold.D2018feec} for a thorough introduction to FEEC. This is due to the well known difficulty that a \emph{conforming} discretization of \eqref{eq:modelhlori} requires finite element subspaces of $H\Lambda^k\cap H^*_0\Lambda^k$ and a piecewise polynomial subspace of this intersection is necessarily contained in $H^1\Lambda^k\cap H^*_0\Lambda^k$, which is a nontrivial closed subspace. --- Consequently, when the exact solution possesses a singular component (as occurs on non-convex domains or domains with non-trivial topology), the conforming finite element solution may converge to a wrong limit. The mixed approach decomposes the second-order problem into first-order equations, discretizes each with conforming Whitney forms, and relies on the de~Rham complex structure to guarantee stability. A key ingredient is that spaces of discrete harmonic forms $\hf_h\Lambda^k$ are established isomorphic to the space of continuous harmonic forms $\hf\Lambda^k$, which is essential for problems on domains with non-trivial topology. 

The primal formulation involves only a single unknown field, and for problems coupled with other physics that also admit primal variational formulations, a primal discretization of the Hodge-Laplace operator can be integrated more naturally. Meanwhile, the mixed approach introduces \emph{nonlocal} discretization of the codifferentials~\cite{Lee.J;Winther.R2018local}. These considerations have motivated continuing efforts to discretize the primal formulation directly. So far, several methods have been proposed for the two-dimensional $H(\curl)\cap H(\dv)$ problem: interior penalty methods~\cite{Brenner.S;Sung.L;Cui.J2008}, quadratic nonconforming elements~\cite{brenner2009quadratic}, nonconforming elements with inter-element penalties~\cite{Brenner.S;Cui.J;Li.F;Sung.L2008nm}, and more recent constructions~\cite{Barker.M2022thesis,Barker.M;Cao.S;Stern.A2024SMAIJCM,Mirebeau.J2012aml}. Virtual element methods have also been designed for the three-dimensional vector potential formulation~\cite{Beiroa.L;Brezzi.F;Marini.D;Alessandro.R2018}, though the analysis there is restricted to domains without re-entrant corners, where harmonic forms vanish and the solution is smooth.

A crucial limitation by all existing primal discretizations is that they have not reconstructed the fundamental identity
\begin{equation}\label{eq:intro2:harmonic}
\hf\Lambda^k=\left\{\fmu\in H\Lambda^k\cap H^*_0\Lambda^k:\od^k\fmu=0,\ \odelta_k\fmu=0\right\}
\end{equation}
at the discrete level. Namely the joint kernel of the discrete exterior derivative and codifferential cannot be characterized as a space of discrete harmonic forms. Without this property, the method cannot correctly capture the topology of the domain.

The construction of nonconforming finite element spaces in this paper is inspired by a reinterpretation (\cite{Zhang.S2026IMA}) of the classical Crouzeix-Raviart (CR) nonconforming element (\cite{Crouzeix.M;Raviart.P1973}). Let $V^{\rm CR}_h$ be the space of piecewise linear polynomials with continuity at edge midpoints, and let $\uV{}^{\rm RT}_{h0}$ be the lowest-degree Raviart-Thomas space with vanishing normal traces on the boundary. On a triangulation $\mathcal{G}_h$, it is well known that these spaces satisfy the integration-by-parts identity
\begin{equation}\label{eq:intro2:crrt}
\sum_{T\in\mathcal{G}_h}\int_T\nabla v_h\cdot\utau{}_h\,\dx=- \sum_{T\in\mathcal{G}_h}\int_T v_h\,\dv\utau{}_h\,\dx,\qquad v_h\in V^{\rm CR}_h,\ \utau{}_h\in\uV{}^{\rm RT}_{h0}.
\end{equation}
The standard perspective treats \eqref{eq:intro2:crrt} as a {\bf consequence} of the definitions of $V^{\rm CR}_h$ and $V^{\rm RT}_{h0}$. A crucial observation, however, is that the identity also serves as a {\bf sufficient} condition, namely, \emph{a piecewise linear function $v_h$ belongs to $V^{\rm CR}_h$ {\bf if and only if} \eqref{eq:intro2:crrt} holds for every $\utau{}_h\in\uV{}^{\rm RT}_{h0}$} (\cite{Zhang.S2026IMA}). In other words, the nonconforming space $V^{\rm CR}_h$ is completely characterized by its adjoint relationship with a conforming space $\uV{}^{\rm RT}_{h0}$. ---- Simultaneously, $\uV{}^{\rm RT}_{h0}$ is also characterized by $V^{\rm CR}_h$. Namely, the pairing of $V^{\rm CR}_h$ and $\uV{}^{\rm RT}_{h0}$ inherits the adjoint relation between $(\nabla,H^1(\Omega))$ and $(\dv,H_0(\dv,\Omega))$. This perspective shifts the problem of constructing a nonconforming finite element space from \emph{imposing trace continuity conditions} (which may be difficult or impossible to formulate for the target function space) to \emph{enforcing an adjoint identity against an already-available conforming space}. Following this idea, a unified family of nonconforming finite element spaces $\fW^{\rm nc}_h\Lambda^k$ ($\fW^{*,\rm nc}_h\Lambda^k$ simultaneously) by piecewise Whitney forms has been established in \cite{Zhang.S2026IMA} for $H\Lambda^k$ ($H^*\Lambda^k$, respectively); the Poincar\'e-Hodge decomposition for $L^2\Lambda^k$ and the Poincar\'e-Lefschetz dualities as equalities were established at the discrete level for the first time therein. 

For the primal Hodge--Laplace problem, in this paper, we construct a nonconforming trial space for the intersection $H\Lambda^k\cap H^*_0\Lambda^k$. The space of piecewise polynomials is $\pddeL^k(\mathcal{G}_h)$, precisely given later, built from a local enrichment of the minimal trimmed space $\mathcal{P}_0\Lambda^k+\okappa(\mathcal{P}_0\Lambda^{k+1})+\star\okappa\star(\mathcal{P}_0\Lambda^{k-1})$; global membership is imposed through simultaneous adjoint continuity against the two spaces $\fW^{*,\rm nc}_{h0}\Lambda^{k+1}$ established in~\cite{Zhang.S2026IMA} for $H^*_0\Lambda^{k+1}$ and $\fW_h\Lambda^{k-1}$ the standard conforming finite element space by piecewise Whitney forms for $H\Lambda^{k-1}$:
\begin{multline}\label{eq:intro2:fems}
\fV^{\od\cap\mathring{\odelta}}_h\Lambda^k:=\Big\{\fmu_h\in\pddeL^k(\mathcal{G}_h):
\langle\od^k_h\fmu_h,\feta_h\rangle-\langle\fmu_h,\odelta_{k+1,h}\feta_h\rangle=0,\ \forall\,\feta_h\in\fW^{*,\rm nc}_{h0}\Lambda^{k+1},
\\
\mbox{and}\ \ \langle\odelta_{k,h}\fmu_h,\ftau_h\rangle-\langle\fmu_h,\od^{k-1}\ftau_h\rangle=0,\ \forall\,\ftau_h\in\fW_h\Lambda^{k-1}
\Big\}.
\end{multline}
It is this pairing of adjoint partners---one nonconforming for the codifferential side and one conforming Whitney for the differential side---that transfers the topology-sensitive discrete Hodge structure of~\cite{Zhang.S2026IMA} to the primal setting and makes nontrivial topology tractable. Particularly, it can be proved that
\begin{equation}\label{eq:intro2:harmonicdis}
\hf_h\Lambda^k=\left\{\fmu_h\in\fV^{\od\cap\mathring{\odelta}}_h\Lambda^k:\od^k_h\fmu_h=0,\ \odelta_{k,h}\fmu_h=0\right\},
\end{equation}
which mirrors the continuous identity \eqref{eq:intro2:harmonic}. Here $\hf_h\Lambda^k$ denotes the space of discrete harmonic forms in the standard FEEC theory, and $\od^k_h$ and $\odelta_{k,h}$ denote the piecewise action of the differential and codifferential operators, respectively. Crucial for correct approximation on domains with non-trivial topology, the identity \eqref{eq:intro2:harmonicdis} is not automatic for a primal discretization; it follows just from taking $\fW^{*,\rm nc}_{h0}\Lambda^{k+1}$ and $\fW_h\Lambda^{k-1}$ as the adjoint partners in~\eqref{eq:intro2:fems}, so that the trial space is aligned with the discrete cohomology already built into these spaces.

The conditions in \eqref{eq:intro2:fems} are precisely the adjoint relations that the function in $H\Lambda^k\cap H^*_0\Lambda^k$ satisfies when tested against functions in the corresponding partner spaces. In this sense, inter-cell continuity is imposed through adjoint relationships rather than through pointwise or tracewise matching on the mesh skeleton; we refer to this strategy as ``adjoint continuity.'' Moreover, at the continuous level, the primal formulation~\eqref{eq:modelhlori} and mixed Hodge--Laplace formulations are equivalent, and a legitimate primal discretization should preserve this link at the discrete level at least approximately, which is necessary for the convergence of the scheme. The construction above is designed so that the primal solution is closed to the solution of a classical mixed discretization based on conforming Whitney forms. This makes it possible---similarly to \cite{Arnold.D;Brezzi.F1985}---to establish the convergence of the primal finite element scheme by the aid of the classical mixed finite element scheme. In particular, Section~\ref{sec:fescheme} shows that the discrete primal and mixed solutions agree up to $\mathcal{O}(h\|\ff\|_{L^2})$ (Proposition~\ref{prop:primal-mixed-chain}). Theorem~\ref{thm:convprimsch}, proved in Section \ref{sec:fescheme}, records the resulting error bounds, including $\mathcal{O}(h)$ convergence in the $L^2\Lambda^k$ and broken $H\Lambda^k$, and $H^*\Lambda^k$ norms for smooth data and $\mathcal{O}(h^s)$ on $s$-regular domains ($0<s\leqslant 1$), independently of the topology; on the perforated domains emphasized below, the rate is regularity-limited by~$s$.


The space $\fV^{\od\cap\mathring{\odelta}}_h\Lambda^k$ does not correspond to a classical finite element space of Ciarlet's type~\cite{Ciarlet.P1978book}, because the adjoint continuity conditions \eqref{eq:intro2:fems} couple degrees of freedom across neighboring cells in a manner that cannot be reduced to matching nodal values on the mesh skeleton. Nevertheless, uniform discrete Poincar\'e inequalities can be proved for $\fV^{\od\cap\mathring{\odelta}}_h\Lambda^k$, therefore the well-posedness of the discretization follows. Implementably, we can construct a set of globally defined, locally supported basis functions; each is supported on a vertex patch spanning at most two cells, not necessarily adjacent. The existence of these basis functions is rigorously guaranteed, and once computed, they can be used in a standard cell-by-cell assembly routine. The construction is presented in a unified manner for all $n$ and $k$, and is illustrated concretely for the two- and three-dimensional $H(\curl)\cap H(\dv)$ cases in the numerical section.

The remaining of the paper is organized as follows. Section~\ref{sec:pre} collects preliminaries and notation. Section~\ref{sec:mainresults} presents the construction of the finite element spaces and schemes and states the main theorems of the paper. Section~\ref{sec:pfspace} proves the discrete Poincar\'e inequality and the existence of locally supported basis functions, and Section~\ref{sec:fescheme} establishes error estimates by the aid of several mixed FEEC schemes. Section~\ref{sec:numer} reports numerical experiments to verify the validity of the schemes by eigenvalue tests, especially on domains with nontrivial topology. Additional numerical and implementation material is collected in the appendices: manufactured-solution boundary-value tests (Appendix~\ref{app:bvp-results}), explicit basis functions (Appendices~\ref{sec:app:basis-2d} and~\ref{sec:app:basis-3d}), and eigenvalue counterexamples with vector Lagrange and vector Crouzeix--Raviart elements (Appendices~\ref{app:lagrange-eigenvalues} and~\ref{app:cr-eigenvalues}). Section~\ref{sec:conc} concludes with remarks and prospects.

\section{Preliminaries}

\label{sec:pre}

\subsection{Preliminary notations} Throughout the paper, we use $\N$ and $\R$ to denote the null space and the range of certain operators. For example, $\N(\oT,\xD)$ denotes $\left\{\xv\in\xD:\oT\xv=0\right\}$, and $\R(\oT,\xD)$ denotes $\left\{\oT\xv:\xv\in\xD\right\}.$ For a Hilbert space $\xH$, we use $\opp_\xH$ and $\omp_\xH$ to denote orthogonal summation and orthogonal difference; namely, for two spaces $\xA$ and $\xB$ in $\xH$, the notation $\xA\opp_\xH\xB$ means that $\xA$ and $\xB$ are orthogonal in $\xH$ and their sum is direct; for $\xA\subset\xB\subset \xH$, $\xB\omp_\xH\xA$ denotes the orthogonal complement of $\xA$ in $\xB$. The subscript $\xH$ may occasionally be dropped.

We use $\od^k$ and $\odelta_k$ for the exterior \emph{differential} and \emph{codifferential} operators on $k$-forms $\Lambda^k$; $\odelta_k=(-1)^{kn}\star\od^{n-k}\star$, with $\star$ the Hodge star operator. We use $\okappa$ for the Koszul operator defined by
\[
\okappa(\dx^{\ixalpha_1}\wedge\dots\wedge\dx^{\ixalpha_k}):=\sum_{j=1}^k(-1)^{j+1}x^{\ixalpha_j}\dx^{\ixalpha_1}\wedge\dots\wedge\dx^{\ixalpha_{j-1}}\wedge\dx^{\ixalpha_{j+1}}\wedge\dots\wedge\dx^{\ixalpha_k},\ \mbox{for}\  \ixalpha\in\mathbb{IX}_{k,n},
\]
where $\mathbb{IX}_{k,n}:=\left\{\ixalpha=(\ixalpha_1,\dots,\ixalpha_k)\in\mathbb{Z}^k:1\leqslant \ixalpha_1<\ixalpha_2<\dots<\ixalpha_k\leqslant n\right\}$ is the set of $k$-indices, $k\leqslant n$.  Then for $\ixalpha\in \mathbb{IX}_{k+1,n}$, $\od^k(\okappa(\dx^{\ixalpha_1}\wedge\dots\wedge\dx^{\ixalpha_{k+1}}))=(k+1)\dx^{\ixalpha_1}\wedge\dots\wedge\dx^{\ixalpha_{k+1}},$ and for $\ixalpha\in \mathbb{IX}_{k-1,n}$, $\odelta_k(\star\okappa\star(\dx^{\ixalpha_1}\wedge\dots\wedge \dx^{\ixalpha_{k-1}}))=(-1)^{kn-n-1} (n-k+1)(\dx^{\ixalpha_1}\wedge\dots\wedge\dx^{\ixalpha_{k-1}}).$ The homotopy formula $(\od^{k-1}\okappa+\okappa\od^k)\fmu=(k+r)\fmu$ holds for $\fmu\in\mathcal{H}_r\Lambda^k$, where $\mathcal{H}_r\Lambda^k$ denotes the space of $k$-forms with homogeneous polynomial coefficients of degree $r$. In the sequel, denote $\okappa^\delta:=\star\circ\okappa\circ\star$ for short.

Denote, on a domain $\Xi$,
$$
H\Lambda^k(\Xi):=\left\{\fomega\in L^2\Lambda^k(\Xi):\od^k\fomega\in L^2\Lambda^{k+1}(\Xi)\right\},\ \ \ 0\leqslant k\leqslant n-1,
$$
and by $H_0\Lambda^k(\Xi)$ the closure of $\mathcal{C}_0^\infty\Lambda^k(\Xi)$ in $H\Lambda^k(\Xi)$. Denote 
$$
H^*\Lambda^k(\Xi):=\left\{\fmu\in L^2\Lambda^k(\Xi):\odelta_k\fmu\in L^2\Lambda^{k-1}(\Xi)\right\},\ \ \ 1\leqslant k\leqslant n,
$$
and $H^*_0\Lambda^k(\Xi)$ the closure of $\mathcal{C}_0^\infty\Lambda^k(\Xi)$ in $H^*\Lambda^k(\Xi)$. $\Xi$ can occasionally be dropped. $H^*\Lambda^k=\star H\Lambda^{n-k}$ and $H^*_0\Lambda^k=\star H_0\Lambda^{n-k}$. The spaces of harmonic forms are $\hf\Lambda^k:=\N(\od^k,H\Lambda^k)\omp\R(\od^{k-1},H\Lambda^{k-1})$, $\hf_0\Lambda^k:=\N(\od^k,H_0\Lambda^k)\omp\R(\od^{k-1},H_0\Lambda^{k-1})$, $\hf^*\Lambda^k:=\N(\odelta_k,H^*\Lambda^k)\omp\R(\odelta_{k+1},H^*\Lambda^{k+1})$, and $\hf^*_0\Lambda^k:=\N(\odelta_k,H^*_0\Lambda^k)\omp\R(\odelta_{k+1},H^*_0\Lambda^{k+1})$. As the Helmholtz decompositions hold that
$$
\N(\od^k,H\Lambda^k)\opp \R(\odelta_{k+1},H^*_0\Lambda^{k+1})=L^2\Lambda^k=\R(\od^{k-1},H\Lambda^{k-1})\opp\N(\odelta_k,H^*_0\Lambda^k),
$$
it follows that $\hf\Lambda^k=\hf^*_0\Lambda^k$ and $\hf_0\Lambda^k=\hf^*\Lambda^k$. This is the Poincar\'e-Lefschetz duality (cf.\ \cite[Section 4.5.5]{Arnold.D2018feec}), which links the two dual complexes connected respectively by $\od^k$ and $\odelta_k$.

For $\Omega$ a domain and $\Xi$ a subdomain of $\Omega$, we denote by $E_\Xi^\Omega$ the extension-by-zero operator from $L^1_{\rm loc}(\Xi)$ to $L^1_{\rm loc}(\Omega)$. Namely,
\[
E_\Xi^\Omega: L^1_{\rm loc}(\Xi)\to L^1_{\rm loc}(\Omega),\qquad
E_\Xi^\Omega v=\begin{cases}
v,&\text{on }\Xi,\\[2pt]
0,&\text{else},
\end{cases}
\quad\text{for }v\in L^1_{\rm loc}(\Xi).
\]
For $V_T\subset L^1(T)$, we write $E_T^\Omega V_T$ to denote $\R(E_T^\Omega,V_T)$. We use the same notation $L^1_{\rm loc}$ for scalar and non-scalar locally integrable functions, and $E_\Xi^\Omega$ for the extension of both. 

Let $\mathcal{F}^\mathcal{G}$ be a set of shape regular simplicial subdivisions of $\Omega$. On a subdivision $\mathcal{G}_h\in \mathcal{F}^\mathcal{G}$, define formally the product of a set of function spaces $\left\{\Upsilon(T)\right\}_{T\in\mathcal{G}_h}$ defined cell by cell such that $\displaystyle E_T^\Omega\Upsilon(T)$ for all $T\in\mathcal{G}_h$ are compatible, $\displaystyle \prod_{T\in\mathcal{G}_h}\Upsilon(T):=\sum_{T\in\mathcal{G}_h} E_T^\Omega\Upsilon(T)$. The summation is direct.

\subsection{Finite element spaces for $H\Lambda^k$ by Whitney forms}
Following \cite{Arnold.D;Falk.R;Winther.R2006acta,Arnold.D;Falk.R;Winther.R2010bams,Arnold.D2018feec}, the space of Whitney forms is denoted by $\mathcal{P}^-_1\Lambda^k=\mathcal{P}_0\Lambda^k+\okappa(\mathcal{P}_0\Lambda^{k+1})$. The operator $\od^k$ maps $\okappa(\mathcal{P}_0\Lambda^{k+1})$ bijectively onto $\mathcal{P}_0\Lambda^{k+1}$. Note that $\mathcal{P}^-_1\Lambda^0=\mathcal{P}_1\Lambda^0$ and $\mathcal{P}^-_1\Lambda^n=\mathcal{P}_0\Lambda^n$. The Whitney forms associated with the operator $\odelta_k$ are denoted by $\mathcal{P}^{*,-}_1\Lambda^k:=\star(\mathcal{P}^-_1\Lambda^{n-k})=\mathcal{P}_0\Lambda^k+\okappa^\delta(\mathcal{P}_0\Lambda^{k-1})$. The operator $\odelta_k$ maps $\okappa^\delta(\mathcal{P}_0\Lambda^{k-1})$ bijectively onto $\mathcal{P}_0\Lambda^{k-1}$. Note that
\begin{equation}\label{eq:n=r=c}
\N(\od^k,\mathcal{P}^-_1\Lambda^k)=\R(\od^{k-1},\mathcal{P}^-_1\Lambda^{k-1})=\mathcal{P}_0\Lambda^k=\R(\odelta_{k+1},\mathcal{P}^{*,-}_1\Lambda^{k+1})=\N(\odelta_k,\mathcal{P}^{*,-}_1\Lambda^k).
\end{equation}

Denote, on a simplicial subdivision $\mathcal{G}_h$ of $\Omega$, for $0\leqslant k\leqslant n$,
\begin{equation}
\mathcal{P}^-_1\Lambda^k(\mathcal{G}_h):=\prod_{T\in\mathcal{G}_h} \mathcal{P}^-_1\Lambda^k(T), \qquad
\mathcal{P}^{*,-}_1\Lambda^k(\mathcal{G}_h):=\prod_{T\in\mathcal{G}_h} \mathcal{P}^{*,-}_1\Lambda^k(T).
\end{equation}
Here and in the sequel, the subscript $\cdot_h$ denotes mesh dependence. In particular, an operator bearing the subscript $\cdot_h$ is understood to act cell by cell. Particularly, throughout, $\od^k_h$ and $\odelta_{k,h}$ denote the cell-by-cell exterior differential and codifferential.

The conforming finite element spaces with Whitney forms are denoted by
\begin{align*}
\fW_h\Lambda^k&:=\mathcal{P}^-_1\Lambda^k(\mathcal{G}_h)\cap H\Lambda^k, &
\fW_{h0}\Lambda^k&:=\mathcal{P}^-_1\Lambda^k(\mathcal{G}_h)\cap H_0\Lambda^k,\\
\fW^*_h\Lambda^k&:=\mathcal{P}^{*,-}_1\Lambda^k(\mathcal{G}_h)\cap H^*\Lambda^k, &
\fW^*_{h0}\Lambda^k&:=\mathcal{P}^{*,-}_1\Lambda^k(\mathcal{G}_h)\cap H^*_0\Lambda^k.
\end{align*}
Then $\fW^*_h\Lambda^k=\star \fW_h\Lambda^{n-k}$ and $\fW^*_{h0}\Lambda^k=\star \fW_{h0}\Lambda^{n-k}$. These spaces coincide with the finite element spaces of piecewise Whitney forms defined by continuity of the nodal parameters~\cite{Arnold.D2018feec}. Denote the spaces of discrete harmonic forms by, orthogonally in $L^2\Lambda^k$,
\begin{align*}
\hf_h\Lambda^k&:=\N(\od^k,\fW_h\Lambda^k)\omp \R(\od^{k-1},\fW_h\Lambda^{k-1}),\\
\hf_{h0}\Lambda^k&:=\N(\od^k,\fW_{h0}\Lambda^k)\omp \R(\od^{k-1},\fW_{h0}\Lambda^{k-1}),\\
\hf^*_h\Lambda^k&:=\N(\odelta_k,\fW^*_h\Lambda^k)\omp \R(\odelta_{k+1},\fW^*_h\Lambda^{k+1}),\\
\hf^*_{h0}\Lambda^k&:=\N(\odelta_k,\fW^*_{h0}\Lambda^k)\omp \R(\odelta_{k+1},\fW^*_{h0}\Lambda^{k+1}).
\end{align*}
Then $\hf_h\Lambda^k=\star\hf^*_h\Lambda^{n-k}$ and $\hf_{h0}\Lambda^k=\star\hf^*_{h0}\Lambda^{n-k}$.

\begin{lemma}[\cite{Arnold.D2018feec}]
$\hf_h\Lambda^k$ and $\hf_{h0}\Lambda^k$ are isomorphic to $\hf\Lambda^k$ and $\hf_0\Lambda^k$, respectively.
\end{lemma}

Following \cite{Zhang.S2026IMA}, we denote the nonconforming finite element spaces for $H\Lambda^k$, $H_0\Lambda^k$, $H^*\Lambda^k$, and $H^*_0\Lambda^k$, respectively, with piecewise Whitney forms by
\begin{align*}
\fW^{\rm nc}_h\Lambda^k&:=\Big\{\fmu_h\in \mathcal{P}^-_1\Lambda^k(\mathcal{G}_h):\sum_{T\in\mathcal{G}_h}\langle\od^k\fmu_h,\feta_h\rangle_{L^2\Lambda^{k+1}(T)}=\sum_{T\in\mathcal{G}_h}\langle\fmu_h,\odelta_{k+1}\feta_h\rangle_{L^2\Lambda^k(T)},\ \forall\,\feta_h\in \fW^*_{h0}\Lambda^{k+1}\Big\},\\
\fW^{\rm nc}_{h0}\Lambda^k&:=\Big\{\fmu_h\in \mathcal{P}^-_1\Lambda^k(\mathcal{G}_h):\sum_{T\in\mathcal{G}_h}\langle\od^k\fmu_h,\feta_h\rangle_{L^2\Lambda^{k+1}(T)}=\sum_{T\in\mathcal{G}_h}\langle\fmu_h,\odelta_{k+1}\feta_h\rangle_{L^2\Lambda^k(T)},\ \forall\,\feta_h\in \fW^*_h\Lambda^{k+1}\Big\},\\
\fW^{*,\rm nc}_h\Lambda^k&:=\Big\{\fmu_h\in \mathcal{P}^{*,-}_1\Lambda^k(\mathcal{G}_h):\sum_{T\in\mathcal{G}_h}\langle\odelta_k\fmu_h,\ftau_h\rangle_{L^2\Lambda^{k-1}(T)}=\sum_{T\in\mathcal{G}_h}\langle\fmu_h,\od^{k-1}\ftau_h\rangle_{L^2\Lambda^k(T)},\ \forall\,\ftau_h\in \fW_{h0}\Lambda^{k-1}\Big\},\\
\fW^{*,\rm nc}_{h0}\Lambda^k&:=\Big\{\fmu_h\in \mathcal{P}^{*,-}_1\Lambda^k(\mathcal{G}_h):\sum_{T\in\mathcal{G}_h}\langle\odelta_k\fmu_h,\ftau_h\rangle_{L^2\Lambda^{k-1}(T)}=\sum_{T\in\mathcal{G}_h}\langle\fmu_h,\od^{k-1}\ftau_h\rangle_{L^2\Lambda^k(T)},\ \forall\,\ftau_h\in \fW_h\Lambda^{k-1}\Big\}.
\end{align*}
Note that $\fW^{\rm nc}_{h(0)}\Lambda^0$ and $\fW^{*,\rm nc}_{h(0)}\Lambda^n$ coincide with the classical lowest-degree Crouzeix-Raviart element spaces~\cite{Crouzeix.M;Raviart.P1973}. Moreover, the conforming spaces admit a dual characterization; i.e., for example,
\[
\fW_h\Lambda^k=\left\{\fmu_h\in\mathcal{P}^-_1\Lambda^k(\mathcal{G}_h):\sum_{T\in\mathcal{G}_h}\langle\od^k\fmu_h,\feta_h\rangle_{L^2\Lambda^{k+1}(T)}=\sum_{T\in\mathcal{G}_h}\langle\fmu_h,\odelta_{k+1}\feta_h\rangle_{L^2\Lambda^k(T)},\ \forall\,\feta_h\in \fW^{*,\rm nc}_{h0}\Lambda^{k+1}\right\}.
\]
Hence, the adjoint relation between $H\Lambda^k$ and $H^*_0\Lambda^{k+1}$ is this way discretely inherited.
\begin{theorem}[{\cite[Theorem~18]{Zhang.S2026IMA}}]
The spaces $\fW^{\rm nc}_h\Lambda^k$ and $\fW^{\rm nc}_{h0}\Lambda^k$ each admit a set of linearly independent basis functions that are each supported on at most two simplices.
\end{theorem}

\begin{theorem}[Discrete Helmholtz decomposition\cite{Zhang.S2026IMA}]\label{thm:dishd}
The following decompositions are orthogonal in $L^2\Lambda^k(\Omega)$:
\begin{align*}
\mathcal{P}_0\Lambda^k(\mathcal{G}_h)&=\R(\od^{k-1}_h,\fW^{\rm nc}_h\Lambda^{k-1}) \opp \N(\odelta_{k,h},\fW_{h0}^*\Lambda^k)=\R(\od^{k-1}_h,\fW^{\rm nc}_{h0}\Lambda^{k-1}) \opp \N(\odelta_{k,h},\fW_h^*\Lambda^k),\\
\mathcal{P}_0\Lambda^k(\mathcal{G}_h)&=\N(\od^k_h,\fW^{\rm nc}_h\Lambda^k) \opp \R(\odelta_{k+1,h},\fW_{h0}^*\Lambda^{k+1})=\N(\od^k_h,\fW^{\rm nc}_{h0}\Lambda^k) \opp \R(\odelta_{k+1,h},\fW_h^*\Lambda^{k+1}),
\end{align*}
where the first line holds for $1\leqslant k\leqslant n$, and the second line holds for $0\leqslant k\leqslant n-1$.
\end{theorem}

Denote
\begin{align*}
\hf^{\rm nc}_h\Lambda^k&:=\N(\od^k_h,\fW^{\rm nc}_h\Lambda^k)\omp \R(\od^{k-1}_h,\fW^{\rm nc}_h\Lambda^{k-1}),\\
\hf^{\rm nc}_{h0}\Lambda^k&:=\N(\od^k_h,\fW^{\rm nc}_{h0}\Lambda^k)\omp \R(\od^{k-1}_h,\fW^{\rm nc}_{h0}\Lambda^{k-1}),\\
\hf^{*,\rm nc}_h\Lambda^k&:=\N(\odelta_{k,h},\fW^{*,\rm nc}_h\Lambda^k)\omp \R(\odelta_{k+1,h},\fW^{*,\rm nc}_h\Lambda^{k+1}),\\
\hf^{*,\rm nc}_{h0}\Lambda^k&:=\N(\odelta_{k,h},\fW^{*,\rm nc}_{h0}\Lambda^k)\omp \R(\odelta_{k+1,h},\fW^{*,\rm nc}_{h0}\Lambda^{k+1}).
\end{align*}

Then
$\hf^{\rm nc}_h\Lambda^k=\star\hf^{*,\rm nc}_h\Lambda^{n-k}$, and $\hf^{\rm nc}_{h0}\Lambda^k=\star\hf_{h0}^{*,\rm nc}\Lambda^{n-k}.$ Further, the discrete Poincar\'e-Lefschetz dualities and discrete Hodge decompositions hold.
\begin{lemma}\cite{Zhang.S2026IMA}$\hf_h\Lambda^k=\hf^{*,\rm nc}_{h0}\Lambda^k$, $\hf_{h0}\Lambda^k = \hf^{*, \rm nc}_h\Lambda^k$, $\hf^*_h\Lambda^k=\hf^{\rm nc}_{h0}\Lambda^k$ and $\hf^*_{h0}\Lambda^k = \hf^{\rm nc}_h\Lambda^k$. 
\end{lemma}

\begin{lemma}\cite{Zhang.S2026IMA}\label{lem:hoddecconst}
The Hodge decompositions hold: with $1\leqslant k\leqslant n-1$,
\begin{multline*}
\qquad\mathcal{P}_0\Lambda^k(\mathcal{G}_h)=\R(\od^{k-1},\fW_h\Lambda^{k-1})\opp\hf_h\Lambda^k \opp \R(\odelta_{k+1,h},\fW^{*,\rm nc}_{h0}\Lambda^{k+1})
\\
=\R(\od^{k-1},\fW_{h0}\Lambda^{k-1})\opp\hf_{h0}\Lambda^k \opp \R(\odelta_{k+1,h},\fW^{*,\rm nc}_h\Lambda^{k+1}).\qquad
\end{multline*}
\end{lemma}

\subsection{A pair of adjoint operators associated with the Hodge-Laplacian}

Define
\[
\oT_k:\Lambda^k\to \Lambda^{k+1}\times \Lambda^{k-1},\qquad \mbox{by}\ \oT_k\fmu:=\begin{pmatrix}\od^k\fmu \\ \odelta_k\fmu\end{pmatrix},
\]
and
\[
\aoT_k:\Lambda^{k+1}\times \Lambda^{k-1}\to \Lambda^k,\qquad \mbox{by}\  \aoT_k\begin{pmatrix}\feta \\ \ftau\end{pmatrix}:=\odelta_{k+1}\feta+\od^{k-1}\ftau.
\]
Formally, we have $\odelta_{k+1}\od^k+\od^{k-1}\odelta_k=\aoT_k\circ\oT_k$.

\begin{lemma}
The operator $(\oT_k,H\Lambda^k\cap H^*_0\Lambda^k)$ is closed, and its range $\R(\oT_k,H\Lambda^k\cap H^*_0\Lambda^k)$ is closed.
\end{lemma}

\begin{proof}
We analyze the structure of $H\Lambda^k\cap H^*_0\Lambda^k$. First, decompose
\[
H\Lambda^k=\N(\od^k,H\Lambda^k)\opp (\N(\od^k,H\Lambda^k))^\perp
=\R(\od^{k-1},H\Lambda^{k-1})\opp\hf\Lambda^k\opp (\N(\od^k,H\Lambda^k))^\perp.
\]
By the Helmholtz decomposition, $(\N(\od^k,H\Lambda^k))^\perp\subset H\Lambda^k\cap \N(\odelta_k,H^*_0\Lambda^k)$. Hence
\[
\R(\od^k,(\N(\od^k,H\Lambda^k))^\perp)=\R(\od^k,H\Lambda^k),\ \mbox{and}\  \R(\odelta_k,(\N(\od^k,H\Lambda^k))^\perp)=\{0\}.
\]
Similarly, decompose $H^*_0\Lambda^k=\N(\odelta_k,H^*_0\Lambda^k)\opp (\N(\odelta_k,H^*_0\Lambda^k))^\perp$, giving
\[
\R(\od^k,(\N(\odelta_k,H^*_0\Lambda^k))^\perp)=\{0\},\ \mbox{and}\  \R(\odelta_k,(\N(\odelta_k,H^*_0\Lambda^k))^\perp)=\R(\odelta_k,H^*_0\Lambda^k).
\]
It follows that
\[
\R(\oT_k,H\Lambda^k\cap H^*_0\Lambda^k)=\R(\od^k,H\Lambda^k)\times\R(\odelta_k,H^*_0\Lambda^k).
\]
Simultaneously,
\begin{multline*}
\N(\oT_k,H\Lambda^k\cap H^*_0\Lambda^k)=\N(\od^k,H\Lambda^k)\cap \N(\odelta_k,H^*_0\Lambda^k)
\\
=[\R(\od^{k-1},H\Lambda^{k-1})\opp \hf\Lambda^k]\cap [\R(\odelta_{k+1},H^*_0\Lambda^{k+1})\opp\hf^*_0\Lambda^k]=\hf\Lambda^k.
\end{multline*}
Both $\od^k$ and $\odelta_k$ are closed operators on their respective domains; consequently $\oT_k$ is closed on $H\Lambda^k\cap H^*_0\Lambda^k$ equipped with the intersection graph norm, and the explicit description of its range shows that the range is closed as the product of two closed ranges.
\end{proof}

\begin{lemma}
The adjoint operator of $(\oT_k,H\Lambda^k\cap H^*_0\Lambda^k)$ is $(\aoT_k,H^*_0\Lambda^{k+1}\times H\Lambda^{k-1})$.
\end{lemma}

\begin{proof}
For any $\fmu\in H\Lambda^k\cap H^*_0\Lambda^k$, $\feta\in H^*_0\Lambda^{k+1}$, and $\ftau\in H\Lambda^{k-1}$, integration by parts gives
\begin{equation}\label{eq:adjointidentity}
\langle\od^k\fmu,\feta\rangle_{L^2\Lambda^{k+1}}+\langle\odelta_k\fmu,\ftau\rangle_{L^2\Lambda^{k-1}}=\langle\fmu,\odelta_{k+1}\feta+\od^{k-1}\ftau\rangle_{L^2\Lambda^k}.
\end{equation}
Moreover,
\[
\R(\aoT_k, H^*_0\Lambda^{k+1}\times H\Lambda^{k-1})=\R(\odelta_{k+1},H^*_0\Lambda^{k+1})\opp \R(\od^{k-1},H\Lambda^{k-1}),
\]
where the orthogonality follows from $\langle\odelta_{k+1}\feta,\od^{k-1}\ftau\rangle_{L^2\Lambda^k}=\langle\feta,\od^k\od^{k-1}\ftau\rangle_{L^2\Lambda^{k+1}}=0$, and
\[
\N(\aoT_k, H^*_0\Lambda^{k+1}\times H\Lambda^{k-1})=\N(\odelta_{k+1},H^*_0\Lambda^{k+1})\times\N(\od^{k-1},H\Lambda^{k-1}).
\]
Consequently,
\begin{align*}
L^2\Lambda^k&=\N(\oT_k,H\Lambda^k\cap H^*_0\Lambda^k)\opp \R(\aoT_k, H^*_0\Lambda^{k+1}\times H\Lambda^{k-1}),\\
L^2\Lambda^{k+1}\times L^2\Lambda^{k-1}&=\R(\oT_k,H\Lambda^k\cap H^*_0\Lambda^k)\opp \N(\aoT_k, H^*_0\Lambda^{k+1}\times H\Lambda^{k-1}).
\end{align*}
Thus $(\aoT_k,H^*_0\Lambda^{k+1}\times H\Lambda^{k-1})$ shares the same range and kernel as the adjoint of $(\oT_k,H\Lambda^k\cap H^*_0\Lambda^k)$. Together with \eqref{eq:adjointidentity}, this establishes that it is indeed the adjoint operator.
\end{proof}

\begin{remark}
In the sequel we write $\oT^*_k$ for $\aoT_k$; that is,
\[
\oT^*_k\begin{pmatrix}\feta \\ \ftau\end{pmatrix}=\odelta_{k+1}\feta+\od^{k-1}\ftau,\qquad  \mbox{for}\  \begin{pmatrix}\feta \\ \ftau\end{pmatrix}\in \Lambda^{k+1}\times \Lambda^{k-1}.
\]
In the discrete setting, $\oT^*_{k,h}$ denotes the piecewise application of $\oT^*_k$.
\end{remark}

%



%
%
%
\section{The main results of this paper}
\label{sec:mainresults}

In this section, we present the main results of this paper; the technical proofs are postponed to Sections~\ref{sec:pfspace} and~\ref{sec:fescheme}. Basic polynomial shape function spaces are introduced in Subsection~\ref{sec:polys}, followed by the global finite element spaces and primal discretization schemes.

\subsection{Shape function spaces on a single simplex}\label{sec:polys}

Given a simplex $T$, let $\tilde{x}^j=x^j-c_j$ where the constants $c_j$ are chosen so that $\int_T\tilde{x}^j=0$. 
Replacing the coordinate functions in the definition of the standard Koszul operator by their centered counterparts yields the simplex-dependent Koszul operator $\okappa_T$ defined by
\[
\okappa_T(\dx^{\ixalpha_1}\wedge\dots\wedge\dx^{\ixalpha_k}):=\sum_{j=1}^k(-1)^{j+1}\tilde{x}^{\ixalpha_j}\dx^{\ixalpha_1}\wedge\dots\wedge\dx^{\ixalpha_{j-1}}\wedge\dx^{\ixalpha_{j+1}}\wedge\dots\wedge\dx^{\ixalpha_k},\ \mbox{for}\  \ixalpha\in\mathbb{IX}_{k,n}.
\]
Then $\od^{k-1}\okappa_T(\dx^{\ixalpha_1}\wedge\dots\wedge\dx^{\ixalpha_k})=k\,\dx^{\ixalpha_1}\wedge\dots\wedge\dx^{\ixalpha_k}$. With the aid of $\okappa_T$, the Whitney forms admit the orthogonal decomposition $\mathcal{P}^-_1\Lambda^k(T)=\mathcal{P}_0\Lambda^k(T)\opp\okappa_T(\mathcal{P}_0\Lambda^{k+1}(T))$ in $L^2\Lambda^k(T)$. We use $\okappa_h$ to denote the cell-wise application of $\okappa_T$. Define $\okappa^\delta_T:=\star\circ\okappa_T\circ\star$ and $\okappa^\delta_h:=\star\circ\okappa_h\circ\star$.

From now on, we make the convention that, for $\ixalpha\in\mathbb{IX}_{k,n}$, we use $\ixalpha^\mathsf{C}$ for one in $\mathbb{IX}_{n-k,n}$, such that $\ixalpha$ and $\ixalpha^\mathsf{C}$ partition $\{1,2,\dots,n\}$. For $\ixalpha\in\mathbb{IX}_{k,n}$, denote 
$$
\tilde\fmu_{\odelta,T}^{\ixalpha}=\sum_{j=1}^k[(\tilde{x}^{\ixalpha_j})^2-c^{\ixalpha_j}]\dx^{\ixalpha_1}\wedge\dots\wedge\dx^{\ixalpha_k}, 
$$
and
$$
\tilde\fmu_{\od,T}^{\ixalpha}=\sum_{j=1}^{n-k}[(\tilde{x}^{\ixalpha^\mathsf{C}_j})^2-c^{\ixalpha^\mathsf{C}_j}]\dx^{\ixalpha_1}\wedge\dx^{\ixalpha_2}\wedge\dots\wedge\dx^{\ixalpha_k}, 
$$
where $c^{\ixalpha_j}$ and $c^{\ixalpha^\mathsf{C}_j}$ are constants such that $\int_T [(\tilde x^{\ixalpha_j})^2-c^{\ixalpha_j}]=0$ and $\int_T[(\tilde x^{\ixalpha^\mathsf{C}_j})^2-c^{\ixalpha^\mathsf{C}_j}]=0$.  Denote 
\begin{equation}
\tilde\fmu_{\od\cap\odelta,T}^{\ixalpha}= \frac{-k}{n-k}\tilde\fmu_{\od,T}^{\ixalpha} + \tilde\fmu_{\odelta,T}^{\ixalpha}.
\end{equation}

\begin{remark}
Direct calculation confirms that, with $\ixalpha\in\mathbb{IX}_{k,n}$,
$$
\od^k\tilde\fmu_{\odelta,T}^{\ixalpha}=0,\quad \odelta_k\tilde\fmu_{\odelta,T}^{\ixalpha}=(-1)^n\cdot 2\okappa_T(\dx^{\ixalpha_1}\wedge\dx^{\ixalpha_2}\wedge\dots\wedge \dx^{\ixalpha_k}),
$$
$$
\odelta_k\tilde\fmu_{\od,T}^{\ixalpha}=0,\quad \quad \od^k\tilde\fmu_{\od,T}^{\ixalpha}=2(-1)^{n(1+k)+1}\okappa_T^\delta (\dx^{\ixalpha_1}\wedge\dots\dx^{\ixalpha_k}),
$$
and 
$$
\oT_k\tilde\fmu_{\od\cap\odelta,T}^{\ixalpha}=2(-1)^n\left(\begin{array}{c} \displaystyle(-1)^{kn}\frac{k}{n-k}\okappa^\delta_T(\dx^{\ixalpha_1}\wedge\dots\wedge\dx^{\ixalpha_k})\\ \displaystyle\okappa_T(\dx^{\ixalpha_1}\wedge\dots\wedge\dx^{\ixalpha_k})\end{array}\right).
$$
\end{remark}

Denote 
\begin{equation}\label{eq:defhd}
\mathcal{H}^2_{\od}\Lambda^k(T):={\rm span}\{\tilde\fmu_{\od,T}^{\ixalpha}:\ixalpha\in\mathbb{IX}_{k,n}\},
\end{equation}
\begin{equation}\label{eq:defhdelta}
\mathcal{H}^2_{\odelta}\Lambda^k(T):={\rm span}\{\tilde\fmu_{\odelta,T}^{\ixalpha}:\ixalpha\in\mathbb{IX}_{k,n}\},
\end{equation}
and
\begin{equation}\label{eq:defhddelta}
\mathcal{H}^2_{\od\cap\odelta}\Lambda^k(T):={\rm span}\left\{\tilde\fmu_{\od\cap\odelta,T}^{\ixalpha}:\ixalpha\in\mathbb{IX}_{k,n}\right\},
\end{equation}

\begin{remark}
$\odelta_k$ is bijective from $\mathcal{H}^2_{\odelta}\Lambda^k(T)$ onto $\okappa_T(\mathcal{P}_0\Lambda^k(T))$, $\od^k$ is bijective from $\mathcal{H}^2_{\od}\Lambda^k(T)$ onto $\okappa^\delta_T(\mathcal{P}_0\Lambda^k(T))$, and both $\od^k$ and $\odelta_k$ are bijective from $\mathcal{H}^2_{\od\cap\odelta}\Lambda^k(T)$ onto $\okappa^\delta_T(\mathcal{P}_0\Lambda^k(T))$ and $\okappa_T(\mathcal{P}_0\Lambda^k(T))$, respectively. 
\end{remark}

Now we introduce, for the $H\Lambda^k\cap H^*\Lambda^k$ space, an enriched trimmed space, defined by
\begin{equation}
\pddeL^k(T):=\mathcal{P}_0\Lambda^k(T)\oplus\okappa_T(\mathcal{P}_0\Lambda^{k+1}(T))\oplus\okappa^\delta_T(\mathcal{P}_0\Lambda^{k-1}(T))\oplus{\rm span}\left\{\tilde\fmu_{\od\cap\odelta,T}^{\ixalpha}:\ixalpha\in\mathbb{IX}_{k,n}\right\},
\end{equation}
and denote
\begin{equation}
\mathbf{P}_{\odelta\times\od}^k(T):=\mathcal{P}^{*,-}_1\Lambda^{k+1}(T)\times \mathcal{P}^-_1\Lambda^{k-1}(T).
\end{equation}

\begin{remark}\label{rem:kernel=range}
The following identities hold:
\begin{equation}
\N(\oT_k,\pddeL^k(T))=\R(\oT^*_k,\mathbf{P}_{\odelta\times\od}^k(T)),\ \ \mbox{and}\ \ \R(\oT_k,\pddeL^k(T))=\N(\oT^*_k,\mathbf{P}_{\odelta\times\od}^k(T)).
\end{equation} 

Actually, for the first item, 
$$
\N(\oT_k,\pddeL^k(T))=\mathcal{P}_0\Lambda^k(T)=\R(\oT^*_k,\mathbf{P}_{\odelta\times\od}^k(T)),
$$
and for the second 
\begin{multline*}
\N(\oT^*_k,\mathbf{P}_{\odelta\times\od}^k(T))=\mathcal{P}_0\Lambda^{k+1}(T)\times\mathcal{P}_0\Lambda^{k-1}(T)
\\
\oplus{\rm span}\left\{\left(\begin{array}{c}\displaystyle (-1)^{kn}\frac{k}{n-k}\okappa^\delta_T(\dx^{\ixalpha_1}\wedge\dots\wedge\dx^{\ixalpha_k})\\ \displaystyle\okappa_T(\dx^{\ixalpha_1}\wedge\dots\wedge\dx^{\ixalpha_k})\end{array}\right),\ \ixalpha\in\mathbb{IX}_{k,n}\right\}=\R(\oT_k,\pddeL^k(T)).
\end{multline*}
\end{remark}

\begin{remark}\label{rem:cell1-1}
Direct calculation leads to that, given $\fmu\in \pddeL^k(T)$, $\fmu=0$ if and only if, for any $\left(\begin{array}{c}\feta \\ \ftau\end{array}\right)\in \mathbf{P}_{\odelta\times\od}^k(T)$, 
\begin{equation}\label{eq:dual-dcde-detd}
\langle\od^k\fmu,\feta\rangle_{L^2\Lambda^{k+1}(T)}+\langle\odelta_k\fmu,\ftau\rangle_{L^2\Lambda^{k-1}(T)}=\langle\fmu,(\odelta_{k+1}\feta+\od^{k-1}\ftau)\rangle_{L^2\Lambda^k(T)}.
\end{equation}
Meanwhile, given $\left(\begin{array}{c}\feta \\ \ftau\end{array}\right)\in \mathbf{P}_{\odelta\times\od}^k(T)$, $\left(\begin{array}{c}\feta \\ \ftau\end{array}\right)=0$ if and only if \eqref{eq:dual-dcde-detd} holds for any $\fmu\in \pddeL^k(T)$.
\end{remark}

%


%
%
\subsection{Finite element space for $H\Lambda^k\cap H^*_0\Lambda^k$}
\label{sec:fes}

On the subdivision $\mathcal{G}_h$ of $\Omega$, define 
\begin{multline}
\fV^{\od\cap\mathring{\odelta}}_h\Lambda^k:=\left\{\fmu_h\in\pddeL^k(\mathcal{G}_h),\ \mbox{such\ that}\ \left\langle\oT_{k,h}\fmu_h,\left(\begin{array}{c}\feta_h \\ \ftau_h\end{array}\right)\right\rangle=\left\langle\fmu_h,\oT^*_{k,h}\left(\begin{array}{c}\feta_h \\ \ftau_h\end{array}\right)\right\rangle_{L^2\Lambda^k},\right.
\\ 
\left.\forall\,\left(\begin{array}{c}\feta_h \\ \ftau_h\end{array}\right)\in \fW^{*,\rm nc}_{h0}\Lambda^{k+1}\times\fW_h\Lambda^{k-1}\right\}.\qquad
\end{multline}
Here, we use $\left\langle \left(\begin{array}{c}\fmu \\ \fvarsigma\end{array}\right),\left(\begin{array}{c}\feta \\ \ftau\end{array}\right)\right\rangle=\langle \fmu,\feta\rangle_{L^2\Lambda^{k+1}}+\langle \fvarsigma,\ftau\rangle_{L^2\Lambda^{k-1}}$for $\left(\begin{array}{c}\fmu \\ \fvarsigma\end{array}\right),\left(\begin{array}{c}\feta \\ \ftau\end{array}\right)\in \Lambda^{k+1}\times \Lambda^{k-1}$  to denote the $L^2$ inner product. 
Namely, 
\begin{multline}\label{eq:deffems}
\fV^{\od\cap\mathring{\odelta}}_h\Lambda^k=\Big\{\fmu_h\in\pddeL^k(\mathcal{G}_h),\ \mbox{such\ that}\ 
\left\langle\od^k_h\fmu_h,\feta_h\right\rangle_{L^2\Lambda^{k+1}}=\left\langle\fmu_h,\odelta_{k+1,h}\feta_h\right\rangle_{L^2\Lambda^k},
\\
\forall\,\feta_h\in\fW^{*,\rm nc}_{h0}\Lambda^{k+1},\ \mbox{and}\ \ \left\langle\odelta_{k,h}\fmu_h,\ftau_h\right\rangle_{L^2\Lambda^{k-1}}=\left\langle\fmu_h,\od^{k-1}\ftau_h\right\rangle_{L^2\Lambda^k},\ \forall\,\ftau_h\in\fW_h\Lambda^{k-1}
\Big\}.
\end{multline}

Therefore, on the other hand, according to Remark \ref{rem:cell1-1}, 
$$
\fW^{*,\rm nc}_{h0}\Lambda^{k+1}\times\fW_h\Lambda^{k-1}=\left\{\left(\begin{array}{c}\feta_h \\ \ftau_h\end{array}\right)\in \mathbf{P}_{\odelta\times\od}^k(\mathcal{G}_h):
\left\langle\oT_{k,h}\fmu_h,\left(\begin{array}{c}\feta_h \\ \ftau_h\end{array}\right)\right\rangle=\left\langle\fmu_h,\oT^*_{k,h}\left(\begin{array}{c}\feta_h \\ \ftau_h\end{array}\right)\right\rangle_{L^2\Lambda^k},\ \fmu_h\in\fV^{\od\cap\mathring{\odelta}}_h\Lambda^k\right\}.
$$

The space $\fV^{\od\cap\mathring{\odelta}}_h\Lambda^k$ possesses these characteristics. 
%
\begin{theorem}\label{thm:disharm}
$\N(\oT_{k,h},\fV^{\od\cap\mathring{\odelta}}_h\Lambda^k)=\hf_h\Lambda^k$.
\end{theorem}
\begin{proof}
Firstly, by Lemma \ref{lem:hoddecconst} and by \eqref{eq:deffems},
\begin{multline*}
\hf_h\Lambda^k=\Big\{\fmu_h\in\mathcal{P}_0\Lambda^k(\mathcal{G}_h): \langle\fmu_h,\odelta_{k+1,h}\feta_h\rangle_{L^2\Lambda^k}=0,\ \forall\,\feta_h\in\fW^{*,\rm nc}_{h0}\Lambda^{k+1},
\\ 
\mbox{and}\ \langle\fmu_h,\od^{k-1}\ftau_h\rangle_{L^2\Lambda^k}=0,\ \forall\,\ftau_h\in\fW_h\Lambda^{k-1}\Big\}=\left\{\fmu_h\in\mathcal{P}_0\Lambda^k(\mathcal{G}_h): \fmu_h\in \fV^{\od\cap\mathring{\odelta}}_h\Lambda^k\right\}.
\end{multline*}
Each $\fmu_h\in\hf_h\Lambda^k$ is piecewise constant and satisfies $\od^k_h\fmu_h=\odelta_{k,h}\fmu_h=0$, hence $\fmu_h\in\mathcal{N}(\oT_{k,h},\fV^{\od\cap\mathring{\odelta}}_h\Lambda^k)$. On the other hand, let $\fmu_h\in\fV^{\od\cap\mathring{\odelta}}_h\Lambda^k$ with $\od^k_h\fmu_h=0$ and $\odelta_{k,h}\fmu_h=0$.
On each cell $T$, $\fmu_h|_T\in\pddeL^k(T)$ and $\oT_k(\fmu_h|_T)=0$, so Remark~\ref{rem:kernel=range} gives $\fmu_h|_T\in\mathcal{P}_0\Lambda^k(T)$; thus $\fmu_h\in\mathcal{P}_0\Lambda^k(\mathcal{G}_h)$.
Therefore, $\fmu_h\in\hf_h\Lambda^k$. This completes the proof.
\end{proof}

\begin{theorem}\label{thm:pifes}
There exists a constant $C_{k,n}$, depending on the regularity of $\mathcal{G}_h$, such that 
\begin{equation}
\|\fmu_h\|_{L^2\Lambda^k}\leqslant C_{k,n}(\|\od^k_h\fmu_h\|_{L^2\Lambda^{k+1}}+\|\odelta_{k,h}\fmu_h\|_{L^2\Lambda^{k-1}}),\ \ \forall\,\fmu_h\in \fV^{\od\cap\mathring{\odelta}}_h\Lambda^k\omp \hf_h\Lambda^k.
\end{equation}
\end{theorem}

\begin{theorem}\label{thm:csbasis}
The space $\fV^{\od\cap\mathring{\odelta}}_h\Lambda^k$ admits a set of basis functions, which is each supported on no more than two cells. 
\end{theorem}

\subsubsection{Two- and three-dimensional realizations of $\fV^{\od\cap\mathring{\odelta}}_h\Lambda^k$}\label{sec:dim-realizations}

The construction~\eqref{eq:deffems} is illustrated below in vector calculus notation relevant to the cases tested in Section~\ref{sec:numer}. The local enriched space $\pddeL^k(T)$ from Section~\ref{sec:polys} specializes to $P_{\rm rd}(T)$ in 2D and $P_{\rm cd}(T)$ in 3D.

\paragraph{\bf Two dimensions ($n=2$, $k=1$).}
On a polygon $\Omega\subset\mathbb{R}^2$, we use $\nabla$, $\dv$, and $\rot$ in the usual way. We identify $H\Lambda^1(\Omega)$ with $H(\rot,\Omega)$ and $H^*_0\Lambda^1(\Omega)$ with $H_0(\dv,\Omega)$. The the corresponding shape function spaces are:
\begin{itemize}
\item $P_{\rm rd}(T)={\rm span}\left\{\left(\begin{array}{c} 1\\ 0\end{array}\right), \left(\begin{array}{c}0\\ 1\end{array}\right), \left(\begin{array}{c}x\\ y\end{array}\right), \left(\begin{array}{c}-y\\ x\end{array}\right),\left(\begin{array}{c}x^2-y^2\\ 0\end{array}\right),\ \left(\begin{array}{c} 0\\ x^2-y^2\end{array}\right)\right\}$;

\item $\mathbf{P}_{\rm c\times g}=\mathcal{P}_1(T)\times\mathcal{P}_1(T)={\rm span}\left\{1,x,y\}\times{\rm span}\{1,x,y\right\}$.
\end{itemize}
Denote $P_{\rm rd}(\mathcal{G}_h):=\prod_{T\in\mathcal{G}_h}P_{\rm rd}(T)$. With $\mathbb{V}^1_h$ the continuous piecewise linear space and $V^{\rm CR}_{h0}$ the homogeneous Crouzeix--Raviart space, the space~\eqref{eq:deffems} becomes
\begin{multline}\label{eq:deffems2d}
\fV^{\rm r\mathring{d}}_h:=\Big\{\fmu_h\in P_{\rm rd}(\mathcal{G}_h):
(\rot_h\fmu_h,\feta_h)=(\fmu_h,\curl_h\feta_h),\ \forall\,\feta_h\in V^{\rm CR}_{h0},
\\
\mbox{and}\ \ (\dv_h\fmu_h,\ftau_h)=-\langle\fmu_h,\nabla_h\ftau_h),\ \forall\,\ftau_h\in \mathbb{V}^1_h\Big\}.
\end{multline}
The local shape functions and the global basis functions are given explicitly in \AppendixRef{sec:app:basis-2d}.

\paragraph{\bf Three dimensions ($n=3$, $k=2$).}
On a polyhedral domain $\Omega\subset\mathbb{R}^3$, we use $\nabla$, $\dv$, and $\curl$ in the usual way. The formulation for $H(\dv,\Omega)\cap H_0(\curl,\Omega)$ corresponds to $H\Lambda^2(\Omega)\cap H^*_0\Lambda^2(\Omega)$. 
The corresponding shape function spaces are 
\begin{itemize}
\item $P_{\rm cd}(T)={\rm span}\left\{\begin{pmatrix} 1\\ 0\\0 \end{pmatrix}, \begin{pmatrix}0\\ 1\\0\end{pmatrix}, \begin{pmatrix}0\\ 0\\1\end{pmatrix}, \begin{pmatrix}-y\\ x\\0\end{pmatrix}, \begin{pmatrix}z\\0\\ -x\end{pmatrix},\begin{pmatrix}0\\-z\\ y\end{pmatrix},\begin{pmatrix}2x^2-y^2-z^2\\ 0\\0\end{pmatrix}, \begin{pmatrix} 0\\ 2y^2-x^2-z^2\\0\end{pmatrix}, \begin{pmatrix} 0\\ 0\\2z^2-x^2-y^2\end{pmatrix}\right\}$;

\item $P_{g\times c}=P_1(T)\times {\rm Ned}(T)={\rm span}\{1,x,,y,z\}\times {\rm span}\left\{\begin{pmatrix} 1\\ 0\\0 \end{pmatrix}, \begin{pmatrix}0\\ 1\\0\end{pmatrix}, \begin{pmatrix}0\\ 0\\1\end{pmatrix}, \begin{pmatrix}-y\\ x\\0\end{pmatrix}, \begin{pmatrix}z\\0\\ -x\end{pmatrix},\begin{pmatrix}0\\-z\\ y\end{pmatrix}\right\}$, where ${\rm Ned}(T)$ stands for the lowest-degree Nedelec element shape function space on $T$. 
\end{itemize}
Denote $P_{\rm cd}(\mathcal{G}_h):=\prod_{T\in\mathcal{G}_h}P_{\rm rd}(T)$.  With $V^{\rm CR}_{h0}$ and $\mathbb{V}^{\rm Ned}_h$ the lowest-order Crouzeix--Raviart and Nedelec spaces, the specialization of~\eqref{eq:deffems} is
\begin{multline}\label{eq:deffems3d}
\fV^{\mathring{\rm c} \rm d}_h:=\Big\{\fmu_h\in P_{\rm cd}(\mathcal{G}_h):
(\dv_h\fmu_h,\ftau_h)+(\fmu_h,\nabla_h\ftau_h)=0,\ \forall\,\ftau_h\in \mathbb{V}^{\rm CR}_{h0}, 
\\
\mbox{and}\ \  
(\rot_h\fmu_h,\feta_h)-\langle\fmu_h,\curl\feta_h)=0,\ \forall\,\feta_h\in \mathbb{V}^{\rm Ned}_{h} \Big\}.
\end{multline}
The local shape functions and global basis functions are reported in \AppendixRef{sec:app:basis-3d}.

\subsection{A primal finite element scheme for the Hodge-Laplace problem}

We consider the finite element problem: find $\fomega_h\in \fV^{\od\cap\mathring{\odelta}}_h\Lambda^k$, such that 
\begin{equation}\label{eq:modelhldisonefield}
\left\{
\begin{array}{rll}
\langle\fomega_h,\fvarsigma_h\rangle_{L^2\Lambda^k}&=0, &\forall\,\fvarsigma_h\in \hf_h\Lambda^k,
\\  
\langle\od^k_h\fomega_h,\od^k_h\fmu_h\rangle_{L^2\Lambda^{k+1}}+\langle\odelta_{k,h}\fomega_h,\odelta_{k,h}\fmu_h\rangle_{L^2\Lambda^{k-1}}&=\langle\ff-\mathbf{P}_{\hf_h\Lambda^k}\ff,\mathbf{P}^k_h\fmu_h\rangle_{L^2\Lambda^k},& \forall\,\fmu_h\in \fV^{\od\cap\mathring{\odelta}}_h\Lambda^k.
\end{array}\right.
\end{equation}
Here $\mathbf{P}_{\hf_h\Lambda^k}$ denotes the projection to $\hf_h\Lambda^k$, and $\mathbf{P}^k_h$ denotes the projection to $\mathcal{P}_0\Lambda^k(\mathcal{G}_h)$. 

\begin{remark}\label{rem:Pk-projection}
The use of $\mathbf{P}^k_h$ on the test functions in the load term is not necessary, but it includes the influence of numerical quadrature and is frequently encountered since no later than \cite{Arnold.D;Brezzi.F1985}. The projector leads to a stable modification: for $\ff\in L^2\Lambda^k$, replacing $\langle\ff,\fmu_h\rangle$ by $\langle\ff,\mathbf{P}^k_h\fmu_h\rangle$ introduces an $\mathcal{O}(h)\|\ff\|_{L^2\Lambda^k}$ perturbation, which is absorbed in the error analysis of Section~\ref{sec:fescheme} (cf.\ Lemma~\ref{lem:Pk-load}).
\end{remark}

By Theorem \ref{thm:pifes} the discrete Poincar\'e inequality, the system \eqref{eq:modelhldisonefield} is well posed. 

The main result of this paper is the theorem below.
\begin{theorem}\label{thm:convprimsch}
Let $\fomega$ and $\fomega_h$ be the solutions of \eqref{eq:modelhlori} and \eqref{eq:modelhldisonefield}, respectively. Then, 
\begin{multline*}
\|\fomega-\fomega_h\|_{L^2\Lambda^k}+\|\od^k\fomega-\od^k_h\fomega_h\|_{L^2\Lambda^{k+1}}+\|\odelta_k\fomega-\odelta_{k,h}\fomega_h\|_{L^2\Lambda^{k-1}}+\|\mathbf{P}_{\hf_h\Lambda^k}\ff-\breve\fvartheta_h\|_{L^2\Lambda^k}
\\
\leqslant C(\inf_{\ftau_h\in\fW_h\Lambda^{k-1}}\|\odelta_k\fomega-\ftau_h\|_{\od^{k-1}}+\inf_{\fmu_h\in\fW_h\Lambda^k}\|\fomega-\fmu_h\|_{\od^k}+\inf_{\fvarsigma_h\in\hf_h\Lambda^k}\|\mathbf{P}_{\hf_h\Lambda^k}\ff-\fvarsigma_h\|_{L^2\Lambda^k}+h\|\ff\|_{L^2\Lambda^k});
\end{multline*}
and further, with $\Omega$ being $s$-regular,
\begin{equation}
\|\fomega-\fomega_h\|_{L^2\Lambda^k}+\|\od^k\fomega-\od^k_h\fomega_h\|_{L^2\Lambda^{k+1}}+\|\odelta_k\fomega-\odelta_{k,h}\fomega_h\|_{L^2\Lambda^{k-1}}\leqslant Ch^s\|\ff\|_{L^2\Lambda^k}.
\end{equation}
\end{theorem}

\section{Proofs of Theorems \ref{thm:pifes} and \ref{thm:csbasis}} 
\label{sec:pfspace}

\subsection{Some technical preparation}

\begin{lemma}\label{lem:pikappa}(\cite{Zhang.S2026IMA})
There exists a constant $C_{k,n}$, depending on the regularity of $T$, such that 
\begin{equation}\label{eq:pimud}
\|\fmu\|_{L^2\Lambda^k(T)}\leqslant C_{k,n}h_T\|\od^k\fmu\|_{L^2\Lambda^{k+1}(T)},\ \ \mbox{for}\ \ \fmu\in\okappa_T(\mathcal{P}_0\Lambda^{k+1}(T)),
\end{equation}
and
\begin{equation}\label{eq:pimude}
\|\fmu\|_{L^2\Lambda^k(T)}\leqslant C_{k,n}h_T\|\odelta_k\fmu\|_{L^2\Lambda^{k+1}(T)}\ \ \mbox{for}\ \fmu\in\okappa^\delta_T(\mathcal{P}_0\Lambda^{k-1}(T)).  
\end{equation}
\end{lemma}

\begin{lemma}\label{lem:pih}
There exists a constant $C_{k,n}$, depending on the regularity of $T$, such that
\begin{equation}\label{eq:hfhde}
\|\fmu\|_{L^2\Lambda^k(T)}\leqslant C_{k,n}h_T\|\odelta_k\fmu\|_{L^2\Lambda^{k-1}(T)},\ \ \mbox{for}\ \fmu\in \mathcal{H}^2_{\odelta}\Lambda^k(T),
\end{equation}
and
\begin{equation}\label{eq:hfhd}
\|\fmu\|_{L^2\Lambda^k(T)}\leqslant C_{k,n}h_T\|\od^k\fmu\|_{L^2\Lambda^{k+1}(T)},\ \ \mbox{for}\ \fmu\in \mathcal{H}^2_{\od}\Lambda^k(T).
\end{equation}
\end{lemma}
\begin{proof}
For $\displaystyle\fmu=\sum_{\ixve\in\mathbb{IX}_{k,n}}C_{\ixve}\tilde\fmu_{\odelta,T}^{\ixve}$, by $H^1$ Poincar\'e inequality,
$$
\|\fmu\|_{L^2\Lambda^k(T)}^2\leqslant C_{k,n}h_T^2\|\sum_{\ixve\in\mathbb{IX}_{k,n}}C_{\ixve}\sum_{1\leqslant j\leqslant k}\nabla(\tilde{x}^{\ixve_j})^2\dx^{\ixve_1}\wedge\dx^{\ixve_2}\wedge\dots\wedge\dx^{\ixve_k}\|_{L^2\Lambda^k(T)}^2\leqslant C_{k,n}h_T^4|T|\sum_{\ixve\in\mathbb{IX}_{k,n}}C_{\ixve}^2.
$$
Note that $\odelta_k\fmu=2(-1)^n\sum_{\ixve\in\mathbb{IX}_{k,n}}C_{\ixve}\okappa_T(\dx^{\ixve_1}\wedge\dots\wedge\dx^{\ixve_k})$ and further $\od^{k-1}\odelta_k\fmu=2k(-1)^n\sum_{\ixve\in\mathbb{IX}_{k,n}}C_{\ixve}\dx^{\ixve_1}\wedge\dots\wedge\dx^{\ixve_k}.$ It follows by the inverse inequality that
$$
4k^2h_T^2|T|\sum_{\ixve\in\mathbb{IX}_{k,n}}C_{\ixve}^2=h_T^2\|\od^{k-1}\odelta_k\fmu\|_{L^2\Lambda^k(T)}^2\leqslant C_{k,n}\|\odelta_k\fmu\|_{L^2\Lambda^{k-1}(T)}^2.
$$
The proof of \eqref{eq:hfhde} is thus completed. Similarly can \eqref{eq:hfhd} be proved. 
\end{proof}

\begin{lemma}\label{lem:hfreq}
There exists a constant $C_{k,n}$, depending on the regularity of $T$, such that
\begin{equation}
\|\fmu\|_{L^2\Lambda^k(T)}\leqslant C_{k,n}h_T\|\odelta_k\fmu\|_{L^2\Lambda^{k-1}(T)},\ \ \mbox{for}\ \fmu\in \okappa^\delta_T(\mathcal{P}_0\Lambda^{k-1}(T))+\mathcal{H}^2_{\odelta}\Lambda^k(T),
\end{equation}
and
\begin{equation}
\|\fmu\|_{L^2\Lambda^k(T)}\leqslant C_{k,n}h_T\|\od^k\fmu\|_{L^2\Lambda^{k+1}(T)},\ \ \mbox{for}\ \fmu\in \okappa_T(\mathcal{P}_0\Lambda^{k+1}(T))+\mathcal{H}^2_{\od}\Lambda^k(T).
\end{equation}
\end{lemma}
\begin{proof}
Note that $\odelta_k(\okappa^\delta_T(\mathcal{P}_0\Lambda^{k-1}(T)))$ and $\odelta_k(\mathcal{H}^2_{\odelta}\Lambda^k(T))$ are orthogonal, and $\od^k(\okappa_T(\mathcal{P}_0\Lambda^{k+1}(T)))$ and  $\od^k(\mathcal{H}^2_{\od}\Lambda^k(T))$ are orthogonal. The lemma follows by Lemmas \ref{lem:pikappa} and \ref{lem:pih}. 
\end{proof}
Denote by $\mathbf{P}^k_{0,T}$ the $L^2$ projection onto $\mathcal{P}_0\Lambda^k(T)$.  
\begin{lemma}\label{lem:hfonT}
There exists a constant $C_{k,n}$, depending on the regularity of $T$, such that
\begin{multline}\label{eq:hfhde+}
\|\fmu\|_{L^2\Lambda^k(T)}\leqslant C_{k,n}h_T(\|\odelta_k\fmu\|_{L^2\Lambda^{k-1}(T)}+\|\od^k\fmu\|_{L^2\Lambda^{k-1}(T)}),
\\ 
\mbox{for}\ \fmu\in \okappa^\delta_T(\mathcal{P}_0\Lambda^{k-1}(T))+\okappa_T(\mathcal{P}_0\Lambda^{k+1}(T))+\mathcal{H}^2_{\odelta}\Lambda^k(T)+\mathcal{H}^2_{\od}\Lambda^k(T),
\end{multline}
and
\begin{multline}\label{eq:hfhd+}
\|\fmu-\mathbf{P}^k_{0,T}\fmu\|_{L^2\Lambda^k(T)}\leqslant C_{k,n}h_T(\|\odelta_k\fmu\|_{L^2\Lambda^{k-1}(T)}+\|\od^k\fmu\|_{L^2\Lambda^{k-1}(T)}),
\\ 
\mbox{for}\ \fmu\in \mathcal{P}_0\Lambda^k(T)+\okappa^\delta_T(\mathcal{P}_0\Lambda^{k-1}(T))+\okappa_T(\mathcal{P}_0\Lambda^{k+1}(T))+\mathcal{H}^2_{\odelta}\Lambda^k(T)+\mathcal{H}^2_{\od}\Lambda^k(T).
\end{multline}
\end{lemma}
\begin{proof}
For $\fmu\in \okappa^\delta_T(\mathcal{P}_0\Lambda^{k-1}(T))+\okappa_T(\mathcal{P}_0\Lambda^{k+1}(T))+\mathcal{H}^2_{\odelta}\Lambda^k(T)+\mathcal{H}^2_{\od}\Lambda^k(T)$, decompose $\fmu=\fmu_1+\fmu_2$, such that $\fmu_1\in \okappa^\delta_T(\mathcal{P}_0\Lambda^{k-1}(T))+\mathcal{H}^2_{\odelta}\Lambda^k(T)$ and $\fmu_2\in \okappa_T(\mathcal{P}_0\Lambda^{k+1}(T))+\mathcal{H}^2_{\od}\Lambda^k(T)$, then $\od^k\fmu_1=0$ and $\odelta_k\fmu_2=0$. Then \eqref{eq:hfhde+} follows by Lemma \ref{lem:hfreq}. Similarly can \eqref{eq:hfhd+} be proved. 
\end{proof}

Now we introduce the notations below for short:
\begin{equation}\label{eq:notaringpdd}
\mathring{\mathbf{P}}_{\odelta\times\od}^k(T):=\N(\oT^*_k,\mathbf{P}_{\odelta\times\od}^k(T)),\ \ \mbox{and},\ \ \mathring{\mathbf{P}}_{\odelta\times\od}^k(\mathcal{G}_h):=\prod_{T\in\mathcal{G}_h}\mathring{\mathbf{P}}_{\odelta\times\od}^k(T).
\end{equation}

\begin{lemma}[Helmholtz-type decomposition]\label{lem:htd}
For $1\leqslant k\leqslant n-1$, orthogonal in $L^2\Lambda^k$,
\begin{equation}\label{eqn:hdp0}
\mathcal{P}_0\Lambda^k(\mathcal{G}_h)=\R(\oT^*_{k,h},\fW^{*,\rm nc}_{h0}\Lambda^{k+1}\times\fW_h\Lambda^{k-1})\opp \N(\oT_{k,h},\fV^{\od\cap\mathring{\odelta}}_h\Lambda^k),
\end{equation}
\begin{equation}\label{eqn:hdpdd}
\mathring{\mathbf{P}}_{\odelta\times\od}^k(\mathcal{G}_h)=\N(\oT^*_{k,h},\fW^{*,\rm nc}_{h0}\Lambda^{k+1}\times\fW_h\Lambda^{k-1})\opp \R(\oT_{k,h},\fV^{\od\cap\mathring{\odelta}}_h\Lambda^k).
\end{equation}
\end{lemma}
\begin{proof}
Clearly, $\mathcal{P}_0\Lambda^k(\mathcal{G}_h)\supset\R(\oT^*_{k,h},\fW^{*,\rm nc}_{h0}\Lambda^{k+1}\times\fW_h\Lambda^{k-1})\opp \N(\oT_{k,h},\fV^{\od\cap\mathring{\odelta}}_h\Lambda^k).$ On the other hand, given $\fmu_h\in \mathcal{P}_0\Lambda^k(\mathcal{G}_h)\omp \R(\oT^*_{k,h},\fW^{*,\rm nc}_{h0}\Lambda^{k+1}\times\fW_h\Lambda^{k-1})$, it holds for $\left(\begin{array}{c}\feta_h \\ \ftau_h\end{array}\right)\in \fW^{*,\rm nc}_{h0}\Lambda^{k+1}\times\fW_h\Lambda^{k-1}$ that $\left\langle\fmu_h,\oT^*_{k,h}\left(\begin{array}{c}\feta_h \\ \ftau_h\end{array}\right)\right\rangle_{L^2\Lambda^k}=0=\left\langle\oT_{k,h}\fmu_h,\left(\begin{array}{c}\feta_h \\ \ftau_h\end{array}\right)\right\rangle$; then it follows that $\fmu_h\in \fV^{\od\cap\mathring{\odelta}}_h\Lambda^k$ and further $\N(\oT_{k,h},\fV^{\od\cap\mathring{\odelta}}_h\Lambda^k)$. Hence \eqref{eqn:hdp0} holds. The proof of \eqref{eqn:hdpdd} is analogous. 
\end{proof}

\subsection{Proof of Theorem \ref{thm:pifes}} 
Firstly, following \cite{Zhang.S2026IMA}, we introduce the notion {\bf Poincar\'e inequality's criterion} of $(\oT,\xD)$. Let $\xX$ and $\yY$ be two Hilbert spaces. For $(\oT,\xD):\xX\to \yY$ a closed  operator, denote 
$$
\xD^{\boldsymbol \lrcorner_{\oT}}:=\left\{\xv\in \xD:\langle \xv,\xw\rangle_\xX=0,\ \forall\,\xw\in \mathcal{N}(\oT,\xD)\right\}.
$$ 
The {\bf Poincar\'e inequality's criterion} of $(\oT,\xD)$ is defined as
\begin{equation}\label{eq:deficr}
\mathsf{pic}(\oT,\xD):=\left\{\begin{array}{rl}
\displaystyle \sup_{0\neq\xv\in \xD^{\boldsymbol \lrcorner_{\oT}}}\frac{\|\xv\|_\xX}{\|\oT\xv\|_\yY},&\mbox{if}\ \xD^{\boldsymbol \lrcorner_{\oT}}\neq\left\{0\right\};
\\
0,&\mbox{if}\ \xD^{\boldsymbol \lrcorner_{\oT}}=\left\{0\right\}.
\end{array}\right.
\end{equation}

If $\icr(\oT,\xD)$ is finite, then the Poincar\'e inequality holds for $(\oT,\xD)$, and $\icr(\oT,\xD)$ is precisely the best constant. 

\begin{lemma}(\cite{Arnold.D;Falk.R;Winther.R2006acta}\cite{Zhang.S2026IMA})\label{lem:piwfs}
There exists a constant $C_{k,n}$, depending on the regularity of $\mathcal{G}_h$, such that 
$$
\icr(\od^k,\fW_h\Lambda^k)\leqslant C_{k,n},\ \ \mbox{and},\ \ \icr(\od^k,\fW_{h0}\Lambda^k)\leqslant C_{k,n},
$$
and 
$$
\icr(\od^k_h,\fW^{\rm nc}_h\Lambda^k)\leqslant C_{k,n},\ \ \mbox{and},\ \ \icr(\od^k_h,\fW^{\rm nc}_{h0}\Lambda^k)\leqslant C_{k,n}.
$$
\end{lemma}
Throughout, the subscript $\cdot_h$ denotes the piecewise operation on $\mathcal{G}_h$.
\begin{remark}
By Lemma \ref{lem:piwfs}, it follows by definition that 
$$
\icr(\odelta_k,\fW^*_h\Lambda^k)\leqslant C_{n-k,n},\ \ \mbox{and},\ \ \icr(\odelta_k,\fW^*_{h0}\Lambda^k)\leqslant C_{n-k,n},
$$
and
$$
\icr(\odelta_{k,h},\fW^{*,\rm nc}_h\Lambda^k)\leqslant C_{n-k,n},\ \ \mbox{and},\ \ \icr(\odelta_{k,h},\fW^{*,\rm nc}_{h0}\Lambda^k)\leqslant C_{n-k,n}.
$$
Without ambiguity, we do not distinguish the generic constants $C_{n-k,n}$ and $C_{k,n}$ in the sequel. 
\end{remark}

\begin{remark}
By Lemma \ref{lem:hoddecconst}, $\R(\odelta_{k+1,h},\fW^{*,\rm nc}_{h0}\Lambda^{k+1})$ is orthogonal to $\R(\od^{k-1}_h,\fW_h\Lambda^{k-1})$. Therefore, $\N(\oT^*_{k,h},\fW^{*,\rm nc}_{h0}\Lambda^{k+1}\times \fW_h\Lambda^{k-1})=\N(\odelta_{k+1,h},\fW^{*,\rm nc}_{h0}\Lambda^{k+1}) \times \N(\od^{k-1}_h,\fW_h\Lambda^{k-1})$. Further, 
\begin{equation}\label{eq:pictstar}
\icr(\oT^*_{k,h},\fW^{*,\rm nc}_{h0}\Lambda^{k+1}\times \fW_h\Lambda^{k-1})=\max\left( \icr(\odelta_{k+1,h},\fW^{*,\rm nc}_{h0}\Lambda^{k+1}), \icr(\od^{k-1}_h,\fW_h\Lambda^{k-1}) \right)\leqslant C_{k,n}.
\end{equation}
\end{remark}

\paragraph{\bf Proof of Theorem~\ref{thm:pifes}.} First, decompose $\fV^{\od\cap\mathring{\odelta}}_h\Lambda^k= (\fV^{\od\cap\mathring{\odelta}}_h\Lambda^k)^{\boldsymbol \lrcorner}\opp\hf_h\Lambda^k.$ Now, given $\fmu_h\in (\fV^{\od\cap\mathring{\odelta}}_h\Lambda^k)^{\boldsymbol \lrcorner}$, decompose further $\fmu_h=\mathring{\fmu}_h+\fmu_h^{\boldsymbol \perp}$, such that, with $\mathring{\fmu}_h\in \mathcal{P}_0\Lambda^k(\mathcal{G}_h)$ and $\fmu_h^{\boldsymbol \perp}\in \pddeL^k(\mathcal{G}_h)$ being orthogonal to $\mathcal{P}_0\Lambda^k(\mathcal{G}_h)$, i.e.\ $\displaystyle \fmu_h^{\boldsymbol \perp}\in\prod_{T\in\mathcal{G}_h}\left[\okappa^\delta_T(\mathcal{P}_0\Lambda^{k-1}(T))+\okappa_T(\mathcal{P}_0\Lambda^{k+1}(T))+\mathcal{H}^2_{\od\cap\odelta}\Lambda^k(T)\right]$. It follows that $\mathring{\fmu}_h=\fmu_h-\fmu_h^{\boldsymbol \perp}$ is orthogonal to $\hf_h\Lambda^k$, and hence $\mathring{\fmu}_h\in \R(\oT^*_{k,h},\fW^{*,\rm nc}_{h0}\Lambda^{k+1}\times\fW_h\Lambda^{k-1})$ by Lemma \ref{lem:htd}. Therefore,
\begin{multline*}
\|\mathring{\fmu}_h\|_{L^2\Lambda^k}
=
\sup_{\left(\begin{array}{c}\feta_h \\ \ftau_h\end{array}\right)\in \left(\fW^{*,\rm nc}_{h0}\Lambda^{k+1}\times\fW_h\Lambda^{k-1}\right)^{\boldsymbol \lrcorner}}
\frac{\left\langle \mathring{\fmu}_h,\oT^*_{k,h}\left(\begin{array}{c}\feta_h \\ \ftau_h\end{array}\right)\right\rangle_{L^2\Lambda^k}}{\left\|\oT^*_{k,h}\left(\begin{array}{c}\feta_h \\ \ftau_h\end{array}\right)\right\|_{L^2\Lambda^k}} 
\\
=
\sup_{\left(\begin{array}{c}\feta_h \\ \ftau_h\end{array}\right)\in \left(\fW^{*,\rm nc}_{h0}\Lambda^{k+1}\times\fW_h\Lambda^{k-1}\right)^{\boldsymbol \lrcorner}}
\frac{\left\langle \fmu_h^{\boldsymbol \perp},\oT^*_{k,h}\left(\begin{array}{c}\feta_h \\ \ftau_h\end{array}\right)\right\rangle_{L^2\Lambda^k}-\left\langle \oT_{k,h}\fmu_h^{\boldsymbol \perp},\left(\begin{array}{c}\feta_h \\ \ftau_h\end{array}\right)\right\rangle}{\left\|\oT^*_{k,h}\left(\begin{array}{c}\feta_h \\ \ftau_h\end{array}\right)\right\|_{L^2\Lambda^k}}
\\
\leqslant 
\| \fmu_h^{\boldsymbol \perp}\|_{L^2\Lambda^k}+\|\oT_{k,h}\fmu_h\|_{L^2\Lambda^{k+1}\times L^2\Lambda^{k-1}}  \icr(\oT^*_{k,h},\fW^{*,\rm nc}_{h0}\Lambda^{k+1}\times\fW_h\Lambda^{k-1}).
\end{multline*}
By Lemma \ref{lem:hfonT}, $\|\fmu_h^{\boldsymbol \perp}\|_{L^2\Lambda^k}\leqslant Ch(\|\od^k_h\fmu_h^{\boldsymbol \perp}\|_{L^2\Lambda^{k+1}}+\|\odelta_{k,h}\fmu_h^{\boldsymbol \perp}\|_{L^2\Lambda^{k-1}})$. Hence, by \eqref{eq:pictstar},
$$
\|\fmu_h\|_{L^2\Lambda^k}\leqslant ch\|\oT_{k,h}\fmu_h\|_{L^2\Lambda^{k+1}\times L^2\Lambda^{k-1}}+C_{k,n}\|\oT_{k,h}\fmu_h\|_{L^2\Lambda^{k+1}\times L^2\Lambda^{k-1}}\leqslant C_{k,n}\|\oT_{k,h}\fmu_h\|_{L^2\Lambda^{k+1}\times L^2\Lambda^{k-1}}.
$$
The proof is completed. \qed

\subsection{Proof of Theorem \ref{thm:csbasis}}\label{sec:csbasis-proof}

Given a triangle $T$, define the cell-wise interpolator 
\begin{equation}
\mathbb{I}^{\od\cap\odelta}_T: H\Lambda^k\cap H^*\Lambda^k\to\mathbf{P}_{\od\cap\odelta}\Lambda^k(T)
\end{equation}
such that
\begin{multline}
\langle\od^k(\mathbb{I}^{\od\cap\odelta}_T\fmu),\feta\rangle_{L^2\Lambda^{k+1}(T)}+\langle\odelta_k(\mathbb{I}^{\od\cap\odelta}_T\fmu),\ftau\rangle_{L^2\Lambda^{k-1}(T)}-\langle\mathbb{I}^{\od\cap\odelta}_T\fmu,(\odelta_{k+1}\feta+\od^{k-1}\ftau)\rangle_{L^2\Lambda^k(T)}
\\
=\langle\od^k\fmu,\feta\rangle_{L^2\Lambda^{k+1}(T)}+\langle\odelta_k\fmu,\ftau\rangle_{L^2\Lambda^{k-1}(T)}-\langle\fmu,(\odelta_{k+1}\feta+\od^{k-1}\ftau)\rangle_{L^2\Lambda^k(T)}\quad \forall\,\left(\begin{array}{c}\feta \\ \ftau\end{array}\right)\in \mathbf{P}_{\odelta\times\od}^k(T).
\end{multline}

Further, define two partial interpolators, by
\begin{equation}
\mathbb{I}^{\od\backslash\odelta}_T: H\Lambda^k\cap H^*\Lambda^k\to\mathbf{P}_{\od\cap\odelta}\Lambda^k(T)
\end{equation}
such that
\begin{itemize}
\item for $\forall\,\feta\in\mathcal{P}^{*,-}_1\Lambda^{k+1}(T),$
$$
\langle\od^k(\mathbb{I}^{\od\backslash\odelta}_T\fmu),\feta\rangle_{L^2\Lambda^{k+1}(T)}-\langle\mathbb{I}^{\od\backslash\odelta}_T\fmu,\odelta_{k+1}\feta\rangle_{L^2\Lambda^k(T)}
=\langle\od^k\fmu,\feta\rangle_{L^2\Lambda^{k+1}(T)}-\langle\fmu,\odelta_{k+1}\feta\rangle_{L^2\Lambda^k(T)}, 
$$
\item and for $\forall\ \ftau\in \mathcal{P}^-_1\Lambda^{k-1}(T)$, 
$$
\langle\odelta_k(\mathbb{I}^{\od\backslash\odelta}_T\fmu),\ftau\rangle_{L^2\Lambda^{k-1}(T)}-\langle\mathbb{I}^{\od\backslash\odelta}_T\fmu,\od^{k-1}\ftau\rangle_{L^2\Lambda^k(T)}=0, 
$$
\end{itemize}
and by
\begin{equation}
\mathbb{I}^{\odelta\backslash\od}_T: H\Lambda^k\cap H^*\Lambda^k\to\mathbf{P}_{\od\cap\odelta}\Lambda^k(T)
\end{equation}
such that
\begin{itemize}
\item for $\forall\ \ftau\in \mathcal{P}^{*,-}_1\Lambda^{k+1}(T).$
$$
\langle\od^k(\mathbb{I}^{\odelta\backslash\od}_T\fmu),\ftau\rangle_{L^2\Lambda^{k-1}(T)}-\langle\mathbb{I}^{\odelta\backslash\od}_T\fmu,\odelta_{k+1}\ftau\rangle_{L^2\Lambda^k(T)}=0,
$$
\item and for $\forall\,\feta\in\mathcal{P}^-_1\Lambda^{k-1}(T)$,
$$
\langle\odelta_k(\mathbb{I}^{\odelta\backslash\od}_T\fmu),\feta\rangle_{L^2\Lambda^{k+1}(T)}-\langle\mathbb{I}^{\odelta\backslash\od}_T\fmu,\od^{k-1}\feta\rangle_{L^2\Lambda^k(T)}
=\langle\odelta_k\fmu,\feta\rangle_{L^2\Lambda^{k+1}(T)}-\langle\fmu,\od^{k-1}\feta\rangle_{L^2\Lambda^k(T)}.
$$
\end{itemize}
Evidently, $\mathbb{I}^{\od\cap\odelta}_T$ is well-posed by Remark \ref{rem:kernel=range}, and $\mathbb{I}^{\od\backslash\odelta}_T$ and $\mathbb{I}^{\odelta\backslash\od}_T$ are both well-posed by \eqref{eq:n=r=c}.
\begin{remark} 
$\mathbb{I}^{\od\backslash\odelta}_T$ is bijective from $\mathcal{P}^-_1\Lambda^k(T)$ onto $\R(\mathbb{I}^{\od\backslash\odelta}_T,\mathcal{P}^-_1\Lambda^k(T))$ and $\mathbb{I}^{\odelta\backslash\od}_T$ is bijective from $\mathcal{P}^{*,-}_1\Lambda^k(T)$ onto $\R(\mathbb{I}^{\odelta\backslash\od}_T, \mathcal{P}^{*,-}_1\Lambda^k(T))$. However, neither $\mathbb{I}^{\od\backslash\odelta}_T$ or $\mathbb{I}^{\odelta\backslash\od}_T$ is projective on $\mathcal{P}_0\Lambda^k=\mathcal{P}^-_1\Lambda^k(T)\cap \mathcal{P}^{*,-}_1\Lambda^k(T)$. 
\end{remark}
\begin{remark}
It holds that $\mathbb{I}^{\od\cap\odelta}_T=\mathbb{I}^{\od\backslash\odelta}_T+\mathbb{I}^{\odelta\backslash\od}_T$ and $\mathbf{P}_{\od\cap\odelta}\Lambda^k(T) = \R(\mathbb{I}^{\od\backslash\odelta}_T,\mathcal{P}^-_1\Lambda^k(T)) \oplus \R(\mathbb{I}^{\odelta\backslash\od}_T, \mathcal{P}^{*,-}_1\Lambda^k(T))$. 
\end{remark}

Further, define two global partial interpolators, by
\begin{equation}
\mathbb{I}^{\od\backslash\odelta}_h: \prod_{T\in\mathcal{G}_h}H\Lambda^k\cap H^*\Lambda^k(T)\to \prod_{T\in\mathcal{G}_h}\mathbf{P}_{\od\cap\odelta}\Lambda^k(T)
\end{equation}
such that
\begin{equation}
(\mathbb{I}^{\od\backslash\odelta}_h\fmu_h)_T=\mathbb{I}^{\od\backslash\odelta}_T(\fmu_h|_T),\ \forall\,T\in\mathcal{G}_h,
\end{equation}
and by
\begin{equation}
\mathbb{I}^{\odelta\backslash\od}_h: \prod_{T\in\mathcal{G}_h}H\Lambda^k\cap H^*\Lambda^k(T)\to\prod_{T\in\mathcal{G}_h}\mathbf{P}_{\od\cap\odelta}\Lambda^k(T),
\end{equation}
such that
\begin{equation}
(\mathbb{I}^{\odelta\backslash\od}_h\fmu_h)_T=\mathbb{I}^{\odelta\backslash\od}_T(\fmu_h|_T),\ \forall\,T\in\mathcal{G}_h. 
\end{equation}

\begin{remark}
$\mathbb{I}^{\od\backslash\odelta}_h$ and $\mathbb{I}^{\odelta\backslash\od}_h$ are injective on $\fW_h\Lambda^k$ and $\fW^{*,\rm nc}_{h0}\Lambda^k$, respectively. Though, neither $\mathbb{I}^{\od\backslash\odelta}_h$ or $\mathbb{I}^{\odelta\backslash\od}_h$ is projective on $\fW_h\Lambda^k\cap \fW^{*,\rm nc}_{h0}\Lambda^k\subset\mathcal{P}_0\Lambda^k(\mathcal{G}_h)$.
\end{remark}

\begin{theorem}\label{thm:decvddelta}
$\fV^{\od\cap\mathring{\odelta}}_h\Lambda^k=\R(\mathbb{I}^{\od\backslash\odelta}_h,\fW_h\Lambda^k) \oplus \R(\mathbb{I}^{\odelta\backslash\od}_h, \fW^{*,\rm nc}_{h0}\Lambda^k). $
\end{theorem}

\begin{proof}
Clearly, $\R(\mathbb{I}^{\od\backslash\odelta}_h,\fW_h\Lambda^k)\subset \fV^{\od\cap\mathring{\odelta}}_h\Lambda^k$ and $\R(\mathbb{I}^{\odelta\backslash\od}_h, \fW^{*,\rm nc}_{h0}\Lambda^k)\subset \fV^{\od\cap\mathring{\odelta}}_h\Lambda^k$. Hence 
$$
\fV^{\od\cap\mathring{\odelta}}_h\Lambda^k\supset\R(\mathbb{I}^{\od\backslash\odelta}_h,\fW_h\Lambda^k) \oplus \R(\mathbb{I}^{\odelta\backslash\od}_h, \fW^{*,\rm nc}_{h0}\Lambda^k).
$$
On the other hand,
$$
\begin{array}{rl}
&\fV^{\od\cap\mathring{\odelta}}_h\Lambda^k
\\
=&\left\{\fmu_h\in\pddeL^k(\mathcal{G}_h):\ \left\langle\oT_{k,h}\fmu_h,\left(\begin{array}{c}\feta_h\\ \ftau_h\end{array}\right)\right\rangle = \left\langle \fmu_h,\oT^*_{k,h}\left(\begin{array}{c}\feta_h\\ \ftau_h\end{array}\right)\right\rangle_{L^2\Lambda^k},\ \forall\,\left(\begin{array}{c}\feta_h\\ \ftau_h\end{array}\right)\in \fW^{*,\rm nc}_{h0}\Lambda^{k+1}\times \fW_h\Lambda^{k-1}\right\}
\\
=&\displaystyle \left\{\fmu_h\in\prod_{T\in\mathcal{G}_h}\left(\R\left(\mathbb{I}^{\od\backslash\odelta}_T,\mathcal{P}^-_1\Lambda^k(T)\right) \oplus \R\left(\mathbb{I}^{\odelta\backslash\od}_T, \mathcal{P}^{*,-}_1\Lambda^k(T)\right)\right):\right.\ 
\\
&\qquad\qquad \left.\left\langle\oT_{k,h}\fmu_h,\left(\begin{array}{c}\feta_h\\ \ftau_h\end{array}\right)\right\rangle = \left\langle \fmu_h,\oT^*_{k,h}\left(\begin{array}{c}\feta_h\\ \ftau_h\end{array}\right)\right\rangle_{L^2\Lambda^k},\ \forall\,\left(\begin{array}{c}\feta_h\\ \ftau_h\end{array}\right)\in \fW^{*,\rm nc}_{h0}\Lambda^{k+1}\times \fW_h\Lambda^{k-1}\right\}
\\
=&\displaystyle \left\{\fmu_h\in\prod_{T\in\mathcal{G}_h}\left(\R\left(\mathbb{I}^{\od\backslash\odelta}_T,\mathcal{P}^-_1\Lambda^k(T)\right)\right):\langle\od^k_h\fmu_h,\feta_h\rangle_{L^2\Lambda^{k+1}}=\langle\fmu_h,\odelta_{k+1,h}\feta_h\rangle_{L^2\Lambda^k},\ \forall\,\feta_h\in\fW^{*,\rm nc}_{h0}\Lambda^{k+1}\right\}
\\
&\displaystyle \bigoplus\left\{\fmu_h\in\prod_{T\in\mathcal{G}_h}\left(\R\left(\mathbb{I}^{\odelta\backslash\od}_T, \mathcal{P}^{*,-}_1\Lambda^k(T)\right)\right): \ \ \langle\odelta_{k,h}\fmu_h,\ftau_h\rangle_{L^2\Lambda^{k-1}}=\langle\fmu_h,\od^{k-1}\ftau_h\rangle_{L^2\Lambda^k},\ \forall\,\ftau_h\in\fW_h\Lambda^{k-1}\right\}
\\
:= &\displaystyle I_1\oplus I_2. 
\end{array}
$$
Denote by $\breve{\mathbb{I}}^{\od\backslash\odelta}_T$ the inverse of $\mathbb{I}^{\od\backslash\odelta}_T$ from $\R(\mathbb{I}^{\od\backslash\odelta}_T,\mathcal{P}^-_1\Lambda^k(T))$ onto $\mathcal{P}^-_1\Lambda^k(T)$, and $\breve{\mathbb{I}}^{\od\backslash\odelta}_h$ the piecewise combination of $\breve{\mathbb{I}}^{\od\backslash\odelta}_T$. Then given $\fmu_h\in I_1$,  it follows that 
\begin{multline*}
\sum_{T\in\mathcal{G}_h}\langle\od^k_h\breve{\mathbb{I}}^{\od\backslash\odelta}_h\fmu_h,\feta_h\rangle_{L^2\Lambda^{k+1}(T)}-\langle\breve{\mathbb{I}}^{\od\backslash\odelta}_h\fmu_h,\odelta_{k+1,h}\feta_h\rangle_{L^2\Lambda^k(T)}
\\
=\sum_{T\in\mathcal{G}_h}\langle\od^k_h\fmu_h,\feta_h\rangle_{L^2\Lambda^{k+1}(T)}-\langle\fmu_h,\odelta_{k+1,h}\feta_h\rangle_{L^2\Lambda^k(T)}=0,\ \ \forall\,\feta_h\in\fW^{*,\rm nc}_{h0}\Lambda^{k+1};
\end{multline*}
namely $\breve{\mathbb{I}}^{\od\backslash\odelta}_h\fmu_h\in \fW_h\Lambda^k$, and $\fmu_h\in \R(\mathbb{I}^{\od\backslash\odelta}_h,\fW_h\Lambda^k)$. Hence, $I_1\subset \R(\mathbb{I}^{\od\backslash\odelta}_h,\fW_h\Lambda^k)$. Similarly $I_2\subset \R(\mathbb{I}^{\odelta\backslash\od}_h, \fW^{*,\rm nc}_{h0}\Lambda^k)$. It follows that $\fV^{\od\cap\mathring{\odelta}}_h\Lambda^k\subset\R(\mathbb{I}^{\od\backslash\odelta}_h,\fW_h\Lambda^k) \oplus \R(\mathbb{I}^{\odelta\backslash\od}_h, \fW^{*,\rm nc}_{h0}\Lambda^k)$, and the proof is completed. 
\end{proof}

\paragraph{\bf Proof of Theorem \ref{thm:csbasis}} It suffices to note that, given a set of linearly dependent basis functions of $\fW_h\Lambda^k$ and a set of $\fW^{*,\rm nc}_{h0}\Lambda^k$, a set of linearly dependent basis functions of $\fV^{\od\cap\mathring{\odelta}}_h\Lambda^k$ follows. Actually, for example, given a set of basis functions $\{\feta{}_h^j\}$ of $\fW_h\Lambda^k$, the support of $\mathbb{I}^{\od\backslash\odelta}_h\feta{}_h^j$ does not exceed the support of $\feta{}_h^j$, and we only have to show that $\{\mathbb{I}^{\od\backslash\odelta}_h\feta{}_h^j\}$ are linearly independent. For this, provided that a linear combination of all these functions $\left\{\mathbb{I}^{\od\backslash\odelta}_h\feta{}_h^j\right\}$ vanishes everywhere, as $\mathbb{I}^{\od\backslash\odelta}_h$ is piecewise bijective, it follows that the same combination of $\left\{\feta{}_h^j\right\}$ vanishes everywhere. Therefore, by Theorem \ref{thm:decvddelta}, an explicit set of basis functions of $\fV^{\od\cap\mathring{\odelta}}_h\Lambda^k$, each supported on no more than two cells, are constructed. 
\qed

%
%
%
\section{Proof of Theorem \ref{thm:convprimsch}}
\label{sec:fescheme}

For the ease of theoretical analysis below, we introduce two auxiliary formulations, namely

\begin{itemize}
\item find $\fomega\in H\Lambda^k(\Omega)\cap H^*_0\Lambda^k(\Omega)$ and $\fvartheta\in\hf\Lambda^k$, such that, for any $\fmu\in H\Lambda^k(\Omega)\cap H^*_0\Lambda^k(\Omega)$ and $\fvarsigma\in \hf\Lambda^k$,
\begin{equation}\label{eq:modelhloriexpa}
\left\{
\begin{array}{rl}
\langle\fomega,\fvarsigma\rangle_{L^2\Lambda^k}&=0,
\\  
\langle\fvartheta,\fmu\rangle_{L^2\Lambda^k}+\langle\od^k\fomega,\od^k\fmu\rangle_{L^2\Lambda^{k+1}}+\langle\odelta_k\fomega,\odelta_k\fmu\rangle_{L^2\Lambda^{k-1}}&=\langle\ff,\fmu\rangle_{L^2\Lambda^k};
\end{array}\right.
\end{equation}

\item find $(\fomega_h,\fvartheta_h)\in \fV^{\od\cap\mathring{\odelta}}_h\Lambda^k \times \hf_h\Lambda^k$, such that, for any $\fvarsigma_h\in\hf_h\Lambda^k$ and $\fmu_h\in \fV^{\od\cap\mathring{\odelta}}_h\Lambda^k$, 
\begin{equation}\label{eq:modelhldis}
\left\{
\begin{array}{rl}
\langle\fomega_h,\fvarsigma_h\rangle_{L^2\Lambda^k}&=0, 
\\  
\langle\fvartheta_h,\fmu_h\rangle_{L^2\Lambda^k}+\langle\od^k_h\fomega_h,\od^k_h\fmu_h\rangle_{L^2\Lambda^{k+1}}+\langle\odelta_{k,h}\fomega_h,\odelta_{k,h}\fmu_h\rangle_{L^2\Lambda^{k-1}}&=\langle\ff,\mathbf{P}^k_h\fmu_h\rangle_{L^2\Lambda^k}.
\end{array}\right.
\end{equation}
\end{itemize}
One readily verifies that \eqref{eq:modelhloriexpa} is equivalent to \eqref{eq:modelhlori}, and \eqref{eq:modelhldisonefield} is equivalent to \eqref{eq:modelhldis}. 

\begin{lemma}\label{lem:Pk-load}
Let $(\fomega_h,\fvartheta_h)$ solve \eqref{eq:modelhldis} and let $(\grave\fomega_h,\grave\fvartheta_h)$ solve the same system with $\langle\ff,\fmu_h\rangle_{L^2\Lambda^k}$ in place of $\langle\ff,\mathbf{P}^k_h\fmu_h\rangle_{L^2\Lambda^k}$ on the right-hand side. Then
\[
\|\fomega_h-\grave\fomega_h\|_{L^2\Lambda^k}+\|\oT_{k,h}(\fomega_h-\grave\fomega_h)\|_{L^2\Lambda^{k+1}\times L^2\Lambda^{k-1}}
\leqslant C h\|\ff\|_{L^2\Lambda^k}.
\]
\end{lemma}
\begin{proof}
The difference $\fomega_h-\grave\fomega_h$ satisfies $\langle\oT_{k,h}(\fomega_h-\grave\fomega_h),\oT_{k,h}\fmu_h\rangle=\langle\ff,(\mathbf{I}-\mathbf{P}^k_h)\fmu_h\rangle_{L^2\Lambda^k}$ for all $\fmu_h\in\fV^{\od\cap\mathring{\odelta}}_h\Lambda^k$ orthogonal to $\hf_h\Lambda^k$. Then Lemma \ref{lem:hfonT} yields the bound.
\end{proof}

\subsection{Auxiliary mixed formulations and their connections}
\label{sec:errorest}

We introduce several auxiliary problems in mixed formulations.

\begin{enumerate}[(A)]
\item Denote by $\mathbf{P}_{\rppdd}$ the $L^2$ projection onto $\rppdd$, which is a combination of the piecewise projection to $\mathring{\mathbf{P}}_{\odelta\times\od}^k(T)$. We consider an auxiliary problem: find $(\overline{\fomega}_h, \overline{\fzeta}_h, \overline{\fsigma}_h, \overline{\fvartheta}_h)\in \mathcal{P}_0\Lambda^k(\mathcal{G}_h) \times \fW^{*,\rm nc}_{h0}\Lambda^{k+1}\times \fW_h\Lambda^{k-1} \times \hf_h\Lambda^k$, such that, for $(\overline{\fmu}_h, \overline{\feta}_h, \overline{\ftau}_h, \overline{\fvarsigma}_h)\in \mathcal{P}_0\Lambda^k(\mathcal{G}_h) \times \fW^{*,\rm nc}_{h0}\Lambda^{k+1}\times \fW_h\Lambda^{k-1} \times \hf_h\Lambda^k$,
\begin{equation}\label{eq:equimodelhldis}
\mbox{\bf (AP-I)}\ \ \left\{
\begin{array}{lccl}
&&\langle\overline{\fomega}_h,\overline{\fvarsigma}_h\rangle_{L^2\Lambda^k}&=0, 
\\
&\left\langle\mathbf{P}_{\rppdd}\left(\begin{array}{c}\overline{\fzeta}_h \\ \overline{\fsigma}_h \end{array}\right),\left(\begin{array}{c}\overline{\feta}_h \\ \overline{\ftau}_h \end{array}\right)\right\rangle
&-\left\langle\overline{\fomega}_h,\oT^*_{k,h}\left(\begin{array}{c}\overline{\feta}_h \\ \overline{\ftau}_h \end{array}\right)\right\rangle_{L^2\Lambda^k}&=0,
\\
\langle\overline{\fvartheta}_h,\overline{\fmu}_h\rangle_{L^2\Lambda^k}&-\left\langle\overline{\fmu}_h,\oT^*_{k,h}\left(\begin{array}{c}\overline{\fzeta}_h \\ \overline{\fsigma}_h \end{array}\right)\right\rangle_{L^2\Lambda^k}&&=\langle\ff,\overline{\fmu}_h\rangle_{L^2\Lambda^k}.
\end{array}
\right.
\end{equation}

\item find $(\tilde{\fomega}_h,\tilde{\fzeta}_h,\tilde{\fsigma}_h,\tilde{\fvartheta}_h)\in \mathcal{P}_0\Lambda^k(\mathcal{G}_h)\times \mathbf{W}^{*,\rm nc}_{h0}\Lambda^{k+1}\times \mathbf{W}_h\Lambda^{k-1}\times\hf_h\Lambda^k$, such that, for $(\tilde{\fmu}_h,\tilde{\feta}_h,\tilde{\ftau}_h,\tilde{\fvarsigma}_h)\in \mathcal{P}_0\Lambda^k(\mathcal{G}_h)\times \mathbf{W}^{*,\rm nc}_{h0}\Lambda^{k+1}\times \mathbf{W}_h\Lambda^{k-1}\times\hf_h\Lambda^k$, 
\begin{equation}\label{eq:mixedhodgedis}
\mbox{\bf (AP-II)}\ \ \left\{
\begin{array}{ccccl}
&&&\langle\tilde{\fomega}_h,\tilde{\fvarsigma}_h\rangle_{L^2\Lambda^k}&=0
\\
&\langle \mathbf{P}_h^{k+1}\tilde{\fzeta}_h,\tilde{\feta}_h\rangle_{L^2\Lambda^{k+1}}&&-\langle\tilde{\fomega}_h,\odelta_{k+1,h}\tilde{\feta}_h\rangle_{L^2\Lambda^k}&=0
\\
&&\langle \mathbf{P}_h^{k-1}\tilde{\fsigma}_h,\tilde{\ftau}_h\rangle_{L^2\Lambda^{k-1}}&-\langle\tilde{\fomega}_h,\od^{k-1}\tilde{\ftau}_h\rangle_{L^2\Lambda^k}&=0
\\
\langle\tilde{\fvartheta}_h,\tilde{\fmu}_h\rangle_{L^2\Lambda^k}&+\langle \odelta_{k+1,h}\tilde{\fzeta}_h,\tilde{\fmu}_h\rangle_{L^2\Lambda^k}&+\langle \od^{k-1}\tilde{\fsigma}_h,\tilde{\fmu}_h \rangle_{L^2\Lambda^k}&&=\langle \ff,\tilde{\fmu}_h\rangle_{L^2\Lambda^k}
\end{array}.
\right.
\end{equation}

\item Find $(\hat\fvartheta_h,\hat\fsigma_h,\hat\fomega_h)\in \hf_h\Lambda^k\times \fW_h\Lambda^{k-1}\times\fW_h\Lambda^k$, such that, for $(\hat{\fvarsigma}_h, \hat{\ftau}_h, \hat{\fmu}_h)\in \hf_h\Lambda^k\times \fW_h\Lambda^{k-1}\times\fW_h\Lambda^k$
\begin{equation}\label{eq:classdisnq}
\mbox{\bf (AP-III)}\ \ \left\{
\begin{array}{cccl}
&&\langle\hat\fomega_h,\hat{\fvarsigma}_h\rangle_{L^2\Lambda^k}&=0
\\
&\langle\mathbf{P}_h^{k-1}\hat\fsigma_h,\mathbf{P}_h^{k-1}\hat{\ftau}_h\rangle_{L^2\Lambda^{k-1}}&-\langle\hat\fomega_h,\od^{k-1}\hat{\ftau}_h\rangle_{L^2\Lambda^k}&=0
\\
\langle\hat\fvartheta,\hat{\fmu}_h\rangle_{L^2\Lambda^k}&+\langle\od^{k-1}\hat\fsigma_h,\hat{\fmu}_h\rangle_{L^2\Lambda^k}&+\langle\od^k\hat\fomega_h,\od^k\hat{\fmu}_h\rangle_{L^2\Lambda^{k+1}}&=\langle\ff,\mathbf{P}_h^k\hat{\fmu}_h\rangle_{L^2\Lambda^k}
\end{array}
\right.
\end{equation}

\item Find $(\breve\fvartheta_h,\breve\fsigma_h,\breve\fomega_h)\in \hf_h\Lambda^k\times \fW_h\Lambda^{k-1}\times\fW_h\Lambda^k$, such that, for $\breve{\fvarsigma}_h\in\,\hf_h\Lambda^k$, $\breve{\ftau}_h\in\fW_h\Lambda^{k-1}$ and $\breve{\fmu}_h\in\fW_h\Lambda^k$,
\begin{equation}\label{eq:classdis}
\mbox{\bf (AP-IV)}\ \ \left\{
\begin{array}{cccll}
&&\langle\breve\fomega_h,\breve{\fvarsigma}_h\rangle_{L^2\Lambda^k}&=0
\\
&\langle\breve\fsigma_h,\breve{\ftau}_h\rangle_{L^2\Lambda^{k-1}}&-\langle\breve\fomega_h,\od^{k-1}\breve{\ftau}_h\rangle_{L^2\Lambda^k}&=0
\\
\langle\breve\fvartheta,\breve{\fmu}_h\rangle_{L^2\Lambda^k}&+\langle\od^{k-1}\breve\fsigma_h,\breve{\fmu}_h\rangle_{L^2\Lambda^k}&+\langle\od^k\breve\fomega_h,\od^k\breve{\fmu}_h\rangle_{L^2\Lambda^{k+1}}&=\langle\ff,\breve{\fmu}_h\rangle_{L^2\Lambda^k}.
\end{array}
\right.
\end{equation}

\end{enumerate}

\begin{lemma}\label{lem:p2d}
The problem \eqref{eq:equimodelhldis} is well-posed. Further, let $(\fomega_h,\fvartheta_h)$ be the solution of \eqref{eq:modelhldis}. Then
\begin{equation}\label{eq:equality}
\oT_{k,h}\fomega_h=\mathbf{P}_{\rppdd}\left(\begin{array}{c}\overline{\fzeta}_h \\ \overline{\fsigma}_h\end{array}\right),\ \ \ \mbox{and}\ \ \overline{\fomega}_h=\mathbf{P}^k_h\fomega_h. 
\end{equation}
\end{lemma}

\begin{proof}
We are to verify Brezzi's conditions. Firstly, for any $\overline{\fvarsigma}_h\in \hf_h\Lambda^k$ and $\left(\begin{array}{c}\overline{\feta}_h \\ \overline{\ftau}_h\end{array}\right)\in\fW^{*,\rm nc}_{h0}\Lambda^{k+1}\times \fW_h\Lambda^{k-1}$, such that $\langle\overline{\fvarsigma}_h,\overline{\fmu}_h\rangle_{L^2\Lambda^k}=\left\langle\overline{\fmu}_h,\oT^*_{k,h}\left(\begin{array}{c}\overline{\feta}_h \\ \overline{\ftau}_h\end{array}\right)\right\rangle_{L^2\Lambda^k}$, $\forall\,\overline{\fmu}_h\in \mathcal{P}_0\Lambda^k(\mathcal{G}_h)$, it holds that $\overline{\fvarsigma}_h=0$, $\od^{k-1}\overline{\ftau}_h=0$ and $\odelta_{k+1,h}\overline{\feta}_h=0$, and further $\overline{\ftau}_h\in\mathcal{P}_0\Lambda^{k-1}(\mathcal{G}_h)$ and $\overline{\feta}_h\in\mathcal{P}_0\Lambda^{k+1}(\mathcal{G}_h)$. Therefore, $\left\langle\mathbf{P}_{\rppdd}\left(\begin{array}{c}\overline{\feta}_h \\ \overline{\ftau}_h\end{array}\right),\left(\begin{array}{c}\overline{\feta}_h \\ \overline{\ftau}_h\end{array}\right)\right\rangle=\left\langle\left(\begin{array}{c}\overline{\feta}_h \\ \overline{\ftau}_h\end{array}\right),\left(\begin{array}{c}\overline{\feta}_h \\ \overline{\ftau}_h\end{array}\right)\right\rangle$, and the coercivity follows. On the other hand, given any $\overline{\fmu}_h\in\mathcal{P}_0\Lambda^k(\mathcal{G}_h)$, by the Hodge decomposition of $\mathcal{P}_0\Lambda^k(\mathcal{G}_h)$ and the Poincar\'e inequalities on $\fW_h\Lambda^{k-1}$ and $\fW^{*,\rm nc}_{h0}\Lambda^{k+1}$, there exists a $\overline{\fvarsigma}_h\in \hf_h\Lambda^k$, a $\overline{\ftau}_h \in\fW_h\Lambda^{k-1}$ and an $\overline{\feta}_h\in\fW^{*,\rm nc}_{h0}\Lambda^{k+1}$, such that $\overline{\fmu}_h=\overline{\fvarsigma}_h+\od^{k-1}\overline{\ftau}_h+\odelta_{k+1,h}\overline{\feta}_h$, and $\|\overline{\ftau}_h\|_{L^2\Lambda^{k-1}}+\|\overline{\feta}_h\|_{L^2\Lambda^{k+1}}\leqslant C\|\overline{\fmu}_h\|_{L^2\Lambda^k}$. Then Brezzi's conditions for \eqref{eq:equimodelhldis} are verified. For any $\ff\in L^2\Lambda^k$, the system \eqref{eq:equimodelhldis} admits a unique solution. 

As $(\fomega_h,\fvartheta_h)$ is the solution of \eqref{eq:modelhldis}, with $\left(\begin{array}{c}\fzeta_h \\ \fsigma_h\end{array}\right)\in\fW^{*,\rm nc}_{h0}\Lambda^{k+1}\times \fW_h\Lambda^{k-1}$, it holds that 
\begin{multline}
\langle\fvartheta_h,\fmu_h\rangle_{L^2\Lambda^k}+
\left\langle\oT_{k,h}\fomega_h,\oT_{k,h}\fmu_h\right\rangle
\\
+\left\langle\oT_{k,h}\fmu_h,\left(\begin{array}{c}\fzeta_h \\ \fsigma_h\end{array}\right)\right\rangle-\left\langle\fmu_h,\oT^*_{k,h}\left(\begin{array}{c}\fzeta_h \\ \fsigma_h\end{array}\right)\right\rangle_{L^2\Lambda^k}=\langle\mathbf{P}^k_h\ff,\fmu_h\rangle_{L^2\Lambda^k}, \forall\,\fmu_h\in \ppdd.
\end{multline}
Then, for $\fmu_h\in\mathcal{P}_0\Lambda^k(\mathcal{G}_h)$,
$$
\left\langle\fvartheta_h,\fmu_h\right\rangle_{L^2\Lambda^k}-\left\langle \fmu_h,\oT^*_{k,h}\left(\begin{array}{c}\fzeta_h \\ \fsigma_h\end{array}\right)\right\rangle=\langle\ff,\fmu_h\rangle_{L^2\Lambda^k}.
$$
Namely 
$$
\oT^*_{k,h}\left(\begin{array}{c}\fzeta_h \\ \fsigma_h\end{array}\right)=\mathbf{P}^k_h\ff-\mathbf{P}_{\hf_h}\ff,\ \mathbf{P}_{\hf_h}\ \mbox{being\ the\ projection\ to}\ \hf_h.
$$
It follows then
$$
\oT_{k,h}\fomega_h=\mathbf{P}_{\rppdd}\left(\begin{array}{c}\fzeta_h \\ \fsigma_h\end{array}\right).
$$
Noting that
$$
\left\langle\oT_{k,h}\fomega_h,\left(\begin{array}{c}\feta_h \\ \ftau_h\end{array}\right)\right\rangle=\left\langle\fomega_h,\oT^*_{k,h}\left(\begin{array}{c}\feta_h \\ \ftau_h\end{array}\right)\right\rangle_{L^2\Lambda^k},\ \forall\,\left(\begin{array}{c}\feta_h \\ \ftau_h\end{array}\right)\in\fW^{*,\rm nc}_{h0}\Lambda^{k+1}\times \fW_h\Lambda^{k-1},
$$
we have further
$$
\left\langle\mathbf{P}_{\rppdd}\left(\begin{array}{c}\fzeta_h \\ \fsigma_h\end{array}\right),\left(\begin{array}{c}\feta_h \\ \ftau_h\end{array}\right)\right\rangle = \left\langle\fomega_h,\oT^*_{k,h}\left(\begin{array}{c}\feta_h \\ \ftau_h\end{array}\right)\right\rangle_{L^2\Lambda^k},\ \ \ \forall\,\left(\begin{array}{c}\feta_h \\ \ftau_h\end{array}\right)\in\fW^{*,\rm nc}_{h0}\Lambda^{k+1}\times \fW_h\Lambda^{k-1}.
$$
Therefore, $(\mathbf{P}_h^k\fomega_h,\fvartheta_h,\fsigma_h,\fzeta_h)$ solves \eqref{eq:equimodelhldis}, and \eqref{eq:equality} follows. The proof is completed. 
\end{proof}

\begin{remark}\label{rem:p2d}
It follows easily that 
$$
\|\fomega_h-\overline{\fomega}_h\|_{L^2\Lambda^k}+ \|\od^k_h\fomega_h-\overline{\fzeta}_h\|_{L^2\Lambda^{k+1}}+ \|\odelta_{k,h}\fomega_h-\overline{\fsigma}_h\|_{L^2\Lambda^{k-1}}\leqslant Ch\|\ff\|_{L^2\Lambda^k}.
$$
\end{remark}

By the same virtue, we can prove the lemmas below. 

\begin{lemma}\label{lem:d12d2}
The system \eqref{eq:mixedhodgedis} is well-posed. Further, let $(\overline{\fomega}_h,\overline{\fzeta}_h,\overline{\fsigma}_h,\overline{\fvartheta}_h)$ and $(\tilde{\fomega}_h,\tilde{\fzeta}_h,\tilde{\fsigma}_h,\tilde{\fvartheta}_h)$ be the solutions of \eqref{eq:equimodelhldis} and \eqref{eq:mixedhodgedis}, respectively. Then $\overline{\fvartheta}_h=\tilde{\fvartheta}_h$, and 
\begin{equation}
\|\overline{\fomega}_h-\tilde{\fomega}_h\|_{L^2\Lambda^k}+\|\overline{\fzeta}_h-\tilde{\fzeta}_h\|_{\odelta_{k+1,h}}+\|\overline{\fsigma}_h-\tilde{\fsigma}_h\|_{\od^{k-1}}\leqslant Ch\|\ff\|_{L^2\Lambda^k}. 
\end{equation}
\end{lemma}

\begin{lemma}\label{lem:d22d3}
The system \eqref{eq:classdisnq} is well-posed. Further, let $(\tilde{\fomega}_h,\tilde{\fzeta}_h,\tilde{\fsigma}_h,\tilde{\fvartheta}_h)$ and $(\hat\fvartheta_h,\hat\fsigma_h,\hat\fomega_h)$ be the solutions of \eqref{eq:mixedhodgedis} and \eqref{eq:classdisnq}, respectively. Then 
$$
\hat{\fvartheta}_h=\tilde{\fvartheta}_h,\ \ \tilde{\fomega}_h=\mathbf{P}^k_h\hat{\fomega}_h,\ \ \tilde{\fsigma}_h=\hat{\fsigma}_h,\ \ \mbox{and}\ \od^k\hat{\fomega}_h=\mathbf{P}^{k+1}_h\tilde{\fzeta}_h.
$$
\end{lemma}

\begin{lemma}\label{lem:d32d4}
The system \eqref{eq:classdis} is well-posed. Let $(\hat\fvartheta_h,\hat\fsigma_h,\hat\fomega_h)$ and $(\breve\fvartheta_h,\breve\fsigma_h,\breve\fomega_h)$ be the solutions of \eqref{eq:classdisnq} and \eqref{eq:classdis}, respectively. Then $\hat\fvartheta_h=\breve\fvartheta_h$, and
\begin{equation}
\|\hat\fsigma_h-\breve\fsigma_h\|_{\od^{k-1}}+\|\hat\fomega_h-\breve\fomega_h\|_{\od^{k+1}}\leqslant Ch\|\ff\|_{L^2\Lambda^k}.
\end{equation}
\end{lemma}

\begin{proposition}\label{prop:primal-mixed-chain}
Let $\ff\in L^2\Lambda^k(\Omega)$. Let $(\fomega_h,\fvartheta_h)$ be the solution of \eqref{eq:modelhldis}. Let $(\breve\fvartheta_h,\breve\fsigma_h,\breve\fomega_h)$ be the solution of  \eqref{eq:classdis}. Then there exists a constant $C>0$, independent of $h$, such that
\begin{equation}\label{eq:primal-mixed-chain}
\|\fomega_h-\breve\fomega_h\|_{L^2\Lambda^k}
+\left\|\oT_{k,h}\fomega_h-\left(\begin{array}{c}\od^k_h\breve\fomega_h \\ \breve\fsigma_h\end{array}\right)\right\|_{L^2\Lambda^{k+1}\times L^2\Lambda^{k-1}}
\leqslant C h\|\ff\|_{L^2\Lambda^k}.
\end{equation}
\end{proposition}

\begin{proof}
Let $(\overline{\fomega}_h,\overline{\fzeta}_h,\overline{\fsigma}_h,\overline{\fvartheta}_h)$, $(\tilde{\fomega}_h,\tilde{\fzeta}_h,\tilde{\fsigma}_h,\tilde{\fvartheta}_h)$ and $(\hat\fvartheta_h,\hat\fsigma_h,\hat\fomega_h)$ be the solutions of \eqref{eq:equimodelhldis}, \eqref{eq:mixedhodgedis} and \eqref{eq:classdisnq}, respectively. By Remark~\ref{rem:p2d} and Lemma~\ref{lem:d12d2}, $\overline{\fvartheta}_h=\tilde{\fvartheta}_h$ and
\begin{equation*}
\|\overline{\fomega}_h-\tilde{\fomega}_h\|_{L^2\Lambda^k}
+\|\overline{\fzeta}_h-\tilde{\fzeta}_h\|_{\odelta_{k+1,h}}
+\|\overline{\fsigma}_h-\tilde{\fsigma}_h\|_{\od^{k-1}}
\leqslant C h\|\ff\|_{L^2\Lambda^k}.
\end{equation*}
By Lemma~\ref{lem:d22d3}, $\hat\fvartheta_h=\tilde{\fvartheta}_h$, $\tilde{\fomega}_h=\mathbf{P}^k_h\hat{\fomega}_h$, $\tilde{\fsigma}_h=\hat{\fsigma}_h$, and $\od^k_h\hat{\fomega}_h=\mathbf{P}^{k+1}_h\tilde{\fzeta}_h$. By Lemma~\ref{lem:d32d4}, $\hat\fvartheta_h=\breve\fvartheta_h$ and
\begin{equation*}
\|\hat\fsigma_h-\breve\fsigma_h\|_{\od^{k-1}}+\|\hat\fomega_h-\breve\fomega_h\|_{\od^k}
\leqslant C h\|\ff\|_{L^2\Lambda^k}.
\end{equation*}
The estimate \eqref{eq:primal-mixed-chain} follows by summing these bounds and applying the triangle inequality.
\end{proof}

\begin{remark}\label{rem:primal-mixed-onefield}
Let $\fomega_h$ solve \eqref{eq:modelhldisonefield}. Since \eqref{eq:modelhldisonefield} and \eqref{eq:modelhldis} are equivalent, the same bound \eqref{eq:primal-mixed-chain} holds if $\fomega_h$ there is replaced by the solution of \eqref{eq:modelhldisonefield}.
\end{remark}

\subsection{Proof of Theorem \ref{thm:convprimsch}}
Corresponding to \eqref{eq:classdis}, an equivalent mixed formulation of the Hodge-Laplace problem reads: find $(\breve\fvartheta,\breve\fsigma,\breve\fomega)\in \hf\Lambda^k\times H\Lambda^{k-1}\times H\Lambda^k$ such that 
\begin{equation}\label{eq:classicalmix}
\left\{
\begin{array}{cccll}
&&\langle\breve\fomega,\fvarsigma\rangle_{L^2\Lambda^k}&=0&\forall\,\fvarsigma\in\,\hf\Lambda^k
\\
&\langle\breve\fsigma,\ftau\rangle_{L^2\Lambda^{k-1}}&-\langle\breve\fomega,\od^{k-1}\ftau\rangle_{L^2\Lambda^k}&=0&\forall\,\ftau\in H\Lambda^{k-1}
\\
\langle\breve\fvartheta,\fmu\rangle_{L^2\Lambda^k}&+\langle\od^{k-1}\breve\fsigma,\fmu_h\rangle_{L^2\Lambda^k}&+\langle\od^k\breve\fomega,\od^k\fmu\rangle_{L^2\Lambda^{k+1}}&=\langle\ff,\fmu\rangle_{L^2\Lambda^k}&\forall\,\fmu\in H\Lambda^k,
\end{array}
\right.
\end{equation}
Note that \eqref{eq:classicalmix} is the classical mixed formulation of the Hodge-Laplace problem and \eqref{eq:classdis} is the corresponding discretization of \eqref{eq:classicalmix}. 

\begin{lemma}\label{lem:classicalmix}\cite[(7.17), (7.30) and Theorem 7.10]{Arnold.D;Falk.R;Winther.R2006acta}
Let $(\breve\fvartheta,\breve\fsigma,\breve\fomega)$ and $(\breve\fvartheta_h,\breve\fsigma_h,\breve\fomega_h)$ be the solutions of \eqref{eq:classicalmix} and \eqref{eq:classdis}, respectively. Then, 
\begin{multline*}
\|\breve\fsigma-\breve\fsigma_h\|_{\od^{k-1}}+\|\breve\fomega-\breve\fomega_h\|_{\od^k}+\|\breve\fvartheta-\breve\fvartheta_h\|_{L^2\Lambda^k}
\\
\leqslant C(\inf_{\ftau_h\in\fW_h\Lambda^{k-1}}\|\breve\fsigma-\ftau_h\|_{\od^{k-1}}+\inf_{\fmu_h\in\fW_h\Lambda^k}\|\breve\fomega-\fmu_h\|_{\od^k}+\inf_{\fvarsigma_h\in\hf_h\Lambda^k}\|\breve\fvartheta-\fvarsigma_h\|_{L^2\Lambda^k}+h\|\ff\|_{L^2\Lambda^k});
\end{multline*}
and further, with $\Omega$ being $s$-regular,
$$
\|\breve\fsigma-\breve\fsigma_h\|_{L^2\Lambda^k}+\|\breve\fomega-\breve\fomega_h\|_{\od^k}+\|\breve\fvartheta-\breve\fvartheta_h\|_{L^2\Lambda^k}\leqslant Ch^s\|\ff\|_{L^2\Lambda^k}.
$$
\end{lemma}

\paragraph{\bf Proof of Theorem \ref{thm:convprimsch}}
Let $\fomega$ and $\fomega_h$ be the solutions of \eqref{eq:modelhlori} and \eqref{eq:modelhldisonefield}, respectively. Let $(\breve\fvartheta,\breve\fsigma,\breve\fomega)$ and $(\breve\fvartheta_h,\breve\fsigma_h,\breve\fomega_h)$ be the solutions of \eqref{eq:classicalmix} and \eqref{eq:classdis}, respectively. Then $\fomega=\breve\fomega$. By Remark~\ref{rem:primal-mixed-onefield} and Proposition~\ref{prop:primal-mixed-chain},
\begin{equation}\label{eq:dispdae}
\|\fomega_h-\breve{\fomega}_h\|_{L^2\Lambda^k}
+\left\|\oT_{k,h}\fomega_h-\left(\begin{array}{c}\od^k_h\breve{\fomega}_h \\ \breve{\fsigma}_h\end{array}\right)\right\|_{L^2\Lambda^{k+1}\times L^2\Lambda^{k-1}}
\leqslant C h\|\ff\|_{L^2\Lambda^k}.
\end{equation}
Therefore,
\begin{equation*}
\|\fomega-\fomega_h\|_{L^2\Lambda^k}
\leqslant \|\fomega-\breve{\fomega}_h\|_{L^2\Lambda^k}+\|\breve{\fomega}_h-\fomega_h\|_{L^2\Lambda^k}
\leqslant \|\fomega-\breve{\fomega}_h\|_{L^2\Lambda^k}+Ch\|\ff\|_{L^2\Lambda^k}.
\end{equation*}
Moreover,
\begin{multline*}
\|\od^k\fomega-\od^k_h\fomega_h\|_{L^2\Lambda^{k+1}}+
\|\odelta_k\fomega-\odelta_{k,h}\fomega_h\|_{L^2\Lambda^{k-1}}
\leqslant 
\|\od^k\breve\fomega-\od^k\breve{\fomega}_h\|_{L^2\Lambda^{k+1}}
+
\|\breve{\fsigma}-\breve{\fsigma}_h\|_{L^2\Lambda^{k-1}}
\\
+
\|\od^k\breve{\fomega}_h-\od^k_h\hat{\fomega}_h\|_{L^2\Lambda^{k+1}}
+
\|\breve{\fsigma}_h-\hat{\fsigma}_h\|_{L^2\Lambda^{k-1}}
+
\|\od^k_h\hat{\fomega}_h-\tilde{\fzeta}_h\|_{L^2\Lambda^{k+1}}
+
\|\hat{\fsigma}_h-\tilde{\fsigma}_h\|_{L^2\Lambda^{k-1}}
\\
+
\|\tilde{\fzeta}_h-\overline{\fzeta}_h\|_{L^2\Lambda^{k+1}}
+
\|\tilde{\fsigma}_h-\overline{\fsigma}_h\|_{L^2\Lambda^{k-1}}
+
\|\overline{\fzeta}_h-\od^k_h\fomega_h\|_{L^2\Lambda^{k+1}}
+
\|\overline{\fsigma}_h-\odelta_{k,h}\fomega_h\|_{L^2\Lambda^{k-1}}
\\
\leqslant 
\|\od^k\breve\fomega-\od^k\breve{\fomega}_h\|_{L^2\Lambda^{k+1}}
+
\|\breve{\fsigma}-\breve{\fsigma}_h\|_{L^2\Lambda^{k-1}}
+Ch\|\ff\|_{L^2\Lambda^k},
\end{multline*}
where the last inequality follows from Remark~\ref{rem:p2d}, Lemmas~\ref{lem:d12d2}, \ref{lem:d22d3}, and~\ref{lem:d32d4}, and the identity $\fomega=\breve\fomega$ together with the link between \eqref{eq:modelhlori} and \eqref{eq:classicalmix}. The proof is completed by Lemma~\ref{lem:classicalmix}. \qed

\begin{remark}
The continuous primal and mixed Hodge--Laplace formulations are equivalent; therefore, an estimation of \eqref{eq:dispdae}-type, not necessarily identical, can be a necessary condition so that $\fomega_h$ converges to $\fomega$. This is why this indirect analysis routine is utilized here.
\end{remark}

%
%

%
%

%
%
%
\section{Numerical experiments}
\label{sec:numer}

\subsection{Overview}\label{sec:numer-overview}
In this section, we use the two- and three- dimensional $H(\dv)\cap H(\curl)$ problem to test the spaces $\fV^{\od\cap\mathring{\odelta}}_h\Lambda^k$, the specific formulations of which are given in Subsection~\ref{sec:dim-realizations}. To particularly illustrate the performance with respect to the complicated topology of the domain, we focus ourselves on eigenvalue problems and compare with standard mixed FEEC discretizations. Particularly, 
Subsection~\ref{sec:numer2d} treats $n=2$ and Subsections~\ref{sec:numer3d} treats $n=3$ (both including domains with nontrivial topology).

To keep a reasonable length of the paper, we report the manufactured-solution boundary-value tests in \AppendixRef{app:bvp-results}, and counterexamples using vector Lagrange and vector Crouzeix--Raviart elements are given in \AppendicesRef{app:lagrange-eigenvalues}{app:cr-eigenvalues}.

\subsection{Two-dimensional tests}\label{sec:numer2d}

For $\Omega\subset\mathbb{R}^2$ a polygon, we consider the following eigenvalue problem: find $(\lambda,\omega\not\equiv 0) \in \mathbb{R}\times H(\rot,\Omega) \cap H_0(\dv,\Omega)$ such that
\begin{equation}\label{primal2D}
(\dv\omega,\dv\tau) + (\rot\omega,\rot\tau) = \lambda (\omega,\tau), \forall \tau \in H(\rot,\Omega) \cap H_0(\dv,\Omega).
\end{equation}
The finite element space $\fV^{\rm r\mathring{d}}_h$ is defined in~\eqref{eq:deffems2d} (Subsection~\ref{sec:dim-realizations}); the cell-wise and locally supported global basis functions can be found in \AppendixRef{sec:app:basis-2d}. The primal discretization of \eqref{primal2D} is to find \((\lambda_h,\omega_h\not\equiv 0)\in \mathbb R\times \fV^{\rm r\mathring d}_h\), such that
\begin{equation}
(\dv_h\omega_h,\dv_h\tau_h)
+
(\rot_h\omega_h,\rot_h\tau_h)
=
\lambda_h(\omega_h,\tau_h),
\qquad
\forall\,\tau_h\in \fV^{\rm r\mathring d}_h .
\label{eq:2d-primal-discrete}
\end{equation}

To verify the computation results, we use the standard mixed formulation for reference: find $(\lambda,\sigma,u\not\equiv 0) \in \mathbb{R}\times H^1(\Omega)\times H(\rot,\Omega)$ such that
\begin{equation}\label{mix2D}
	\left\{
	\begin{array}{cccll}
		(\sigma,\tau)&-(u,\grad \tau)&=0&\forall\,\tau\in H^1(\Omega)
		\\
		(\grad\sigma,v)&+(\rot u ,\rot v)&=\lambda (u,v)&\forall\,v\in H(\rot,\Omega).
	\end{array}
	\right.
\end{equation}
Its classical discretization is to find
\[
(\lambda_h,\sigma_h,u_h\not\equiv 0)
\in
\mathbb R\times \mathbb{V}^1_h\times \mathbb{V}^{\rm Ned}_h,
\]
such that
\begin{equation}\label{eq:2d-mixed-discrete}
	\left\{
	\begin{aligned}
		(\sigma_h,\tau_h)-(u_h,\grad\tau_h)&=0,
		&& \forall\,\tau_h\in \mathbb{V}^1_h,\\
		(\grad\sigma_h,v_h)+(\rot u_h,\rot v_h)
		&=\lambda_h(u_h,v_h),
		&& \forall\,v_h\in \mathbb{V}^{\rm Ned}_h,
	\end{aligned}
	\right.
\end{equation}
where $\mathbb{V}^{\rm Ned}_h$ is the lowest-degree Nedelec element space for $H(\rot,\Omega)$. 

We compute the ten smallest eigenvalues using the primal formulation \eqref{eq:2d-primal-discrete} with the $\fV^{\rm r\mathring{d}}_h$ element and the mixed formulation \eqref{eq:2d-mixed-discrete} with the classical mixed element on three domains, namely
\begin{itemize}
\item Square domain: \(\Omega_{\rm S}=[0,1]^2\).
\item L-shaped domain: \(\Omega_{\rm L}=[0,1]^2\setminus\bigl((0.5,1)\times(0,0.5)\bigr)\); see Figure \ref{fig:Lshape}.
\item Square domain with an interior hole: \(\Omega_{\rm H}=[0,1]^2\setminus[0.5,0.75]^2\); see Figure \ref{fig:hole}. 
\end{itemize}

\begin{figure}[htbp]
	\centering
	
	\begin{minipage}[t]{0.49\textwidth}
		\centering
		\includegraphics[height=4cm]{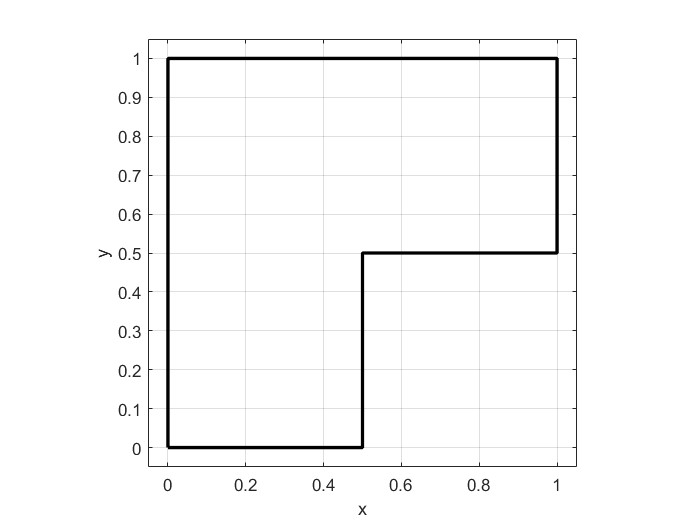}
		\caption{Computational domain for the L-shaped example.}
		\label{fig:Lshape}
	\end{minipage}
	\hfill
	\begin{minipage}[t]{0.49\textwidth}
		\centering
		\includegraphics[height=4cm]{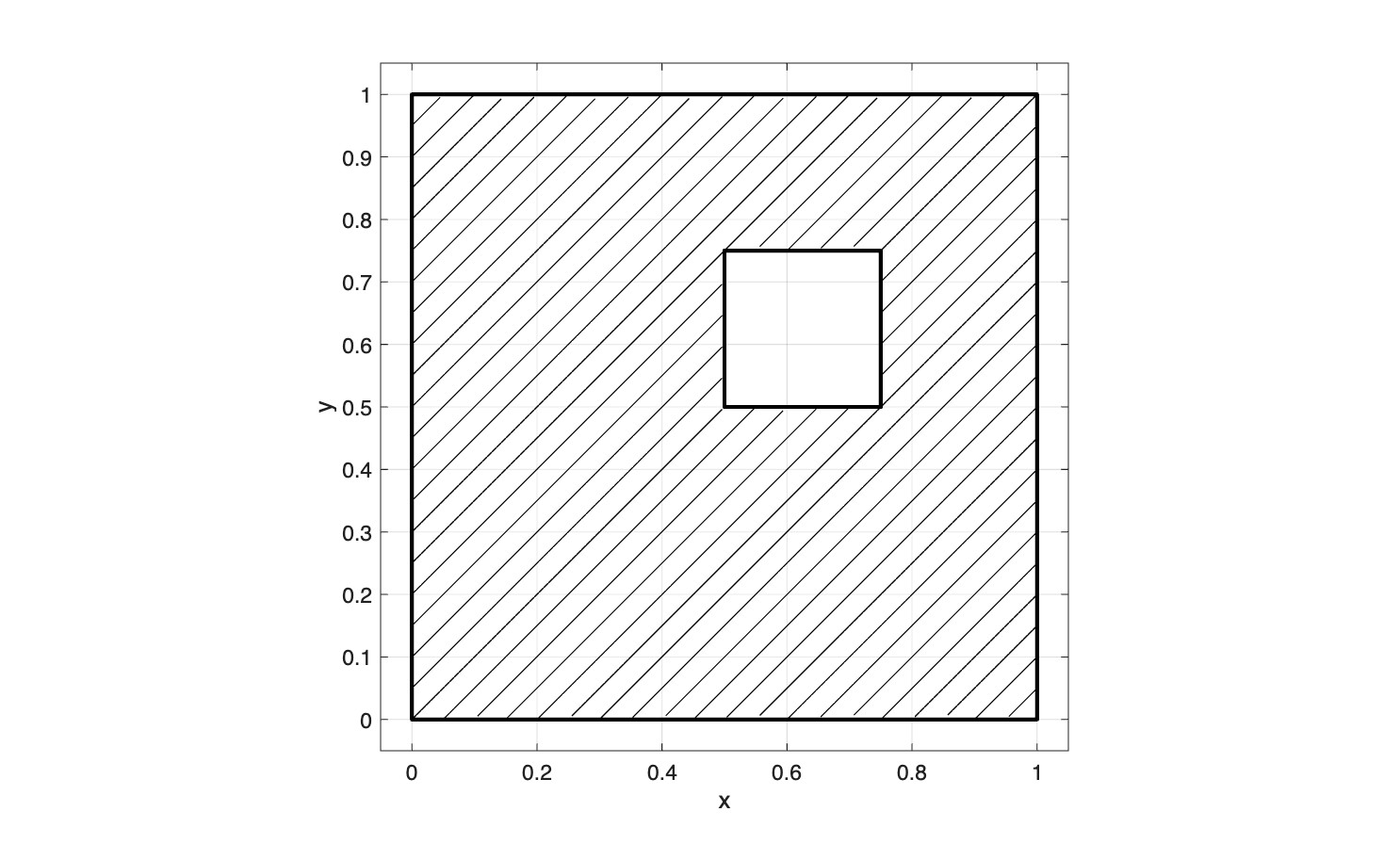}
		\caption{Computational domain for the square example with an interior hole.}
		\label{fig:hole}
	\end{minipage}
	
\end{figure}

Tables~\ref{tab:sq-primal-new}--\ref{tab:hole-mix} record the computed eigenvalues, where $L$ denotes the mesh level. The eigenvalues computed by \eqref{eq:2d-primal-discrete} agree closely with those computed by \eqref{eq:2d-mixed-discrete}. In particular, on the third domain \(\Omega_{\rm H}\), the first Betti number $\mathtt{b}_1$ is~1, and 0 appears as a simple eigenvalue for both \eqref{eq:2d-primal-discrete} and \eqref{eq:2d-mixed-discrete}. 

\begin{table}[htbp]
\centering
\small
\setlength{\tabcolsep}{3pt}
\begin{tabular}{|c|*{10}{c|}}
\hline
$L$ & $\lambda_h^1$ & $\lambda_h^2$ & $\lambda_h^3$ & $\lambda_h^4$ & $\lambda_h^5$ & $\lambda_h^6$ & $\lambda_h^7$ & $\lambda_h^8$ & $\lambda_h^9$ & $\lambda_h^{10}$ \\
\hline
1 & 10.188 & 10.188 & 18.982 & 20.516 & 44.578 & 44.578 & 45.776 & 45.776 & 55.394 & 55.394 \\ \hline
2 &  9.949 &  9.949 & 19.549 & 19.930 & 40.753 & 40.753 & 48.447 & 48.447 & 50.837 & 50.837 \\ \hline
3 &  9.889 &  9.889 & 19.692 & 19.787 & 39.796 & 39.796 & 49.122 & 49.122 & 49.718 & 49.718 \\ \hline
4 &  9.875 &  9.875 & 19.727 & 19.751 & 39.558 & 39.558 & 49.292 & 49.292 & 49.440 & 49.440 \\
\hline
\end{tabular}
\caption{Computed eigenvalues on \(\Omega_{\rm S}=[0,1]^2\), using \eqref{eq:2d-primal-discrete}.}
\label{tab:sq-primal-new}
\end{table}

\begin{table}[htbp]
\centering
\small
\setlength{\tabcolsep}{3pt}
\begin{tabular}{|c|*{10}{c|}}
\hline
$L$ & $\lambda_h^1$ & $\lambda_h^2$ & $\lambda_h^3$ & $\lambda_h^4$ & $\lambda_h^5$ & $\lambda_h^6$ & $\lambda_h^7$ & $\lambda_h^8$ & $\lambda_h^9$ & $\lambda_h^{10}$ \\
\hline
1 & 10.211 & 10.211 & 19.398 & 20.608 & 45.012 & 45.012 & 48.291 & 48.291 & 56.070 & 56.070 \\ \hline
2 &  9.954 &  9.954 & 19.655 & 19.952 & 40.843 & 40.843 & 49.100 & 49.100 & 50.977 & 50.977 \\ \hline
3 &  9.891 &  9.891 & 19.718 & 19.792 & 39.817 & 39.817 & 49.287 & 49.287 & 49.751 & 49.751 \\ \hline
4 &  9.875 &  9.875 & 19.734 & 19.752 & 39.563 & 39.563 & 49.333 & 49.333 & 49.449 & 49.449 \\
\hline
\end{tabular}
\caption{Computed eigenvalues on \(\Omega_{\rm S}=[0,1]^2\), using \eqref{eq:2d-mixed-discrete}.}
\label{tab:sq-mixed}
\end{table}

\begin{table}[htbp]
\centering
\small
\setlength{\tabcolsep}{3pt}
\begin{tabular}{|c|*{10}{c|}}
\hline
$L$ & $\lambda_h^1$ & $\lambda_h^2$ & $\lambda_h^3$ & $\lambda_h^4$ & $\lambda_h^5$ & $\lambda_h^6$ & $\lambda_h^7$ & $\lambda_h^8$ & $\lambda_h^9$ & $\lambda_h^{10}$ \\
\hline
1 & 6.412 & 14.621 & 34.269 & 44.578 & 44.578 & 51.021 & 55.628 & 58.114 & 67.036 & 91.817 \\ \hline
2 & 6.076 & 14.264 & 37.143 & 40.753 & 40.753 & 46.927 & 52.346 & 59.443 & 75.928 & 82.064 \\ \hline
3 & 5.966 & 14.169 & 38.075 & 39.796 & 39.796 & 45.900 & 50.863 & 60.444 & 78.197 & 79.722 \\ \hline
4 & 5.926 & 14.144 & 38.387 & 39.558 & 39.558 & 45.643 & 50.458 & 60.701 & 78.767 & 79.147 \\
\hline
\end{tabular}
\caption{Computed eigenvalues on $\Omega_{\rm L}$ (Figure~\ref{fig:Lshape}) using \eqref{eq:2d-primal-discrete}.}
\label{tab:L-primal-new}
\end{table}

\begin{table}[htbp]
\centering
\small
\setlength{\tabcolsep}{3pt}
\begin{tabular}{|c|*{10}{c|}}
\hline
$L$ & $\lambda_h^1$ & $\lambda_h^2$ & $\lambda_h^3$ & $\lambda_h^4$ & $\lambda_h^5$ & $\lambda_h^6$ & $\lambda_h^7$ & $\lambda_h^8$ & $\lambda_h^9$ & $\lambda_h^{10}$ \\
\hline
1 & 6.421 & 14.667 & 35.592 & 45.012 & 45.012 & 51.592 & 58.859 & 59.559 & 73.337 & 93.723 \\ \hline
2 & 6.078 & 14.275 & 37.520 & 40.843 & 40.843 & 47.047 & 52.495 & 60.436 & 77.594 & 82.432 \\ \hline
3 & 5.966 & 14.172 & 38.174 & 39.817 & 39.817 & 45.928 & 50.898 & 60.694 & 78.618 & 79.808 \\ \hline
4 & 5.926 & 14.145 & 38.412 & 39.563 & 39.563 & 45.651 & 50.467 & 60.764 & 78.872 & 79.169 \\
\hline
\end{tabular}
\caption{Computed eigenvalues on $\Omega_{\rm L}$ (Figure~\ref{fig:Lshape}) using \eqref{eq:2d-mixed-discrete}.}
\label{tab:L-mixed}
\end{table}

\begin{table}[htbp]
\centering
\small
\setlength{\tabcolsep}{3pt}
\begin{tabular}{|c|*{10}{c|}}
\hline
$L$ & $\lambda_h^1$ & $\lambda_h^2$ & $\lambda_h^3$ & $\lambda_h^4$ & $\lambda_h^5$ & $\lambda_h^6$ & $\lambda_h^7$ & $\lambda_h^8$ & $\lambda_h^9$ & $\lambda_h^{10}$ \\
\hline
1 & 0.000 & 8.707 & 8.931 & 19.531 & 33.867 & 41.637 & 45.949 & 53.379 & 53.732 & 55.952 \\ \hline
2 & 0.000 & 8.216 & 8.406 & 18.918 & 36.748 & 40.149 & 42.097 & 48.935 & 53.802 & 72.364 \\ \hline
3 & 0.000 & 8.049 & 8.225 & 18.732 & 35.954 & 37.763 & 42.731 & 49.967 & 52.093 & 60.796 \\ \hline
4 & 0.000 & 7.989 & 8.160 & 18.676 & 34.948 & 38.091 & 42.453 & 46.388 & 49.654 & 59.727 \\
\hline
\end{tabular}
\caption{Computed eigenvalues on $\Omega_{\rm H}$ (Figure~\ref{fig:hole}) using \eqref{eq:2d-primal-discrete}. The method can correctly capture the zero eigenvalue associated with the
topology of the domain.}
\label{tab:hole-primal-new}
\end{table}

\begin{table}[htbp]
\centering
\small
\setlength{\tabcolsep}{3pt}
\begin{tabular}{|c|*{10}{c|}}
\hline
$L$ & $\lambda_h^1$ & $\lambda_h^2$ & $\lambda_h^3$ & $\lambda_h^4$ & $\lambda_h^5$ & $\lambda_h^6$ & $\lambda_h^7$ & $\lambda_h^8$ & $\lambda_h^9$ & $\lambda_h^{10}$ \\
\hline
1 & 0.000 & 8.724 & 8.949 & 19.614 & 35.146 & 42.014 & 46.384 & 54.127 & 56.354 & 56.373 \\ \hline
2 & 0.000 & 8.219 & 8.410 & 18.937 & 36.788 & 37.128 & 41.756 & 48.474 & 51.643 & 58.173 \\ \hline
3 & 0.000 & 8.050 & 8.226 & 18.737 & 35.375 & 37.838 & 40.569 & 47.091 & 50.374 & 58.802 \\ \hline
4 & 0.000 & 7.989 & 8.160 & 18.677 & 34.951 & 38.105 & 40.255 & 46.722 & 50.049 & 59.025 \\
\hline
\end{tabular}
\caption{Computed eigenvalues on $\Omega_{\rm H}$ (Figure~\ref{fig:hole}) using \eqref{eq:2d-mixed-discrete}.}
\label{tab:hole-mix}
\end{table}

Overall, the proposed element gives eigenvalue approximations that are
consistent with those obtained from the mixed formulation on all tested
two-dimensional domains. In particular, the zero eigenvalue associated with the
topology of the multiply connected domain is correctly captured by both
methods.The present results therefore provide evidence that the proposed element is
well suited to the primal eigenvalue formulation considered here. The
agreement is observed not only on the simply connected square domain, but
also on the nonconvex L-shaped domain and on the multiply connected domain
with an interior hole. This is noteworthy in view of the known difficulty
of designing nonconforming primal discretizations for curl--curl type problems. For example, Brenner-Sun-Cui \cite{Brenner.S;Sung.L;Cui.J2008} observed that a naive
CR-type weakly continuous \(P_1\) discretization of the two-dimensional
curl--curl problem fails to converge unless additional
consistency terms controlling the tangential and normal jumps are included.

%
%

%
%

%
%
%
\subsection{Three-dimensional tests}\label{sec:numer3d}

Let $\Omega \subset \mathbb{R}^3$ be a polyhedral domain. 

\subsubsection{The case $H(\dv,\Omega)\cap H_0(\curl,\Omega)$}
In the first case, we consider eigenvalue problem associated with the space $H(\dv,\Omega)\cap H_0(\curl,\Omega)$. The primal weak formulation is to find $(\lambda,\fomega)\in \mathbb{R}\times H(\dv,\Omega)\cap H_0(\curl,\Omega)$ such that
\begin{equation}
(\dv\omega,\dv\tau) + (\curl\omega,\curl\tau) = \lambda(\omega,\tau),\forall\tau\in H(\dv,\Omega)\cap H_0(\curl,\Omega).
\label{eq:3d-primal-first}
\end{equation}
For $H(\dv,\Omega)\cap H_0(\curl,\Omega)$ the finite element space is $\fV^{\rm \mathring{c} \rm d}_h$ in~\eqref{eq:deffems3d} (see \AppendixRef{sec:app:basis-3d} for the cell-wise and locally supported global basis functions). The discrete primal weak formulation is to find
\[
(\lambda_h,\omega_h)\in 
\mathbb R\times \fV^{\mathring{\mathrm c}{\rm d}}_h,
\qquad \omega_h\ne0,
\]
such that
\begin{equation}
	(\dv_h\omega_h,\dv_h\tau_h)
	+
	(\curl_h\omega_h,\curl_h\tau_h)
	=
	\lambda_h(\omega_h,\tau_h),
	\qquad
	\forall\,\tau_h\in \fV^{\mathring{\mathrm c}{\rm d}}_h .
	\label{eq:3d-primal-first-discrete}
\end{equation}

The corresponding mixed formulation is to find $(\lambda,\sigma,u)\in \mathbb{R}\times H(\curl,\Omega) \times H(\dv,\Omega)$ such that
\begin{equation}\label{eq:3d-mixed}
\left\{
\begin{array}{cclll}
(\sigma,\tau)&-(u,\curl\tau)&=0&\forall\,\tau\in H(\curl,\Omega)
\\
(\curl\sigma,v)&+(\dv u ,\dv v)&=\lambda (u,v)&\forall\,v\in H(\dv,\Omega)
\end{array}
\right.
\end{equation}

The discrete mixed formulation is to find
\[
(\lambda_h,\sigma_h,u_h)\in
\mathbb R\times\mathbb{V}^{\rm Ned}_h\times\mathbb{V}^{\rm RT}_h,
\qquad u_h\ne0,
\]
such that
\begin{equation}\label{eq:3d-mixed-first-discrete}
	\left\{
	\begin{aligned}
		(\sigma_h,\tau_h)-(u_h,\curl\tau_h)&=0,
		&& \forall\,\tau_h\in \mathbb{V}^{\rm Ned}_h,\\
		(\curl\sigma_h,v_h)+(\dv u_h,\dv v_h)&=\lambda_h(u_h,v_h),
		&& \forall\,v_h\in \mathbb{V}^{\rm RT}_h .
	\end{aligned}
	\right.
\end{equation}

We compute the eigenvalues on a representative non-simply connected domain $\Omega_1 = [0,1]^3\setminus
\bigl((0.25,0.5)\times(0.25,0.5)\times[0,1]\bigr)$(see Figure~\ref{fig:3d-through-domain}) and compare them with those obtained from the mixed formulation on the same meshes. Unless otherwise stated, the smallest ten eigenvalues are reported; the results are presented in Tables~\ref{tab:3d-through-new-first}--\ref{tab:3d-through-mixed-first}.



\begin{figure}[htbp]
	\centering
	
	\begin{minipage}[t]{0.49\textwidth}
		\centering
		\includegraphics[height=4cm]{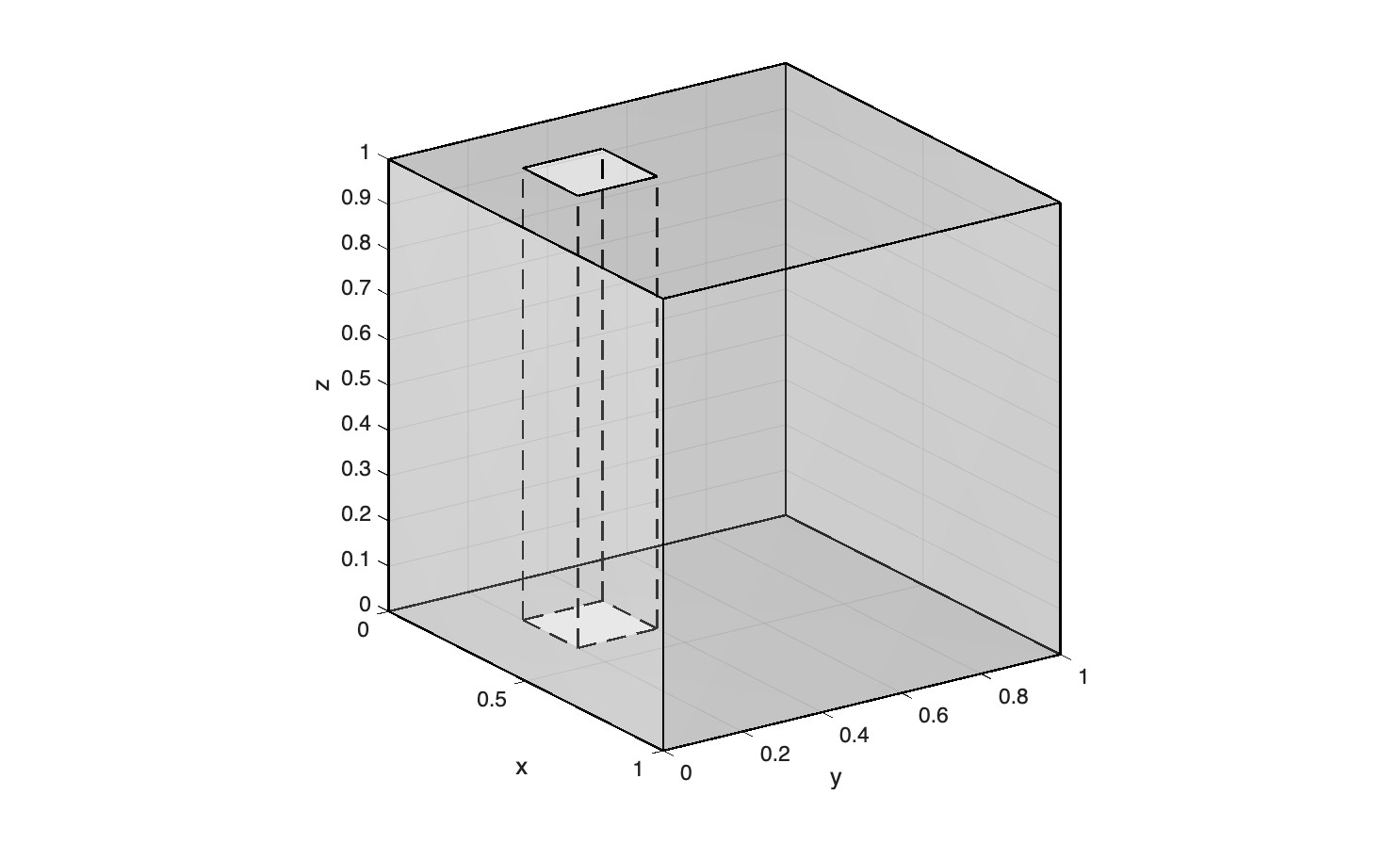}
		\caption{Computational domain \(\Omega_1\): the cube with a rectangular through-hole.}
		\label{fig:3d-through-domain}
	\end{minipage}
	\hfill
	\begin{minipage}[t]{0.49\textwidth}
		\centering
		\includegraphics[height=4cm]{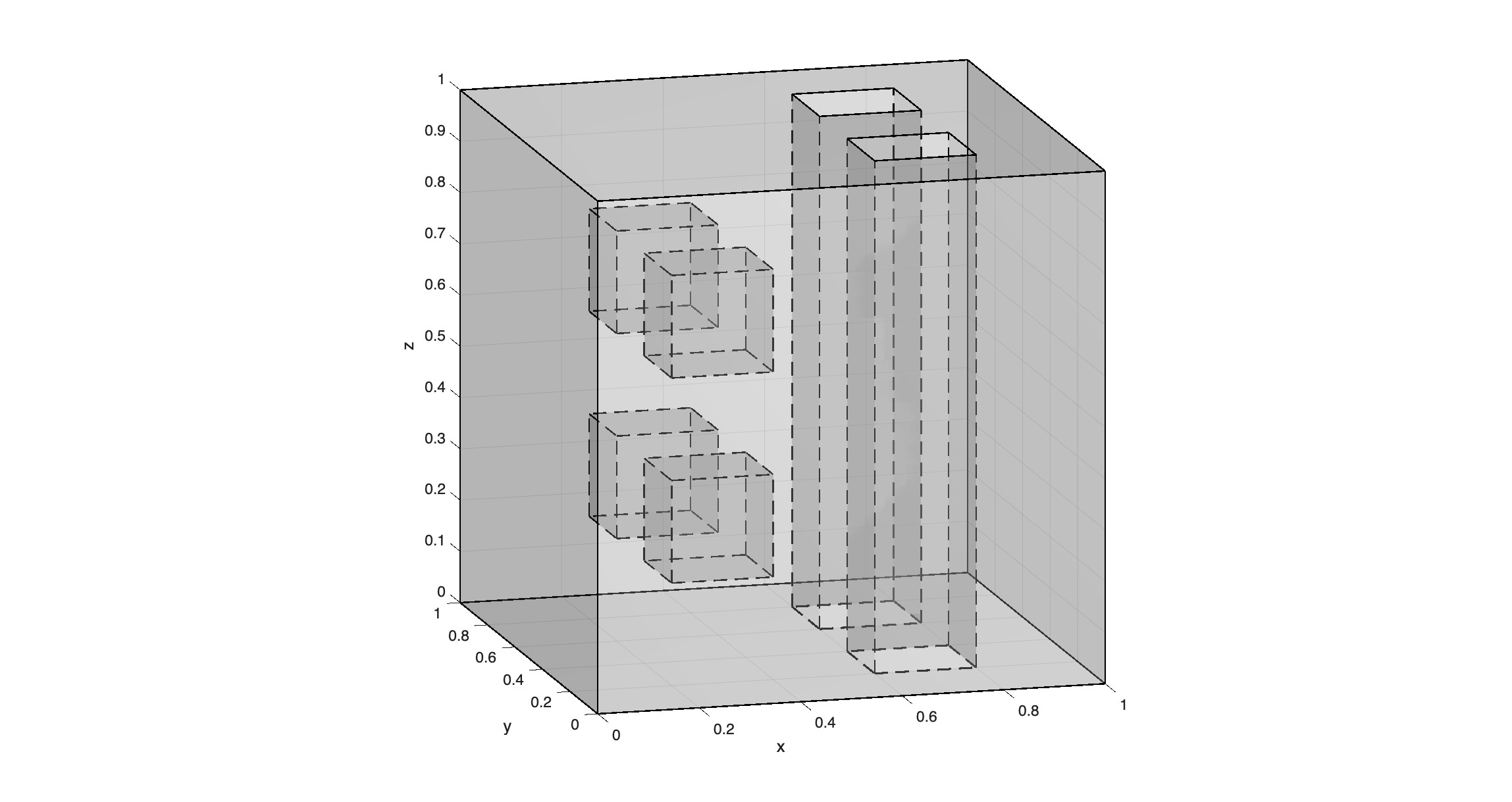}
		\caption{Computational domain \(\Omega_2\): the cube with four enclosed rectangular cavities and two through-holes.}
		\label{fig:3d-complex-domain}
	\end{minipage}
	
\end{figure}

%

\begin{table}[htbp]
	\centering
	\small
	\setlength{\tabcolsep}{3pt}
	\begin{tabular}{|c|*{10}{c|}}
		\hline
		\(L\) & \(\lambda_h^1\) & \(\lambda_h^2\) & \(\lambda_h^3\) & \(\lambda_h^4\) & \(\lambda_h^5\) & \(\lambda_h^6\) & \(\lambda_h^7\) & \(\lambda_h^8\) & \(\lambda_h^9\) & \(\lambda_h^{10}\)\\
		\hline
		1 & 9.139 & 18.149 & 18.443 & 28.730 & 33.144 & 33.664 & 41.078 & 43.122 & 44.284 & 44.695 \\ \hline
		2 & 9.602 & 17.967 & 18.150 & 28.632 & 36.483 & 37.620 & 45.080 & 45.776 & 46.381 & 46.735 \\ \hline
		3 & 9.774 & 17.879 & 18.047 & 28.562 & 37.588 & 38.907 & 44.883 & 47.269 & 47.345 & 47.363 \\ \hline
		4 & 9.834 & 17.845 & 18.011 & 28.535 & 38.007 & 39.288 & 44.734 & 47.380 & 47.520 & 47.869 \\
		\hline
	\end{tabular}
	\caption{Computed eigenvalues on \(\Omega_1\) shown in Figure~\ref{fig:3d-through-domain}, using the primal weak formulation \eqref{eq:3d-primal-first-discrete}.}
	\label{tab:3d-through-new-first}
\end{table}

\begin{table}[htbp]
	\centering
	\small
	\setlength{\tabcolsep}{3pt}
	\begin{tabular}{|c|*{10}{c|}}
		\hline
		\(L\) & \(\lambda_h^1\) & \(\lambda_h^2\) & \(\lambda_h^3\) & \(\lambda_h^4\) & \(\lambda_h^5\) & \(\lambda_h^6\) & \(\lambda_h^7\) & \(\lambda_h^8\) & \(\lambda_h^9\) & \(\lambda_h^{10}\)\\
		\hline
		1 & 9.200 & 18.419 & 18.613 & 29.282 & 33.983 & 34.524 & 44.736 & 45.095 & 45.181 & 45.894 \\ \hline
		2 & 9.618 & 18.032 & 18.193 & 28.765 & 36.726 & 37.864 & 45.417 & 46.772 & 46.945 & 46.990 \\ \hline
		3 & 9.778 & 17.895 & 18.058 & 28.595 & 37.652 & 38.969 & 44.965 & 47.390 & 47.411 & 47.552 \\ \hline
		4 & 9.835 & 17.849 & 18.014 & 28.543 & 38.024 & 39.304 & 44.754 & 47.410 & 47.537 & 47.895 \\
		\hline
	\end{tabular}
	\caption{Computed eigenvalues on \(\Omega_1\) shown in Figure~\ref{fig:3d-through-domain}, using the mixed formulation \eqref{eq:3d-mixed-first-discrete}.}
	\label{tab:3d-through-mixed-first}
\end{table}


\paragraph{\bf Convergence of the discrepancies between the two formulations.}

To further quantify the agreement between the proposed element and the mixed
formulation, we compare the corresponding numerical eigenvalues and
eigenvectors on \(\Omega_1\).

For the eigenvalue comparison, we consider the first ten computed eigenvalues and compute
\[
\bigl|\lambda_i^{\rm mix}-\lambda_i^{\rm new}\bigr|,
\qquad
i = 1,\dots, 10
\]
where \(\lambda_i^{\rm mix}\) and \(\lambda_i^{\rm new}\) denote the corresponding
eigenvalues computed by the mixed formulation and the proposed element,
respectively. The corresponding log--log plots are shown in
Figures~\ref{fig:value-omega1}.

For the eigenvector comparison, we consider the eigenvectors associated with the first ten computed eigenvalues on
\(\Omega_1\). The errors are measured in the \(L^2\) norm and in the combined
\(H(\dv)\)--\(H(\curl)\) norm,
\[
\|u_i^{\rm mix}-u_i^{\rm new}\|_{L^2},
\qquad
|u_i^{\rm mix}-u_i^{\rm new}|_{H(\dv)+H(\curl)},
\qquad
i = 1,\dots,10 .
\]
The results are displayed in
Figures~\ref{fig:eig10-l2}, \ref{fig:eig10-divcurl}.


\begin{figure}[htbp]
	\centering
	\hspace{-2.0em}%
	\begin{subfigure}[t]{0.32\textwidth}
		\makebox[\linewidth][l]{%
			\hspace{-2em}\includegraphics[width=1.5\linewidth]{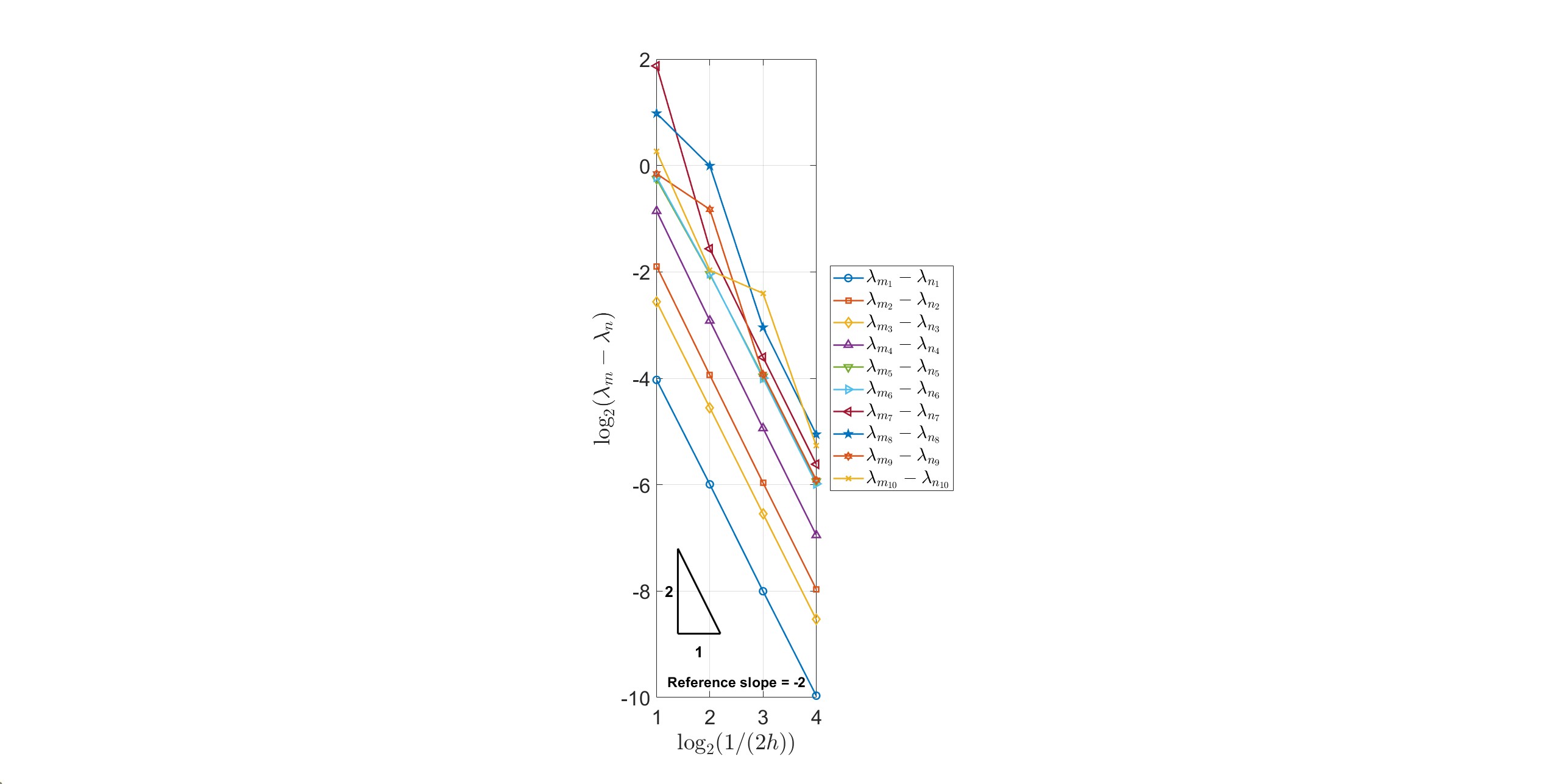}%
		}
		\caption{}
		\label{fig:value-omega1}
	\end{subfigure}%
	\hspace{-0.2em}%
	\begin{subfigure}[t]{0.32\textwidth}
		\makebox[\linewidth][l]{%
			\hspace{-1.0em}\includegraphics[width=1.3\linewidth]{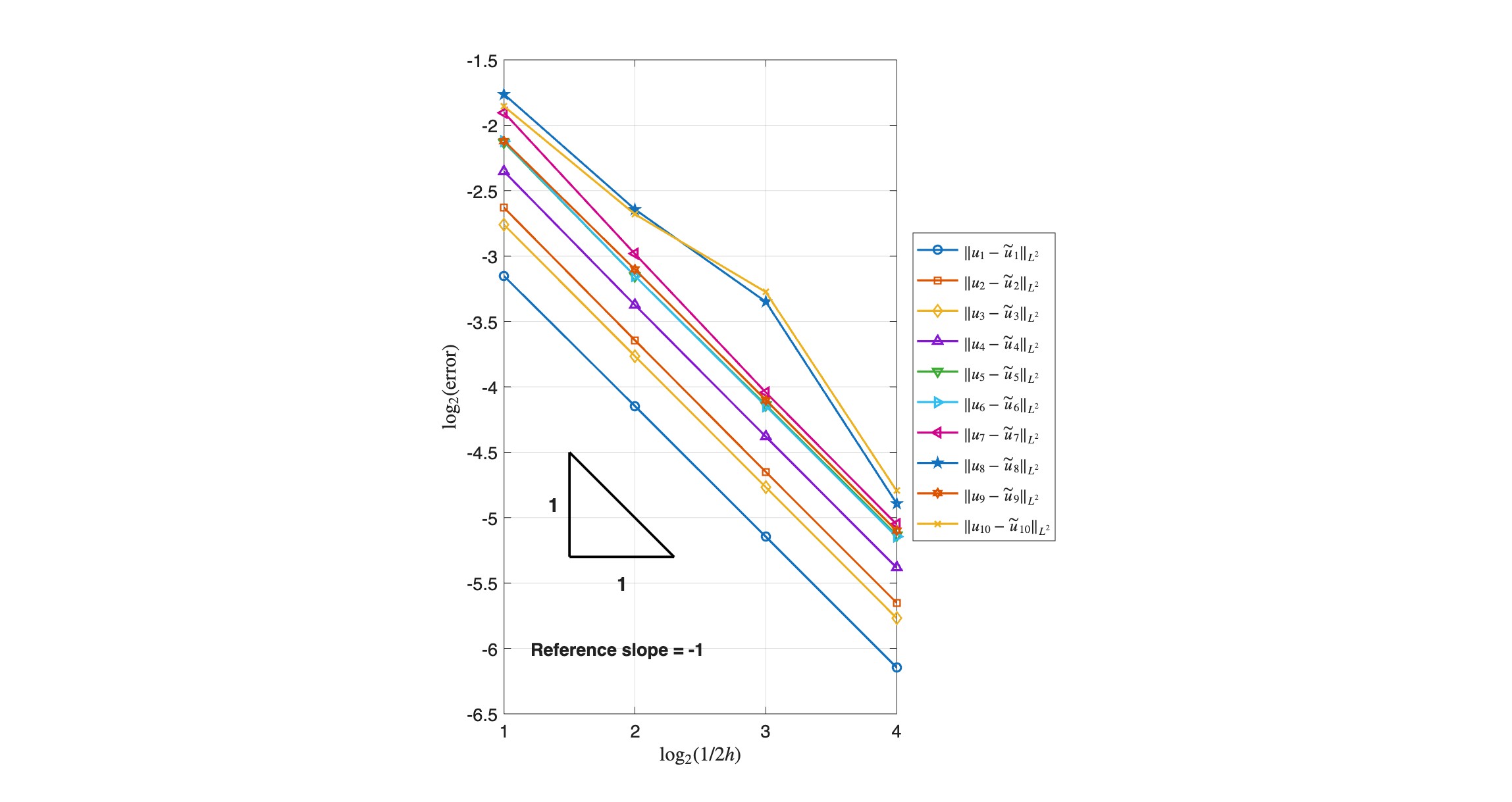}%
		}
		\caption{}
		\label{fig:eig10-l2}
	\end{subfigure}%
	\hspace{-0.4em}%
	\begin{subfigure}[t]{0.32\textwidth}
		\makebox[\linewidth][l]{%
			\hspace{-1.0em}\includegraphics[width=1.3\linewidth]{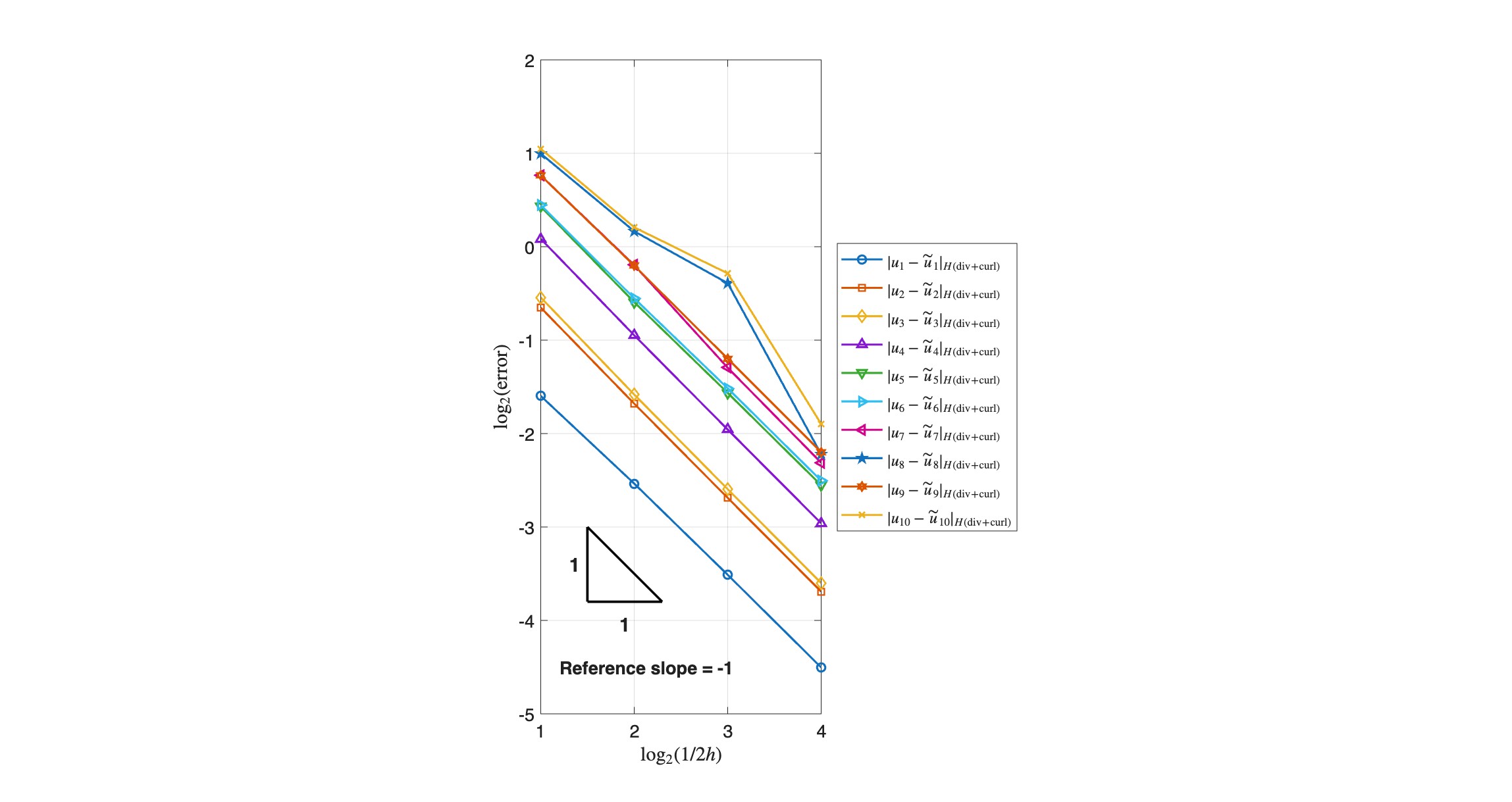}%
		}
		\caption{}
		\label{fig:eig10-divcurl}
	\end{subfigure}
	
	\caption{%
		(a) Log--log plots of the differences between the first ten eigenvalues computed by the mixed formulation and the proposed element on \(\Omega_1\); 
		(b) Errors between the first ten matched eigenvectors computed by the mixed formulation and the proposed element in the \(L^2\) norm on \(\Omega_1\); 
		(c) Errors between the first ten matched eigenvectors computed by the mixed formulation and the proposed element in the combined \(H(\dv)\)--\(H(\curl)\) norm on \(\Omega_1\).}
	\label{fig:omega1-eig-errors}
\end{figure}

The log--log plots show that the discrepancies between the two formulations
decrease as the mesh is refined. This provides an additional consistency check:
not only do the two methods produce close eigenvalue tables on each fixed mesh,
but their eigenvalue and eigenvector approximations also become increasingly
consistent under mesh refinement.

%

%
\subsubsection{Influence of topology associated with the boundary condition}
\label{sec:twobc}

For comparison, we also consider the space $H(\curl,\Omega)\cap H_0(\dv,\Omega)$. It corresponds to $H\Lambda^1\cap H^*_0\Lambda^1$, and leads to Hodge-Laplace problems on $\Lambda^1$. The first Betti number $\mathtt{b}_1$ influence its computation. 

We consider the eigenvalue problem: find $(\lambda,\omega)\in \mathbb{R}\times H(\curl,\Omega)\cap H_0(\dv,\Omega)$ such that
\begin{equation}
	(\dv\fomega,\dv\fmu) + (\curl\fomega,\curl\fmu) = \lambda(\omega,\fmu),\forall\fmu\in H(\curl,\Omega)\cap H_0(\dv,\Omega)
	\label{eq:3d-primal-second}
\end{equation}

To compare and illustrate the influence of the boundary condition, we just make slight modification on $\fV^{\mathring{\mathrm c}{\rm d}}_h$, rather than define directly the space $\fV^{\od\cap\mathring{\odelta}}_h\Lambda^1$. Denote by $\mathbb{V}^{\rm Ned}_{h0}$ the Nedelec element space on $\mathcal{T}_h$ with homogeneous boundary condition and by $V^{\rm CR}_{h}$ the Crouzeix-Raviart element space on $\mathcal{T}_h$. The finite element space reads:
\begin{multline}
	\fV^{\rm c\mathring{\mathrm d} }_h:=\Big\{\fmu_h\in P_{\rm rd}(\mathcal{G}_h):
	(\dv_h\fmu_h,\ftau_h)=-(\fmu_h,\nabla_h\ftau_h),\ \forall\,\ftau_h\in \mathbb{V}^{\rm CR}_h, 
	\\
	\mbox{and}\ \  
	(\rot_h\fmu_h,\feta_h)=\langle\fmu_h,\curl\feta_h),\ \forall\,\feta_h\in V^{\rm Ned}_{h0} \Big\}.
\end{multline}
We remark here that the homogeneous boundary condition for the adjoint partner is transferred from the Crouzeix-Raviart element space side to the Nedelec element space side compared to $\fV^{\mathring{\mathrm c}{\rm d}}_h$. 

The discrete primal weak formulation is to find
\[
(\lambda_h,\omega_h)\in 
\mathbb R\times \fV^{\rm c \mathring{\rm d} }_h,
\qquad \omega_h\ne0,
\]
such that
\begin{equation}
	(\dv_h\fomega_h,\dv_h\fmu_h)
	+
	(\curl_h\fomega_h,\curl_h\fmu_h)
	=
	\lambda_h(\fomega_h,\fmu_h),
	\qquad
	\forall\,\fmu_h\in \fV^{{\rm c}\mathring{\mathrm d}}_h .
	\label{eq:3d-primal-second-discrete}
\end{equation}

The corresponding mixed formulation is to find $(\lambda,\sigma,u)\in \mathbb{R}\times H^1(\Omega) \times H(\curl,\Omega)$ such that
\begin{equation}\label{eq:3d-mixed-second}
	\left\{
	\begin{array}{cclll}
		(\sigma,\tau)&-(u,\grad\tau)&=0&\forall\,\tau\in H^1(\Omega)
		\\
		(\grad\sigma,v)&+(\curl u ,\curl v)&=\lambda (u,v)&\forall\,v\in H(\curl,\Omega)
	\end{array}
	\right.
\end{equation}

The discrete mixed formulation is to find
\[
(\lambda_h,\sigma_h,u_h)\in
\mathbb R\times\mathbb{V}^1_{h}\times\mathbb{V}^{\rm Ned}_{h},
\qquad u_h\ne0,
\]
such that
\begin{equation}\label{eq:3d-mixed-second-discrete}
	\left\{
	\begin{aligned}
		(\sigma_h,\tau_h)-(u_h,\grad\tau_h)&=0,
		&& \forall\,\tau_h\in \mathbb{V}^1_{h},\\
		(\grad\sigma_h,v_h)+(\curl u_h,\curl v_h)&=\lambda_h(u_h,v_h),
		&& \forall\,v_h\in \mathbb{V}^{\rm Ned}_{h} .
	\end{aligned}
	\right.
\end{equation}



%
%


 For this type of boundary condition, the number of zero eigenvalues is determined by the first Betti number $\mathtt{b}_1$ of the domain.

To examine the influence of boundary conditions on the numerical results, we consider $\Omega_2
= [0,1]^3\setminus
\left[
(0.2,0.4)\times Y\times Y
\;\cup\;
(0.6,0.8)\times Y\times(0,1)
\right],Y = (0.2,0.4)\cup(0.6,0.8)$. This domain combines cavity-type and handle-type topological features, and is illustrated in Figure~\ref{fig:3d-complex-domain}. Its first Betti number $\mathtt{b}_1=2$ and second Betti number $\mathtt{b}_2=4$. Tables~\ref{tab:3d-complex-new-first} and~\ref{tab:3d-complex-mixed-first} report the computed eigenvalues under the boundary condition \(H_0(\rot,\Omega)\cap H(\dv,\Omega)\), while the results under the boundary condition \(H_0(\dv,\Omega)\cap H(\rot,\Omega)\) are presented in Tables~\ref{New Element, six hole, second bdc}--\ref{Mixed Element, six hole, second bdc}.

\begin{table}[htbp]
	\centering
	\small
	\setlength{\tabcolsep}{3pt}
	\begin{tabular}{|c|*{10}{c|}}
		\hline
		\(L\) & \(\lambda_h^1\) & \(\lambda_h^2\) & \(\lambda_h^3\) & \(\lambda_h^4\) & \(\lambda_h^5\) & \(\lambda_h^6\) & \(\lambda_h^7\) & \(\lambda_h^8\) & \(\lambda_h^9\) & \(\lambda_h^{10}\)\\
		\hline
		1 & 0.000 & 0.000 & 0.000 & 0.000 & 9.124 & 9.140 & 17.248 & 17.381 & 26.886 & 27.005 \\ \hline
		2 & 0.000 & 0.000 & 0.000 & 0.000 & 9.407 & 9.513 & 16.579 & 16.698 & 26.023 & 26.193 \\ \hline
		3 & 0.000 & 0.000 & 0.000 & 0.000 & 9.492 & 9.674 & 16.242 & 16.359 & 25.367 & 25.813 \\ \hline
		4 & 0.000 & 0.000 & 0.000 & 0.000 & 9.515 & 9.734 & 16.094 & 16.213 & 25.036 & 25.651 \\
		\hline
	\end{tabular}
	\caption{Computed eigenvalues on \(\Omega_2\) shown in Figure~\ref{fig:3d-complex-domain}, using the primal weak formulation \eqref{eq:3d-primal-first-discrete}.}
	\label{tab:3d-complex-new-first}
\end{table}

\begin{table}[htbp]
	\centering
	\small
	\setlength{\tabcolsep}{3pt}
	\begin{tabular}{|c|*{10}{c|}}
		\hline
		\(L\) & \(\lambda_h^1\) & \(\lambda_h^2\) & \(\lambda_h^3\) & \(\lambda_h^4\) & \(\lambda_h^5\) & \(\lambda_h^6\) & \(\lambda_h^7\) & \(\lambda_h^8\) & \(\lambda_h^9\) & \(\lambda_h^{10}\)\\
		\hline
		1 & 0.000 & 0.000 & 0.000 & 0.000 & 9.162 & 9.179 & 17.343 & 17.537 & 27.162 & 27.351 \\ \hline
		2 & 0.000 & 0.000 & 0.000 & 0.000 & 9.417 & 9.523 & 16.604 & 16.731 & 26.092 & 26.269 \\ \hline
		3 & 0.000 & 0.000 & 0.000 & 0.000 & 9.494 & 9.676 & 16.248 & 16.367 & 25.384 & 25.830 \\ \hline
		4 & 0.000 & 0.000 & 0.000 & 0.000 & 9.515 & 9.734 & 16.096 & 16.215 & 25.040 & 25.655 \\
		\hline
	\end{tabular}
	\caption{Computed eigenvalues on \(\Omega_2\) shown in Figure~\ref{fig:3d-complex-domain}, using the mixed formulation \eqref{eq:3d-mixed-first-discrete}.}
	\label{tab:3d-complex-mixed-first}
\end{table}

\begin{table}[htbp]
	\begin{tabular}{|c|c|c|c|c|c|c|c|c|c|c|}
		\hline
		L & $\lambda_h^1$ & $\lambda_h^2$ & $\lambda_h^3$ & $\lambda_h^4$ & $\lambda_h^5$ &$\lambda_h^6$ & $\lambda_h^7$ & $\lambda_h^8$ & $\lambda_h^9$ & $\lambda_h^{10}$\\
		\hline
		1 & 0.000 & 0.000 & 6.958 & 7.338 & 8.507 & 8.736 & 8.973 & 13.233 & 13.417 & 16.041 \\
		\hline
		2 & 0.000 & 0.000 & 7.491 & 7.767 & 9.122 & 9.252 & 9.385 & 14.783 & 14.840 & 16.796 \\
		\hline
		3 & 0.000 & 0.000 & 7.711 & 7.954 & 9.353 & 9.367 & 9.627 & 15.477 & 15.580 & 17.100 \\
		\hline
		4 & 0.000 & 0.000 & 7.799 & 8.031 & 9.390 & 9.464 & 9.717 & 15.784 & 15.899 & 17.218 \\
		\hline
	\end{tabular}
	\caption{Computed eigenvalues on \(\Omega_2\) shown in Figure~\ref{fig:3d-complex-domain}, using the primal weak formulation \eqref{eq:3d-primal-second-discrete}.}
	\label{New Element, six hole, second bdc}
\end{table}

\begin{table}[htbp]
	\begin{tabular}{|c|c|c|c|c|c|c|c|c|c|c|}
		\hline
		L & $\lambda_h^1$ & $\lambda_h^2$ & $\lambda_h^3$ & $\lambda_h^4$ & $\lambda_h^5$ &$\lambda_h^6$ & $\lambda_h^7$ & $\lambda_h^8$ & $\lambda_h^9$ & $\lambda_h^{10}$\\
		\hline
		1 & 0.000 & 0.000 & 8.825 & 8.974 & 9.162 & 9.179 & 9.889 & 17.343 & 17.537 & 19.520 \\
		\hline
		2 & 0.000 & 0.000 & 8.302 & 8.489 & 9.417 & 9.523 & 9.605 & 16.604 & 16.731 & 18.126 \\
		\hline
		3 & 0.000 & 0.000 & 8.042 & 8.250 & 9.488 & 9.494 & 9.676 & 16.248 & 16.367 & 17.597 \\
		\hline
		4 & 0.000 & 0.000 & 7.930 & 8.148 & 9.441 & 9.515 & 9.734 & 16.096 & 16.215 & 17.404 \\
		\hline
	\end{tabular}
	\caption{Computed eigenvalues on \(\Omega_2\) shown in Figure~\ref{fig:3d-complex-domain}, using the mixed formulation \eqref{eq:3d-mixed-second-discrete}.}
	\label{Mixed Element, six hole, second bdc}
\end{table}

As in the cases of \(\Omega_1\), we further compare the
numerical eigenvalues and eigenvectors produced by the primal weak formulation \eqref{eq:3d-primal-first-discrete} and the mixed formulation \eqref{eq:3d-mixed-first-discrete}. The positive eigenvalues among the first
ten computed eigenvalues, together with the associated eigenvectors, are
considered in the comparison. The corresponding results are presented in
Figure~\ref{fig:omega2-eig-errors}.

\begin{figure}[htbp]
	\centering
	\hspace{-2.0em}%
	\begin{subfigure}[t]{0.32\textwidth}
		\makebox[\linewidth][l]{%
			\hspace{-1.5em}\includegraphics[width=1.3\linewidth]{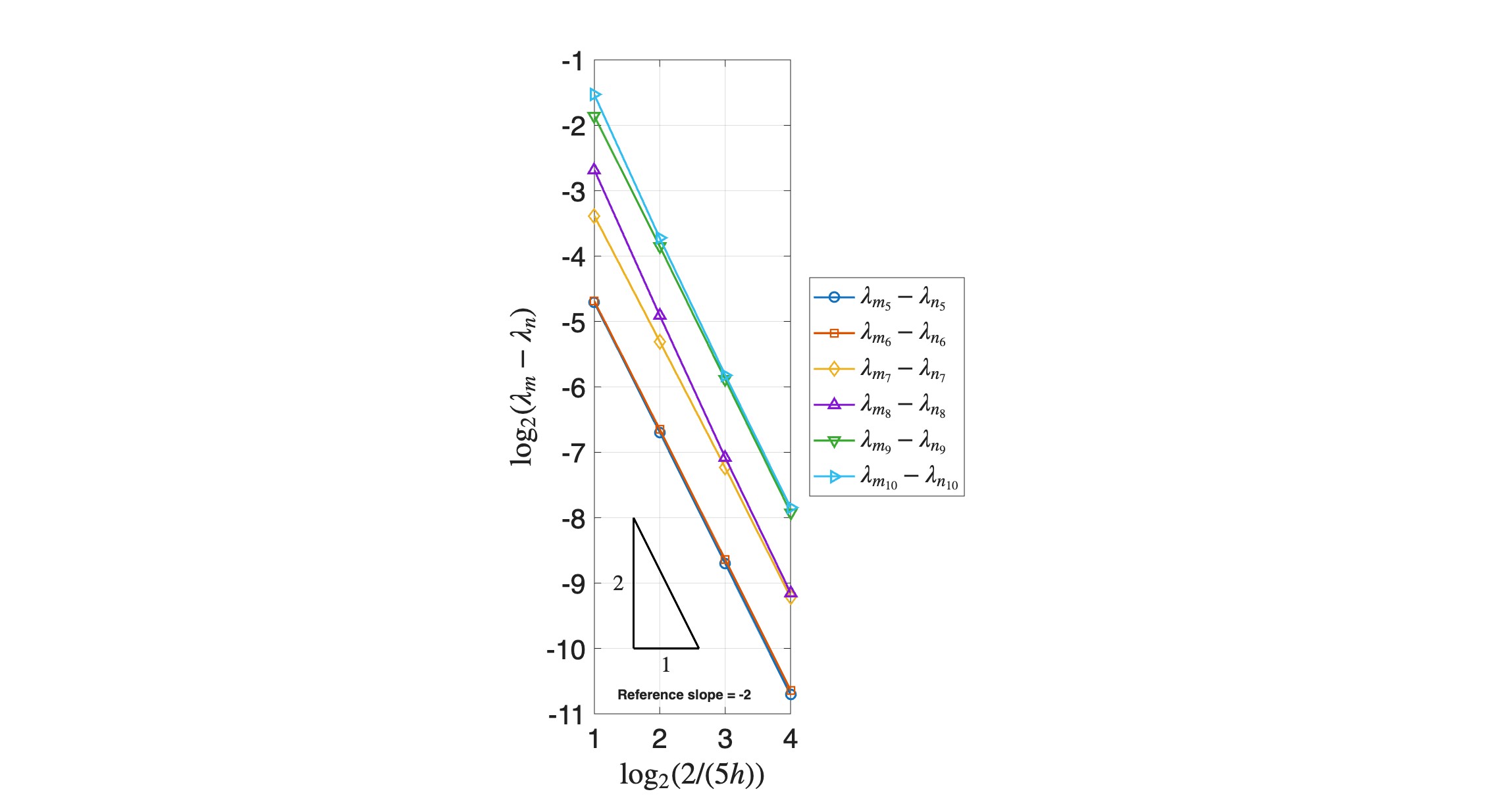}%
		}
		\caption{}
		\label{fig:value-omega2}
	\end{subfigure}%
	\hspace{-0.2em}%
	\begin{subfigure}[t]{0.32\textwidth}
		\makebox[\linewidth][l]{%
			\hspace{-1.0em}\includegraphics[width=1.3\linewidth]{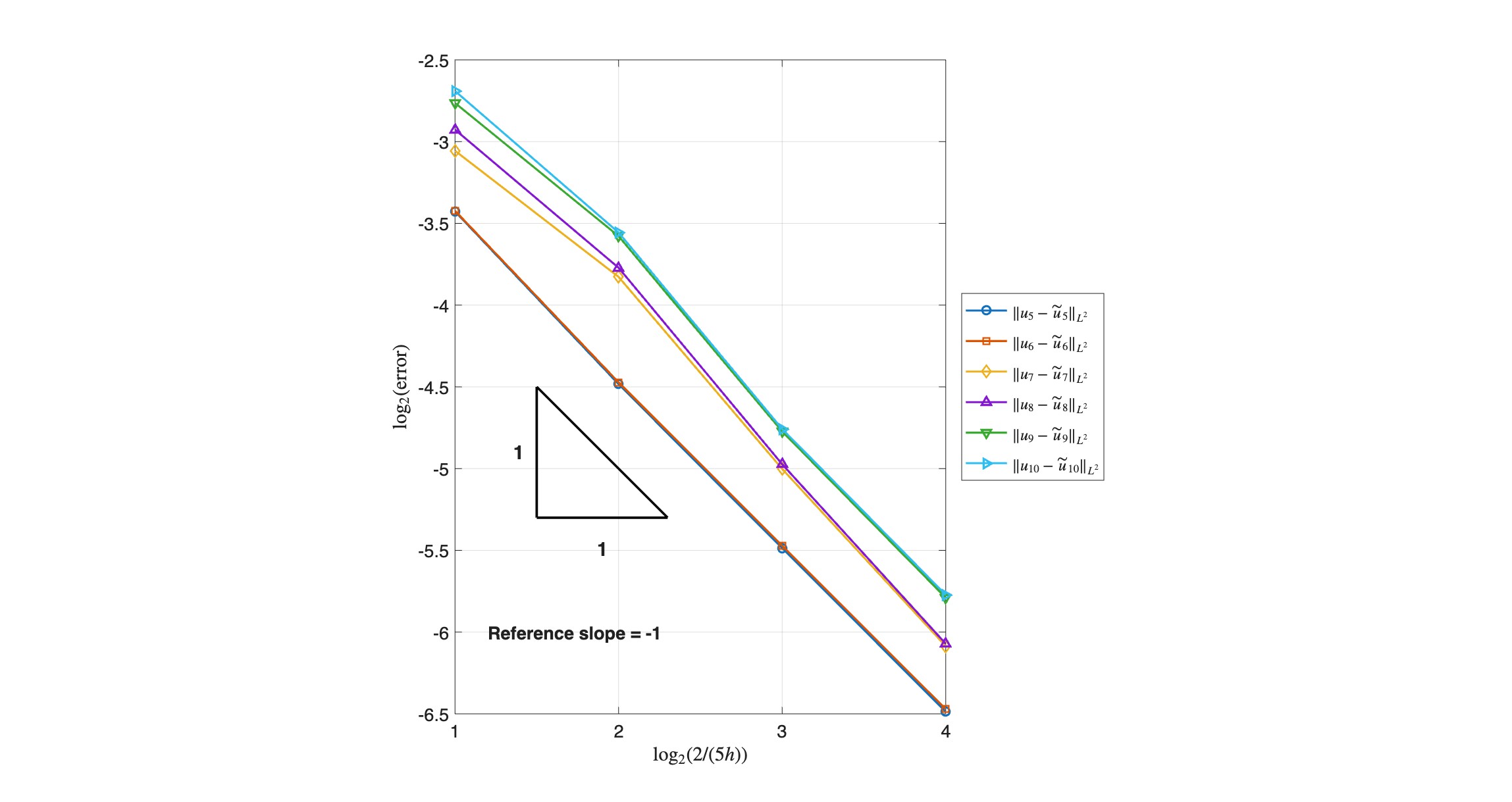}%
		}
		\caption{}
		\label{fig:eig10-l2-omega2}
	\end{subfigure}%
	\hspace{-0.4em}%
	\begin{subfigure}[t]{0.32\textwidth}
		\makebox[\linewidth][l]{%
			\hspace{-1.0em}\includegraphics[width=1.3\linewidth]{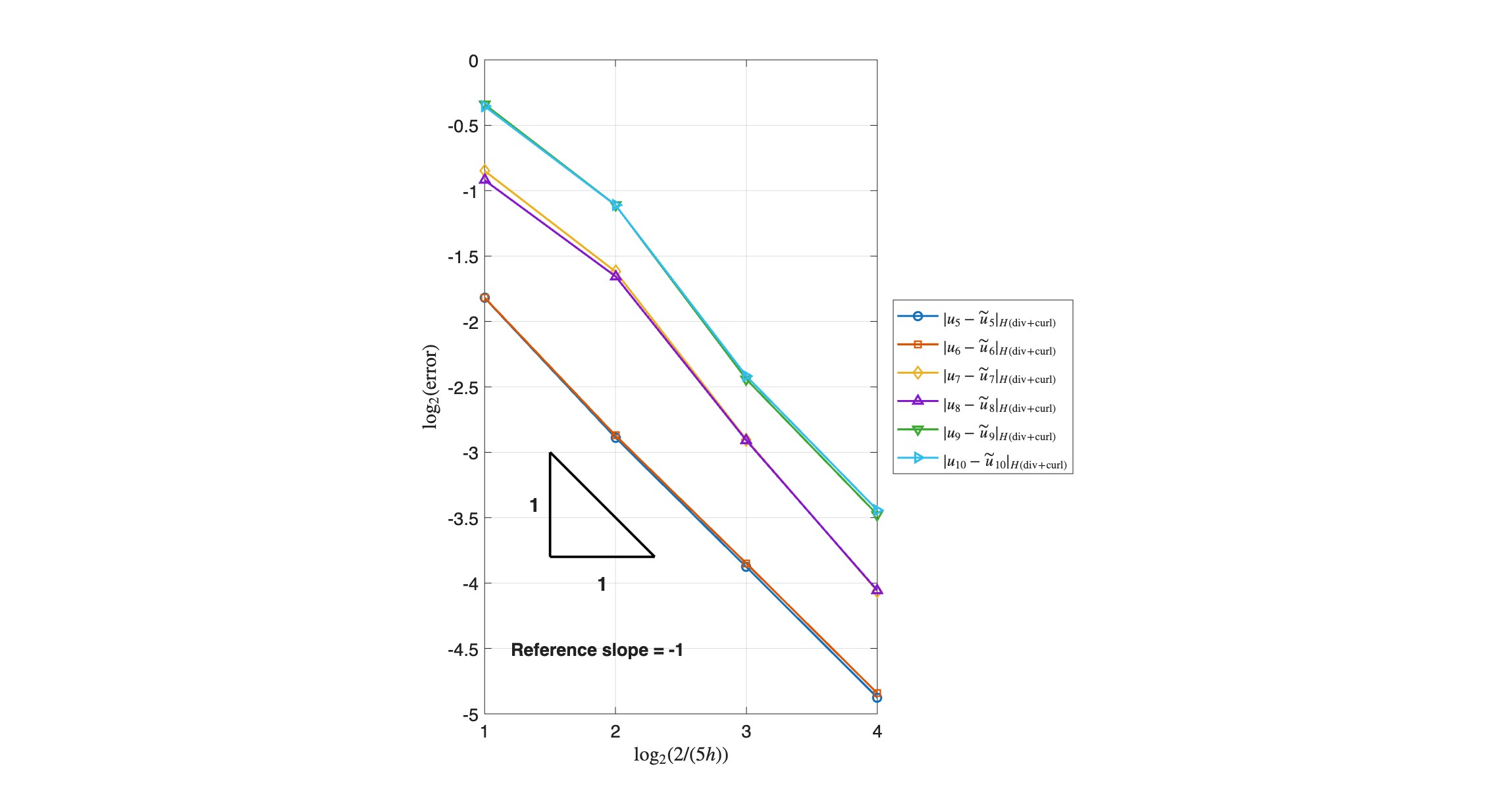}%
		}
		\caption{}
		\label{fig:eig10-divcurl-omega2}
	\end{subfigure}
	
	\caption{%
		(a) Log--log plots of the differences between the first six positive eigenvalues computed by the mixed formulation and the proposed element on \(\Omega_2\); 
		(b) Errors between the corresponding six matched eigenvectors computed by the mixed formulation and the proposed element in the \(L^2\) norm on \(\Omega_2\); 
		(c) Errors between the corresponding six matched eigenvectors computed by the mixed formulation and the proposed element in the combined \(H(\dv)\)--\(H(\curl)\) norm on \(\Omega_2\).}
	\label{fig:omega2-eig-errors}
\end{figure}

\paragraph{\bf Summary.}

For all tested three-dimensional domains, the eigenvalues computed by the proposed element agree well with those obtained from the mixed formulation. The agreement remains valid for domains with nontrivial topology. In particular, for the domain with multiple cavities and through-holes, the proposed element correctly captures the expected zero eigenvalues associated with the corresponding Betti number. It is also worth noting that, whenever zero eigenvalues occur, either simple or multiple, we further compared the corresponding zero eigenspaces computed by the proposed element and by the mixed element. The comparison shows that the two discretizations give exactly the same eigenspace associated with the eigenvalue 0, up to roundoff error. This demonstrates that the proposed element preserves the relevant topological structure of the continuous problem.

\section{Concluding remarks}
\label{sec:conc}

In this paper, we constructed a unified nonconforming primal finite element framework for the Hodge-Laplace problem on $H\Lambda^k\cap H^*_0\Lambda^k$, for any $n\geqslant 2$ and $1\leqslant k\leqslant n-1$ on simplicial meshes. The construction is based on adjoint continuity: instead of imposing trace continuity directly, the space is defined by discrete adjoint identities against Whitney-type test spaces. This construction preserves the key FEEC structures needed for topology-sensitive problems. In particular, discrete harmonic forms are recovered as the joint kernel of the discrete exterior derivative and codifferential, the discrete Poincar\'e inequality holds uniformly in $h$, and the resulting primal scheme is well posed. By linking the primal scheme to a classical mixed FEEC discretization, we obtain the error estimate in Theorem~\ref{thm:convprimsch}, including the $\mathcal{O}(h^s)$ rate on $s$-regular domains. The numerical eigenvalue tests in Section~\ref{sec:numer} (including domains with nontrivial topology) are consistent with the theoretical structure and with mixed-method reference results.

A systematic higher-order extension that preserves the same adjoint continuity and topological robustness will still be of interest, particularly for smoother solutions and hp/adaptive settings.  For the topology-sensitive domains emphasized here, however, the error rate is fundamentally constrained by solution regularity: on $s$-regular domains, the dominant behavior is $\mathcal{O}(h^s)$ regardless of polynomial enrichment beyond degree one. In this sense, the present low-degree construction is already aligned with the regularity-limited regime addressed in Theorem~\ref{thm:convprimsch} as well as Theorem 7.10 of \cite{Arnold.D;Falk.R;Winther.R2006acta}. At the same time, the space remains non-Ciarlet type in nature, so basis construction and implementation rely on the adjoint-constraint framework rather than standard nodal continuity machinery. On the computational side, robust preconditioners and conditioning estimates adapted to the discrete Hodge decomposition are needed for large-scale solvers. It is also natural to investigate whether the same adjoint-construction strategy can be transferred to other coupled systems, such as Stokes, Maxwell, and linear elasticity.

\appendix

\section*{Guide to the appendices}
Appendices~\ref{app:lagrange-eigenvalues} and~\ref{app:cr-eigenvalues} document eigenvalue counterexamples with vector Lagrange and vector Crouzeix--Raviart elements; they supplement the proposed-element and mixed results in Section~\ref{sec:numer}.
Appendix~\ref{sec:ncbasisfunction} illustrates the adjoint partner spaces $\fW^{*,\rm nc}_h\Lambda^k$ used in~\eqref{eq:deffems}; Appendices~\ref{sec:app:basis-2d} and~\ref{sec:app:basis-3d} give explicit basis constructions for the proposed trial space $\fV^{\od\cap\mathring{\odelta}}_h\Lambda^1$ and $\fV^{\od\cap\mathring{\odelta}}_h\Lambda^2$ in two and three dimensions, respectively. Appendix~\ref{app:bvp-results} reports manufactured-solution boundary-value tests on $[0,1]^2$ and $[0,1]^3$, including log--log plots of mesh-refinement errors.

\section{Eigenvalue results obtained by the vector Lagrange element}
\label{app:lagrange-eigenvalues}

This appendix collects supplementary eigenvalue results for the
linear vector Lagrange element, to be compared with the proposed-element and mixed results in Section~\ref{sec:numer}. 

\subsection{Two-dimensional vector Lagrange discretization and results}
\label{app:lagrange-eigenvalues-2d}

In two dimensions, the vector Lagrange space is used as a replacement for the
proposed space \(\fV^{\rm r\mathring d}_h\) in
\eqref{eq:2d-primal-discrete}, which discretizes the primal formulation
\eqref{primal2D}.

Let
\[
S_{h,2D}^{\rm Lag}
=
\{v_h\in H^1(\Omega): v_h|_T\in \mathbb P_1(T),
\ \forall T\in\mathcal{G}_h\}.
\]
For the two-dimensional primal formulation \eqref{primal2D}, we define
\[
\boldsymbol V_{h,2D}^{\rm Lag}
=
[S_{h,2D}^{\rm Lag}]^2
\cap
\bigl(H(\rot,\Omega)\cap H_0(\dv,\Omega)\bigr).
\]
The corresponding vector Lagrange discretization is to find
\[
(\lambda_h,\omega_h)\in
\mathbb R\times \boldsymbol V_{h,2D}^{\rm Lag},
\qquad \omega_h\ne0,
\]
such that
\begin{equation}
	(\dv\omega_h,\dv\tau_h)
	+
	(\rot\omega_h,\rot\tau_h)
	=
	\lambda_h(\omega_h,\tau_h),
	\qquad
	\forall\,\tau_h\in \boldsymbol V_{h,2D}^{\rm Lag}.
	\label{eq:lagrange-2d-primal-discrete}
\end{equation}

The corresponding two-dimensional eigenvalue results are reported below.

\begin{table}[htbp]
	\centering
	\small
	\setlength{\tabcolsep}{3pt}
	\begin{tabular}{|c|*{10}{c|}}
		\hline
		$L$ & $\lambda_h^1$ & $\lambda_h^2$ & $\lambda_h^3$ & $\lambda_h^4$ & $\lambda_h^5$ & $\lambda_h^6$ & $\lambda_h^7$ & $\lambda_h^8$ & $\lambda_h^9$ & $\lambda_h^{10}$ \\
		\hline
		1 & 10.211 & 10.211 & 20.608 & 20.608 & 45.012 & 45.012 & 56.070 & 56.070 & 56.070 & 56.070 \\ \hline
		2 &  9.954 &  9.954 & 19.952 & 19.952 & 40.843 & 40.843 & 50.977 & 50.977 & 50.977 & 50.977 \\ \hline
		3 &  9.891 &  9.891 & 19.792 & 19.792 & 39.817 & 39.817 & 49.751 & 49.751 & 49.751 & 49.751 \\ \hline
		4 &  9.875 &  9.875 & 19.752 & 19.752 & 39.563 & 39.563 & 49.449 & 49.449 & 49.449 & 49.449 \\
		\hline
	\end{tabular}
	\caption{Computed eigenvalues on \(\Omega=[0,1]^2\) using the vector Lagrange discretization \eqref{eq:lagrange-2d-primal-discrete}, corresponding to the two-dimensional primal formulation \eqref{primal2D}.}
	\label{tab:sq-primal-lagrange}
\end{table}

\begin{table}[htbp]
	\centering
	\small
	\setlength{\tabcolsep}{3pt}
	\begin{tabular}{|c|*{10}{c|}}
		\hline
		$L$ & $\lambda_h^1$ & $\lambda_h^2$ & $\lambda_h^3$ & $\lambda_h^4$ & $\lambda_h^5$ & $\lambda_h^6$ & $\lambda_h^7$ & $\lambda_h^8$ & $\lambda_h^9$ & $\lambda_h^{10}$ \\
		\hline
		1 & 17.418 & 17.418 & 45.012 & 45.012 & 52.760 & 52.760 & 80.195 & 80.195 & 93.723 & 93.723 \\ \hline
		2 & 15.731 & 15.731 & 40.843 & 40.843 & 47.644 & 47.644 & 67.110 & 67.110 & 82.432 & 82.432 \\ \hline
		3 & 15.026 & 15.026 & 39.817 & 39.817 & 46.282 & 46.282 & 63.522 & 63.522 & 79.808 & 79.808 \\ \hline
		4 & 14.664 & 14.664 & 39.563 & 39.563 & 45.870 & 45.870 & 62.209 & 62.209 & 79.169 & 79.169 \\
		\hline
	\end{tabular}
	\caption{Computed eigenvalues on the L-shaped domain shown in Figure~\ref{fig:Lshape} using the vector Lagrange discretization \eqref{eq:lagrange-2d-primal-discrete}, corresponding to the two-dimensional primal formulation \eqref{primal2D}.}
	\label{tab:L-primal-lagrange}
\end{table}

\begin{table}[htbp]
	\centering
	\small
	\setlength{\tabcolsep}{3pt}
	\begin{tabular}{|c|*{10}{c|}}
		\hline
		$L$ & $\lambda_h^1$ & $\lambda_h^2$ & $\lambda_h^3$ & $\lambda_h^4$ & $\lambda_h^5$ & $\lambda_h^6$ & $\lambda_h^7$ & $\lambda_h^8$ & $\lambda_h^9$ & $\lambda_h^{10}$ \\
		\hline
		1 & 19.407 & 19.407 & 39.473 & 39.473 & 52.977 & 52.977 & 71.814 & 71.814 & 89.737 & 89.737 \\ \hline
		2 & 17.451 & 17.451 & 31.855 & 31.855 & 46.680 & 46.680 & 56.952 & 56.952 & 74.032 & 74.032 \\ \hline
		3 & 16.594 & 16.594 & 28.731 & 28.731 & 44.767 & 44.767 & 53.076 & 53.076 & 69.231 & 69.231 \\ \hline
		4 & 16.136 & 16.136 & 27.107 & 27.107 & 44.057 & 44.057 & 51.652 & 51.652 & 67.177 & 67.177 \\
		\hline
	\end{tabular}
	\caption{Computed eigenvalues on the square domain with an interior hole shown in Figure~\ref{fig:hole} using the vector Lagrange discretization \eqref{eq:lagrange-2d-primal-discrete}, corresponding to the two-dimensional primal formulation \eqref{primal2D}.}
	\label{tab:hole-primal-lagrange}
\end{table}

\subsection{Three-dimensional vector Lagrange discretization and results}
\label{app:lagrange-eigenvalues-3d}

In three dimensions, we consider the primal formulation with respect to
\(H(\dv,\Omega)\cap H_0(\curl,\Omega)\). The vector Lagrange space is used as a replacement
for the proposed space \(\fV^{\mathring{\mathrm r}{\rm d}}_h\) in
\eqref{eq:3d-primal-first-discrete}, which discretizes the primal formulation
\eqref{eq:3d-primal-first}.

Let
\[
S_{h,3D}^{\rm Lag}
=
\{v_h\in H^1(\Omega): v_h|_T\in \mathbb P_1(T),
\ \forall T\in\mathcal{G}_h\}.
\]
For the three-dimensional primal formulation with respect to
\(H(\dv,\Omega)\cap H_0(\curl,\Omega)\), we define
\[
\boldsymbol V_{h,3D}^{\rm Lag}
=
[S_{h,3D}^{\rm Lag}]^3
\cap
\bigl(H(\dv,\Omega)\cap H_0(\curl,\Omega)\bigr).
\]
The corresponding vector Lagrange discretization is to find
\[
(\lambda_h,\omega_h)\in
\mathbb R\times \boldsymbol V_{h,3D}^{\rm Lag},
\qquad \omega_h\ne0,
\]
such that
\begin{equation}
	(\dv\omega_h,\dv\tau_h)
	+
	(\curl\omega_h,\curl\tau_h)
	=
	\lambda_h(\omega_h,\tau_h),
	\qquad
	\forall\,\tau_h\in \boldsymbol V_{h,3D}^{\rm Lag}.
	\label{eq:lagrange-3d-primal-first-discrete}
\end{equation}

The corresponding three-dimensional eigenvalue results are reported below.

\begin{table}[htbp]
	\centering
	\small
	\setlength{\tabcolsep}{3pt}
	\begin{tabular}{|c|*{10}{c|}}
		\hline
		\(L\) & \(\lambda_h^1\) & \(\lambda_h^2\) & \(\lambda_h^3\) & \(\lambda_h^4\) & \(\lambda_h^5\) & \(\lambda_h^6\) & \(\lambda_h^7\) & \(\lambda_h^8\) & \(\lambda_h^9\) & \(\lambda_h^{10}\)\\
		\hline
		1 & 22.838 & 22.838 & 22.838 & 37.293 & 37.293 & 37.293 & 62.453 & 62.453 & 62.453 & 71.457 \\ \hline
		2 & 20.503 & 20.503 & 20.503 & 31.509 & 31.509 & 31.509 & 52.609 & 52.609 & 52.609 & 54.575 \\ \hline
		3 & 19.930 & 19.930 & 19.930 & 30.084 & 30.084 & 30.084 & 50.165 & 50.165 & 50.165 & 50.630 \\ \hline
		4 & 19.787 & 19.787 & 19.787 & 29.728 & 29.728 & 29.728 & 49.552 & 49.552 & 49.552 & 49.667 \\
		\hline
	\end{tabular}
	\caption{Computed eigenvalues on \(\Omega=[0,1]^3\) using the vector Lagrange discretization \eqref{eq:lagrange-3d-primal-first-discrete}, corresponding to the three-dimensional primal formulation with respect to \(H(\dv,\Omega)\cap H_0(\curl,\Omega)\), namely \eqref{eq:3d-primal-first}.}
	\label{tab:3d-cube-lagrange}
\end{table}

\begin{table}[htbp]
	\centering
	\small
	\setlength{\tabcolsep}{3pt}
	\begin{tabular}{|c|*{10}{c|}}
		\hline
		\(L\) & \(\lambda_h^1\) & \(\lambda_h^2\) & \(\lambda_h^3\) & \(\lambda_h^4\) & \(\lambda_h^5\) & \(\lambda_h^6\) & \(\lambda_h^7\) & \(\lambda_h^8\) & \(\lambda_h^9\) & \(\lambda_h^{10}\)\\
		\hline
		1 & 22.721 & 23.634 & 25.790 & 28.346 & 36.700 & 40.694 & 56.825 & 57.045 & 63.791 & 68.227 \\ \hline
		2 & 13.030 & 18.518 & 20.541 & 22.443 & 30.541 & 33.227 & 35.167 & 38.375 & 46.014 & 49.681 \\ \hline
		3 & 4.634 & 12.735 & 15.612 & 18.327 & 21.557 & 26.625 & 28.629 & 30.702 & 31.754 & 35.672 \\ \hline
		4 & 1.341 & 3.708 & 9.613 & 9.784 & 9.933 & 14.384 & 14.404 & 16.443 & 20.643 & 21.309 \\
		\hline
	\end{tabular}
	\caption{Computed eigenvalues on \(\Omega_1\) shown in Figure~\ref{fig:3d-through-domain} using the vector Lagrange discretization \eqref{eq:lagrange-3d-primal-first-discrete}, corresponding to the three-dimensional primal formulation with respect to \(H(\dv,\Omega)\cap H_0(\curl,\Omega)\), namely \eqref{eq:3d-primal-first}.}
	\label{tab:3d-through-lagrange}
\end{table}

\begin{table}[htbp]
	\centering
	\small
	\setlength{\tabcolsep}{3pt}
	\begin{tabular}{|c|*{10}{c|}}
		\hline
		\(L\) & \(\lambda_h^1\) & \(\lambda_h^2\) & \(\lambda_h^3\) & \(\lambda_h^4\) & \(\lambda_h^5\) & \(\lambda_h^6\) & \(\lambda_h^7\) & \(\lambda_h^8\) & \(\lambda_h^9\) & \(\lambda_h^{10}\)\\
		\hline
		1 & 21.085 & 21.351 & 22.319 & 25.715 & 33.602 & 34.655 & 51.717 & 54.819 & 59.254 & 60.737 \\ \hline
		2 & 13.179 & 16.340 & 17.601 & 19.744 & 22.272 & 27.151 & 29.351 & 31.654 & 32.334 & 41.989 \\ \hline
		3 & 4.918 & 7.508 & 11.822 & 11.965 & 14.570 & 14.994 & 15.168 & 18.498 & 21.025 & 21.263 \\ \hline
		4 & 1.448 & 2.141 & 3.500 & 4.176 & 5.100 & 6.810 & 6.850 & 8.104 & 8.344 & 8.958 \\
		\hline
	\end{tabular}
	\caption{Computed eigenvalues on \(\Omega_2\) shown in Figure~\ref{fig:3d-complex-domain} using the vector Lagrange discretization \eqref{eq:lagrange-3d-primal-first-discrete}, corresponding to the three-dimensional primal formulation with respect to \(H(\dv,\Omega)\cap H_0(\curl,\Omega)\), namely \eqref{eq:3d-primal-first}.}
	\label{tab:3d-complex-lagrange}
\end{table}

\FloatBarrier
\subsection{Summary of the vector Lagrange element}
\label{app:lagrange-eigenvalues-summary}

The vector Lagrange element gives reasonable eigenvalue approximations on
convex domains, such as the square and the cube. However, this favorable
behavior does not persist on nonconvex or multiply connected domains, where
the computed spectra deviate significantly from the expected results. In
particular, on the perforated square \(\Omega_{\rm H}\) (Figure~\ref{fig:hole}), where the first Betti number is~1, the smallest computed eigenvalue at mesh level \(L=4\) is \(\lambda_h^1\approx 16.1\) (Table~\ref{tab:hole-primal-lagrange}), whereas both the proposed element and the mixed formulation in Section~\ref{sec:numer} correctly capture a simple zero eigenvalue (Tables~\ref{tab:hole-primal-new} and~\ref{tab:hole-mix}). More generally, the method does not reliably reproduce the zero eigenvalues
associated with the topology of the domain. Therefore, although the vector
Lagrange element may be effective on convex domains, it is not robust as a
general discretization for the eigenvalue problems considered here.

\section{Eigenvalue results obtained by the vector Crouzeix--Raviart element}
\label{app:cr-eigenvalues}

This appendix reports eigenvalue computations based on the vector
Crouzeix--Raviart element, to be compared with the proposed-element and mixed results in Section~\ref{sec:numer}. We remark that relevant numerical experiments concerning the Crouzeix-Raviart element were reported by \cite{Brenner.S;Cui.J;Li.F;Sung.L2008nm}. 

\subsection{Two-dimensional vector Crouzeix--Raviart discretizations and results}
\label{app:cr-eigenvalues-2d}

In two dimensions, the continuous space is
\[
\boldsymbol V_{2D}
=
H_0(\operatorname{div};\Omega)\cap H(\operatorname{rot};\Omega).
\]
Let \(\mathcal F_h^i\) and \(\mathcal F_h^b\) denote the sets of interior and
boundary edges of \(\mathcal{G}_h\), respectively. The scalar
Crouzeix--Raviart space is defined by
\[
CR_h
=
\left\{
v_h\in L^2(\Omega):
v_h|_T\in \mathbb P_1(T),\ \forall T\in\mathcal{G}_h,
\quad
\int_F \llbracket v_h\rrbracket\,ds=0,\ \forall F\in\mathcal F_h^i
\right\}.
\]
The two-dimensional vector Crouzeix--Raviart space used for
\(H_0(\operatorname{div};\Omega)\cap H(\operatorname{rot};\Omega)\) is
\[
\boldsymbol V_{h,2D}^{\rm CR}
=
\left\{
\boldsymbol v_h\in [CR_h]^2:
\int_F \boldsymbol v_h\cdot\boldsymbol n\,ds=0,
\ \forall F\in\mathcal F_h^b
\right\}.
\]

\subsubsection{Direct discretization}

The direct vector Crouzeix--Raviart discretization is to find
\[
(\lambda_h,\boldsymbol u_h)\in
\mathbb R\times \boldsymbol V_{h,2D}^{\rm CR},
\qquad
\boldsymbol u_h\ne\boldsymbol 0,
\]
such that
\begin{equation}
	a_{h,2D}(\boldsymbol u_h,\boldsymbol v_h)
	=
	\lambda_h(\boldsymbol u_h,\boldsymbol v_h),
	\qquad
	\forall\,\boldsymbol v_h\in \boldsymbol V_{h,2D}^{\rm CR},
	\label{eq:cr-2d-direct-discrete}
\end{equation}
where
\[
a_{h,2D}(\boldsymbol u_h,\boldsymbol v_h)
=
\sum_{T\in\mathcal{G}_h}
\int_T
\left(
\operatorname{div}\boldsymbol u_h\,\operatorname{div}\boldsymbol v_h
+
\operatorname{rot}\boldsymbol u_h\,\operatorname{rot}\boldsymbol v_h
\right).
\]
The corresponding two-dimensional eigenvalue results are listed below.

\begin{table}[htbp]
	\centering
	\small
	\setlength{\tabcolsep}{3pt}
	\begin{tabular}{|c|*{10}{c|}}
		\hline
		$L$ & $\lambda_h^1$ & $\lambda_h^2$ & $\lambda_h^3$ & $\lambda_h^4$ & $\lambda_h^5$ & $\lambda_h^6$ & $\lambda_h^7$ & $\lambda_h^8$ & $\lambda_h^9$ & $\lambda_h^{10}$ \\
		\hline
		1 & 19.568 & 19.568 & 38.797 & 38.797 & 76.056 & 76.056 & 94.284 & 94.284 & 94.284 & 94.284 \\ \hline
		2 & 19.697 & 19.697 & 39.309 & 39.309 & 78.271 & 78.271 & 97.629 & 97.629 & 97.629 & 97.629 \\ \hline
		3 & 19.729 & 19.729 & 39.436 & 39.436 & 78.787 & 78.787 & 98.431 & 98.431 & 98.431 & 98.431 \\ \hline
		4 & 19.737 & 19.737 & 39.468 & 39.468 & 78.915 & 78.915 & 98.630 & 98.630 & 98.630 & 98.630 \\
		\hline
	\end{tabular}
	\caption{Computed eigenvalues on \(\Omega=[0,1]^2\), using \eqref{eq:cr-2d-direct-discrete}.}
	\label{tab:sq-primal-cr}
\end{table}

\begin{table}[htbp]
	\centering
	\small
	\setlength{\tabcolsep}{3pt}
	\begin{tabular}{|c|*{10}{c|}}
		\hline
		$L$ & $\lambda_h^1$ & $\lambda_h^2$ & $\lambda_h^3$ & $\lambda_h^4$ & $\lambda_h^5$ & $\lambda_h^6$ & $\lambda_h^7$ & $\lambda_h^8$ & $\lambda_h^9$ & $\lambda_h^{10}$ \\
		\hline
		1 & 10.931 & 28.089 & 69.092 & 76.056 & 76.056 & 87.522 & 92.974 & 116.905 & 146.675 & 146.675 \\ \hline
		2 & 11.462 & 28.211 & 74.346 & 78.271 & 78.271 & 90.240 & 98.357 & 120.230 & 155.188 & 155.188 \\ \hline
		3 & 11.670 & 28.255 & 76.116 & 78.787 & 78.787 & 90.896 & 99.872 & 121.217 & 157.236 & 157.236 \\ \hline
		4 & 11.752 & 28.267 & 76.744 & 78.915 & 78.915 & 91.060 & 100.341 & 121.484 & 157.744 & 157.744 \\
		\hline
	\end{tabular}
	\caption{Computed eigenvalues on the L-shaped domain shown in Figure~\ref{fig:Lshape}, using \eqref{eq:cr-2d-direct-discrete}.}
	\label{tab:L-primal-cr}
\end{table}

\begin{table}[htbp]
	\centering
	\small
	\setlength{\tabcolsep}{3pt}
	\begin{tabular}{|c|*{10}{c|}}
		\hline
		$L$ & $\lambda_h^1$ & $\lambda_h^2$ & $\lambda_h^3$ & $\lambda_h^4$ & $\lambda_h^5$ & $\lambda_h^6$ & $\lambda_h^7$ & $\lambda_h^8$ & $\lambda_h^9$ & $\lambda_h^{10}$ \\
		\hline
		1 & 14.641 & 14.837 & 36.549 & 60.880 & 68.017 & 76.117 & 87.467 & 97.405 & 109.834 & 121.091 \\ \hline
		2 & 15.401 & 15.680 & 36.985 & 66.630 & 73.489 & 79.027 & 91.424 & 99.031 & 115.414 & 144.318 \\ \hline
		3 & 15.706 & 16.018 & 37.179 & 68.454 & 75.419 & 79.868 & 92.593 & 99.632 & 117.314 & 148.538 \\ \hline
		4 & 15.827 & 16.152 & 37.254 & 69.095 & 76.120 & 80.130 & 92.961 & 99.808 & 117.959 & 149.491 \\
		\hline
	\end{tabular}
	\caption{Computed eigenvalues on the square domain with an interior hole shown in Figure~\ref{fig:hole}, using \eqref{eq:cr-2d-direct-discrete}.}
	\label{tab:hole-primal-cr}
\end{table}

\subsubsection{\(\nabla\)-formulation}

The \(\nabla\)-formulation with the vector Crouzeix--Raviart element is to find
\[
(\lambda_h,\boldsymbol u_h)\in
\mathbb R\times \boldsymbol V_{h,2D}^{\rm CR},
\qquad
\boldsymbol u_h\ne\boldsymbol 0,
\]
such that
\begin{equation}
	b_{h,2D}(\boldsymbol u_h,\boldsymbol v_h)
	=
	\lambda_h(\boldsymbol u_h,\boldsymbol v_h),
	\qquad
	\forall\,\boldsymbol v_h\in \boldsymbol V_{h,2D}^{\rm CR},
	\label{eq:cr-2d-nabla-discrete}
\end{equation}
where
\[
b_{h,2D}(\boldsymbol u_h,\boldsymbol v_h)
=
\sum_{T\in\mathcal{G}_h}
\int_T \nabla\boldsymbol u_h:\nabla\boldsymbol v_h .
\]
The corresponding two-dimensional eigenvalue results are listed below.

\begin{table}[htbp]
	\centering
	\small
	\setlength{\tabcolsep}{3pt}
	\begin{tabular}{|c|*{10}{c|}}
		\hline
		$L$ & $\lambda_h^1$ & $\lambda_h^2$ & $\lambda_h^3$ & $\lambda_h^4$ & $\lambda_h^5$ & $\lambda_h^6$ & $\lambda_h^7$ & $\lambda_h^8$ & $\lambda_h^9$ & $\lambda_h^{10}$ \\
		\hline
		1 & 9.827 & 9.827 & 19.061 & 19.061 & 38.797 & 38.797 & 46.236 & 46.236 & 46.236 & 46.236 \\ \hline
		2 & 9.859 & 9.859 & 19.570 & 19.570 & 39.309 & 39.309 & 48.575 & 48.575 & 48.575 & 48.575 \\ \hline
		3 & 9.867 & 9.867 & 19.697 & 19.697 & 39.436 & 39.436 & 49.155 & 49.155 & 49.155 & 49.155 \\ \hline
		4 & 9.869 & 9.869 & 19.729 & 19.729 & 39.468 & 39.468 & 49.300 & 49.300 & 49.300 & 49.300 \\
		\hline
	\end{tabular}
	\caption{Computed eigenvalues on \(\Omega=[0,1]^2\), using \eqref{eq:cr-2d-nabla-discrete}.}
	\label{tab:sq-nabla-cr}
\end{table}

\begin{table}[htbp]
	\centering
	\small
	\setlength{\tabcolsep}{3pt}
	\begin{tabular}{|c|*{10}{c|}}
		\hline
		$L$ & $\lambda_h^1$ & $\lambda_h^2$ & $\lambda_h^3$ & $\lambda_h^4$ & $\lambda_h^5$ & $\lambda_h^6$ & $\lambda_h^7$ & $\lambda_h^8$ & $\lambda_h^9$ & $\lambda_h^{10}$ \\
		\hline
		1 & 12.358 & 12.358 & 38.797 & 38.797 & 42.589 & 42.589 & 52.532 & 52.532 & 68.016 & 68.016 \\ \hline
		2 & 12.996 & 12.996 & 39.309 & 39.309 & 44.472 & 44.472 & 57.086 & 57.086 & 76.242 & 76.242 \\ \hline
		3 & 13.402 & 13.402 & 39.436 & 39.436 & 45.075 & 45.075 & 58.794 & 58.794 & 78.280 & 78.280 \\ \hline
		4 & 13.664 & 13.664 & 39.468 & 39.468 & 45.304 & 45.304 & 59.602 & 59.602 & 78.788 & 78.788 \\
		\hline
	\end{tabular}
	\caption{Computed eigenvalues on the L-shaped domain shown in Figure~\ref{fig:Lshape}, using \eqref{eq:cr-2d-nabla-discrete}.}
	\label{tab:L-nabla-cr}
\end{table}

\begin{table}[htbp]
	\centering
	\small
	\setlength{\tabcolsep}{3pt}
	\begin{tabular}{|c|*{10}{c|}}
		\hline
		$L$ & $\lambda_h^1$ & $\lambda_h^2$ & $\lambda_h^3$ & $\lambda_h^4$ & $\lambda_h^5$ & $\lambda_h^6$ & $\lambda_h^7$ & $\lambda_h^8$ & $\lambda_h^9$ & $\lambda_h^{10}$ \\
		\hline
		1 & 12.797 & 12.797 & 16.934 & 16.934 & 40.379 & 40.379 & 43.852 & 43.852 & 52.459 & 52.459 \\ \hline
		2 & 13.765 & 13.765 & 19.564 & 19.564 & 41.688 & 41.688 & 46.869 & 46.869 & 57.918 & 57.918 \\ \hline
		3 & 14.377 & 14.377 & 21.333 & 21.333 & 42.333 & 42.333 & 48.195 & 48.195 & 60.565 & 60.565 \\ \hline
		4 & 14.763 & 14.763 & 22.515 & 22.515 & 42.699 & 42.699 & 48.907 & 48.907 & 62.059 & 62.059 \\
		\hline
	\end{tabular}
	\caption{Computed eigenvalues on the square domain with an interior hole shown in Figure~\ref{fig:hole}, using \eqref{eq:cr-2d-nabla-discrete}.}
	\label{tab:hole-nabla-cr}
\end{table}

\subsection{Three-dimensional vector Crouzeix--Raviart discretizations and results}
\label{app:cr-eigenvalues-3d}

In three dimensions, the continuous space is
\[
\boldsymbol V_{3D}
=
H(\operatorname{div};\Omega)\cap H_0(\operatorname{curl};\Omega).
\]
Let \(\mathcal F_h^i\) and \(\mathcal F_h^b\) denote the sets of interior and
boundary faces of \(\mathcal{G}_h\), respectively. The scalar
Crouzeix--Raviart space is defined by
\[
CR_h
=
\left\{
v_h\in L^2(\Omega):
v_h|_T\in \mathbb P_1(T),\ \forall T\in\mathcal{G}_h,
\quad
\int_F \llbracket v_h\rrbracket\,ds=0,\ \forall F\in\mathcal F_h^i
\right\}.
\]
The three-dimensional vector Crouzeix--Raviart space used for
\(H(\operatorname{div};\Omega)\cap H_0(\operatorname{curl};\Omega)\) is
\[
\boldsymbol V_{h,3D}^{\rm CR}
=
\left\{
\boldsymbol v_h\in [CR_h]^3:
\int_F \boldsymbol v_h\times\boldsymbol n\,ds=\boldsymbol 0,
\ \forall F\in\mathcal F_h^b
\right\}.
\]

\subsubsection{Direct discretization}

The direct vector Crouzeix--Raviart discretization is to find
\[
(\lambda_h,\boldsymbol u_h)\in
\mathbb R\times \boldsymbol V_{h,3D}^{\rm CR},
\qquad
\boldsymbol u_h\ne\boldsymbol 0,
\]
such that
\begin{equation}
	a_{h,3D}(\boldsymbol u_h,\boldsymbol v_h)
	=
	\lambda_h(\boldsymbol u_h,\boldsymbol v_h),
	\qquad
	\forall\,\boldsymbol v_h\in \boldsymbol V_{h,3D}^{\rm CR},
	\label{eq:cr-3d-direct-discrete}
\end{equation}
where
\[
a_{h,3D}(\boldsymbol u_h,\boldsymbol v_h)
=
\sum_{T\in\mathcal{G}_h}
\int_T
\left(
\operatorname{div}\boldsymbol u_h\,\operatorname{div}\boldsymbol v_h
+
(\nabla\times\boldsymbol u_h)\cdot(\nabla\times\boldsymbol v_h)
\right).
\]
The corresponding three-dimensional eigenvalue results are listed below.

\begin{table}[htbp]
	\centering
	\small
	\setlength{\tabcolsep}{3pt}
	\begin{tabular}{|c|*{10}{c|}}
		\hline
		\(L\) & \(\lambda_h^1\) & \(\lambda_h^2\) & \(\lambda_h^3\) & \(\lambda_h^4\) & \(\lambda_h^5\) & \(\lambda_h^6\) & \(\lambda_h^7\) & \(\lambda_h^8\) & \(\lambda_h^9\) & \(\lambda_h^{10}\)\\
		\hline
		1 & 23.835 & 23.835 & 26.341 & 36.609 & 36.609 & 46.154 & 53.202 & 53.202 & 60.000 & 65.876 \\ \hline
		2 & 20.434 & 23.349 & 23.349 & 28.762 & 28.762 & 30.438 & 32.339 & 32.339 & 33.718 & 33.718 \\ \hline
		3 & 22.066 & 26.198 & 26.198 & 35.663 & 35.663 & 51.299 & 52.055 & 52.055 & 59.355 & 59.355 \\ \hline
		4 & 22.503 & 26.994 & 26.994 & 37.274 & 37.274 & 58.071 & 58.830 & 58.830 & 65.948 & 65.948 \\
		\hline
	\end{tabular}
	\caption{Computed eigenvalues on \(\Omega=[0,1]^3\), using \eqref{eq:cr-3d-direct-discrete}.}
	\label{tab:3d-cube-cr-primal-first}
\end{table}

\begin{table}[htbp]
	\centering
	\small
	\setlength{\tabcolsep}{3pt}
	\begin{tabular}{|c|*{10}{c|}}
		\hline
		\(L\) & \(\lambda_h^1\) & \(\lambda_h^2\) & \(\lambda_h^3\) & \(\lambda_h^4\) & \(\lambda_h^5\) & \(\lambda_h^6\) & \(\lambda_h^7\) & \(\lambda_h^8\) & \(\lambda_h^9\) & \(\lambda_h^{10}\)\\
		\hline
		1 & 9.423 & 19.933 & 21.729 & 29.581 & 34.689 & 39.819 & 43.040 & 45.321 & 45.897 & 47.059 \\ \hline
		2 & 11.492 & 21.160 & 23.390 & 36.338 & 43.749 & 46.256 & 52.659 & 55.647 & 58.888 & 59.267 \\ \hline
		3 & 12.292 & 21.620 & 24.007 & 36.787 & 48.429 & 48.496 & 55.216 & 57.656 & 62.524 & 62.644 \\ \hline
		4 & 12.658 & 22.915 & 23.970 & 37.775 & 49.861 & 50.241 & 58.615 & 59.734 & 62.799 & 63.228 \\
		\hline
	\end{tabular}
	\caption{Computed eigenvalues on \(\Omega_1\) shown in Figure~\ref{fig:3d-through-domain}, using \eqref{eq:cr-3d-direct-discrete}.}
	\label{tab:3d-through-cr-primal-first}
\end{table}

\begin{table}[htbp]
	\centering
	\small
	\setlength{\tabcolsep}{3pt}
	\begin{tabular}{|c|*{10}{c|}}
		\hline
		\(L\) & \(\lambda_h^1\) & \(\lambda_h^2\) & \(\lambda_h^3\) & \(\lambda_h^4\) & \(\lambda_h^5\) & \(\lambda_h^6\) & \(\lambda_h^7\) & \(\lambda_h^8\) & \(\lambda_h^9\) & \(\lambda_h^{10}\)\\
		\hline
		1 & 9.102 & 9.563 & 17.980 & 19.241 & 27.206 & 30.054 & 31.158 & 32.072 & 37.370 & 37.417 \\ \hline
		2 & 11.143 & 11.218 & 19.042 & 20.559 & 29.944 & 31.842 & 39.702 & 41.147 & 41.472 & 42.476 \\ \hline
		3 & 11.882 & 12.043 & 19.541 & 21.175 & 31.164 & 32.585 & 41.897 & 42.714 & 43.185 & 45.981 \\ \hline
		4 & 12.237 & 12.403 & 20.551 & 21.201 & 31.961 & 33.436 & 43.797 & 43.842 & 44.084 & 47.260 \\
		\hline
	\end{tabular}
	\caption{Computed eigenvalues on \(\Omega_2\) shown in Figure~\ref{fig:3d-complex-domain}, using \eqref{eq:cr-3d-direct-discrete}.}
	\label{tab:3d-complex-cr-primal-first}
\end{table}

\subsubsection{\(\nabla\)-formulation}

The \(\nabla\)-formulation with the vector Crouzeix--Raviart element is to find
\[
(\lambda_h,\boldsymbol u_h)\in
\mathbb R\times \boldsymbol V_{h,3D}^{\rm CR},
\qquad
\boldsymbol u_h\ne\boldsymbol 0,
\]
such that
\begin{equation}
	b_{h,3D}(\boldsymbol u_h,\boldsymbol v_h)
	=
	\lambda_h(\boldsymbol u_h,\boldsymbol v_h),
	\qquad
	\forall\,\boldsymbol v_h\in \boldsymbol V_{h,3D}^{\rm CR},
	\label{eq:cr-3d-nabla-discrete}
\end{equation}
where
\[
b_{h,3D}(\boldsymbol u_h,\boldsymbol v_h)
=
\sum_{T\in\mathcal{G}_h}
\int_T \nabla\boldsymbol u_h:\nabla\boldsymbol v_h .
\]
The corresponding three-dimensional eigenvalue results are listed below.

\begin{table}[htbp]
	\centering
	\small
	\setlength{\tabcolsep}{3pt}
	\begin{tabular}{|c|*{10}{c|}}
		\hline
		\(L\) & \(\lambda_h^1\) & \(\lambda_h^2\) & \(\lambda_h^3\) & \(\lambda_h^4\) & \(\lambda_h^5\) & \(\lambda_h^6\) & \(\lambda_h^7\) & \(\lambda_h^8\) & \(\lambda_h^9\) & \(\lambda_h^{10}\)\\
		\hline
		1 & 16.917 & 16.917 & 16.917 & 25.090 & 25.090 & 25.090 & 28.268 & 28.268 & 28.268 & 28.935 \\ \hline
		2 & 19.000 & 19.000 & 19.000 & 28.383 & 28.383 & 28.383 & 43.302 & 43.302 & 43.967 & 43.967 \\ \hline
		3 & 19.552 & 19.552 & 19.552 & 29.295 & 29.295 & 29.295 & 47.774 & 47.774 & 47.967 & 47.967 \\ \hline
		4 & 19.692 & 19.692 & 19.692 & 29.530 & 29.530 & 29.530 & 48.951 & 48.951 & 49.001 & 49.001 \\
		\hline
	\end{tabular}
	\caption{Computed eigenvalues on \(\Omega=[0,1]^3\), using \eqref{eq:cr-3d-nabla-discrete}.}
	\label{tab:3d-cube-nabla-cr-first}
\end{table}

\begin{table}[htbp]
	\centering
	\small
	\setlength{\tabcolsep}{3pt}
	\begin{tabular}{|c|*{10}{c|}}
		\hline
		\(L\) & \(\lambda_h^1\) & \(\lambda_h^2\) & \(\lambda_h^3\) & \(\lambda_h^4\) & \(\lambda_h^5\) & \(\lambda_h^6\) & \(\lambda_h^7\) & \(\lambda_h^8\) & \(\lambda_h^9\) & \(\lambda_h^{10}\)\\
		\hline
		1 & 21.648 & 21.648 & 25.798 & 25.798 & 32.095 & 41.235 & 44.819 & 44.819 & 46.230 & 46.230 \\ \hline
		2 & 23.300 & 23.300 & 28.924 & 28.924 & 36.188 & 45.870 & 50.100 & 50.100 & 51.594 & 51.594 \\ \hline
		3 & 24.116 & 24.116 & 30.926 & 30.926 & 37.569 & 47.391 & 51.799 & 51.799 & 53.391 & 53.391 \\ \hline
		4 & 24.570 & 24.570 & 32.223 & 32.223 & 38.027 & 47.884 & 52.443 & 52.443 & 54.096 & 54.096 \\
		\hline
	\end{tabular}
	\caption{Computed eigenvalues on \(\Omega_1\) shown in Figure~\ref{fig:3d-through-domain}, using \eqref{eq:cr-3d-nabla-discrete}.}
	\label{tab:3d-through-nabla-cr-first}
\end{table}

\begin{table}[htbp]
	\centering
	\small
	\setlength{\tabcolsep}{3pt}
	\begin{tabular}{|c|*{10}{c|}}
		\hline
		\(L\) & \(\lambda_h^1\) & \(\lambda_h^2\) & \(\lambda_h^3\) & \(\lambda_h^4\) & \(\lambda_h^5\) & \(\lambda_h^6\) & \(\lambda_h^7\) & \(\lambda_h^8\) & \(\lambda_h^9\) & \(\lambda_h^{10}\)\\
		\hline
		1 & 29.254 & 30.565 & 31.455 & 32.531 & 41.934 & 43.680 & 44.674 & 45.642 & 51.309 & 54.097 \\ \hline
		2 & 34.945 & 37.077 & 37.370 & 39.956 & 53.783 & 54.480 & 55.317 & 58.571 & 59.404 & 63.649 \\ \hline
		3 & 38.923 & 41.335 & 41.482 & 44.644 & 59.485 & 59.846 & 64.118 & 65.595 & 68.301 & 69.247 \\ \hline
		4 & 41.618 & 44.104 & 44.210 & 47.533 & 62.806 & 63.072 & 67.109 & 71.168 & 71.907 & 72.432 \\
		\hline
	\end{tabular}
	\caption{Computed eigenvalues on \(\Omega_2\) shown in Figure~\ref{fig:3d-complex-domain}, using \eqref{eq:cr-3d-nabla-discrete}.}
	\label{tab:3d-complex-nabla-cr-first}
\end{table}

\FloatBarrier
\subsection{Summary of the vector Crouzeix--Raviart element}

The direct primal discretization using the vector Crouzeix--Raviart element
does not yield satisfactory eigenvalue approximations for the present problem.
The computed spectra show noticeable discrepancies from those obtained by the
proposed element and the mixed formulation. In contrast, the
\(\nabla\)-formulation with the vector Crouzeix--Raviart element gives
reasonable approximations on convex domains. However, this favorable behavior
does not persist uniformly on nonconvex or topologically more complicated
domains. For example, on \(\Omega_{\rm H}\) neither the direct discretization (Table~\ref{tab:hole-primal-cr}) nor the \(\nabla\)-formulation (Table~\ref{tab:hole-nabla-cr}) produces a zero eigenvalue, in contrast to Section~\ref{sec:numer}. These observations suggest that, although the \(\nabla\)-formulation
may be effective in the convex case, the direct use of the vector
Crouzeix--Raviart element in the primal formulation is not reliable as a
general discretization for the present eigenvalue problem.

%

%
%
%
\section{Illustration of basis functions of $\fW^{*,\rm nc}_h\Lambda^k$ in two and three dimensions}
\label{sec:ncbasisfunction}

This appendix illustrates the basis functions of nonconforming Whitney spaces $\fW^{*,\rm nc}_h\Lambda^k$ defined in \cite{Zhang.S2026IMA} used in the adjoint continuity constraints~\eqref{eq:deffems}. The two- and three-dimensional realizations are below; note that, the choice of the two cells of the supports of the basis functions is not unique, and they are not necessarily adjacent to each other. In two dimensions, $\mathbb{RT}^{\rm nc}_h$ coincides with the adjoint partner $\fW^{*,\rm nc}_{h0}\Lambda^1$ under the vector proxy used in~\eqref{eq:deffems2d}.
\subsection{$\fW^{*,\rm nc}_h\Lambda^1$ in 2D revisited}

Let $\mathbb{V}^1_h$ denote the continuous piecewise linear element space, and $\uV^{\rm RT}_h$ denote the Raviart-Thomas \cite{Raviart.P;Thomas.J1977} element space of lowest degree on $\mathcal{G}_h$. Denote $\mathbb{V}^1_{h0}:=\mathbb{V}^1_h\cap H^1_0(\Omega)$ and $\uV^{\rm RT}_{h0}:=\uV^{\rm RT}_h\cap H_0(\dv,\Omega)$.

On a triangle $T$, denote the space of the lowest-degree Raviart-Thomas shape functions by $\mathbb{RT}(T):={\rm span}\left\{\ualpha+\beta\ux:\ualpha\in\mathbb{R}^2,\beta\in\mathbb{R}\right\}$. Then
\begin{equation}\label{eq:nabladivdual}
	\mathcal{R}(\dv,\mathbb{RT}(T))=\mathbb{R}=\mathcal{N}(\nabla,P_1(T)),\ \ \ \mbox{and}\  \ \mathcal{N}(\dv,\mathbb{RT}(T))=\mathbb{R}^2=\mathcal{R}(\nabla,P_1(T)).
\end{equation}
Denote $\displaystyle\mathbb{RT}(\mathcal{G}_h):=\bigoplus_{T\in\mathcal{G}_h}E_T^\Omega\mathbb{RT}(T)$. We define the nonconforming finite element spaces
\begin{equation}\label{eq:2drtnc}
	\mathbb{RT}^{\rm nc}_h:=\left\{\utau{}_h\in \mathbb{RT}(\mathcal{G}_h):\sum_{T\in\mathcal{G}_h}(\utau{}_h,\nabla \xv_h)_T+(\dv\utau{}_h,\xv_h)_T=0,\ \forall\,\xv_h\in \mathbb{V}^1_{h0}\right\},
\end{equation}
and
\begin{equation}\label{eq:2drt0nc}
	\mathbb{RT}^{\rm nc}_{h0}:=\left\{\utau{}_h\in \mathbb{RT}(\mathcal{G}_h):\sum_{T\in\mathcal{G}_h}(\utau{}_h,\nabla \xv_h)_T+(\dv\utau{}_h,\xv_h)_T=0,\ \forall\,\xv_h\in \mathbb{V}^1_h\right\}.
\end{equation}
Figure~\ref{fig:globalbasis} shows a representative global basis function of $\mathbb{RT}^{\rm nc}_h$. 
\begin{figure}[htbp]
	\begin{center}
		\includegraphics[width=0.45\textwidth]{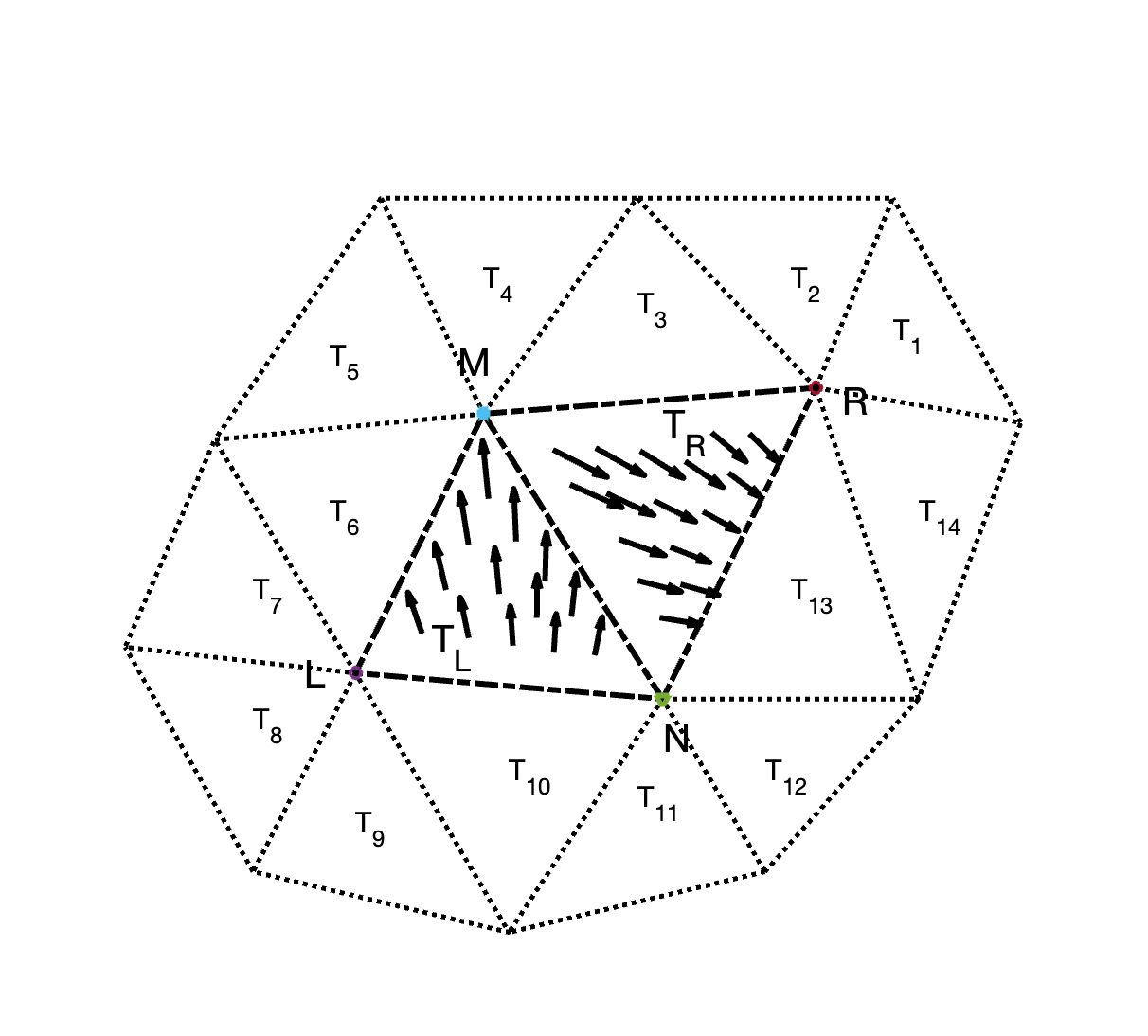}
	\end{center}
	\caption{Profile of a representative global basis function in $\mathbb{RT}^{\rm nc}_h$, supported on two cells $T_L$ and $T_R$. }\label{fig:globalbasis}
\end{figure}
%
%

%
%
\subsection{$\fW^{*,\rm nc}_h\Lambda^2$ in 3D revisited}

Let $\Omega\subset\mathbb{R}^3$ be a polyhedral domain, and let
$\mathcal{G}_h$ be a tetrahedral partition of $\Omega$. Under the
three-dimensional vector proxy, the lowest-order Whitney $1$-form space is
identified with the lowest-order Nedelec edge element space of the first kind.
We denote this space by $\boldsymbol V_h^{\rm Ned}$, and set
\[
\boldsymbol V_{h0}^{\rm Ned}
:=
\boldsymbol V_h^{\rm Ned}\cap H_0(\curl,\Omega).
\]

On a tetrahedron $T$, the corresponding local Nedelec shape-function space is
\[
\boldsymbol{\rm Ned}(T)
=
\left\{
\boldsymbol a+\boldsymbol b\times\boldsymbol x:
\boldsymbol a,\boldsymbol b\in\mathbb{R}^3
\right\}.
\]
Meanwhile, $\boldsymbol{\rm Ned}(T)$ represents $\mathcal{P}^{*,-}_1\Lambda^2(T)$ under the vector proxy, and $\odelta_2$ is essentially $\curl$. Namely, $\fW^{*,\rm nc}_h\Lambda^2$ is essentially a nonconforming $H(\curl)$ finite element space. In the sequel, we denote it by $\mathbb{Ned}^{\rm nc}_h$, and define it by
\begin{equation}\label{eq:3d-dual-nc-2form}
	\mathbb{Ned}^{\rm nc}_h
	:=
	\left\{
	\fmu_h\in \mathbb{Ned}(\mathcal{G}_h):
	\sum_{T\in\mathcal{G}_h}
	\left[
	(\curl\boldsymbol\mu_h,\boldsymbol\tau_h)_T
	-
	(\boldsymbol\mu_h,\curl\boldsymbol\tau_h)_T
	\right]
	=0,\ 
	\forall\,\boldsymbol\tau_h\in \boldsymbol V_{h0}^{\rm Ned}
	\right\}.
\end{equation}
Similarly, the space with the homogeneous boundary condition is defined by
\begin{equation}\label{eq:3d-dual-nc-2form-zero}
	\mathbb{Ned}^{\rm nc}_{h0}
	:=
	\left\{
	\fmu_h\in \mathbb{Ned}(\mathcal{G}_h):
	\sum_{T\in\mathcal{G}_h}
	\left[
	(\curl\boldsymbol\mu_h,\boldsymbol\tau_h)_T
	-
	(\boldsymbol\mu_h,\curl\boldsymbol\tau_h)_T
	\right]
	=0,\ 
	\forall\,\boldsymbol\tau_h\in \boldsymbol V_h^{\rm Ned}
	\right\}.
\end{equation}

Figure~\ref{fig:nc3D} shows a typical global basis function in
$\mathbb{Ned}^{\rm nc}_h$, whose support consists of two adjacent
tetrahedra sharing an edge. Again, we note that the two cells in the support are not necessarily adjacent to each other. 

\begin{figure}[htbp]
	\centering
	\includegraphics[width=\textwidth]{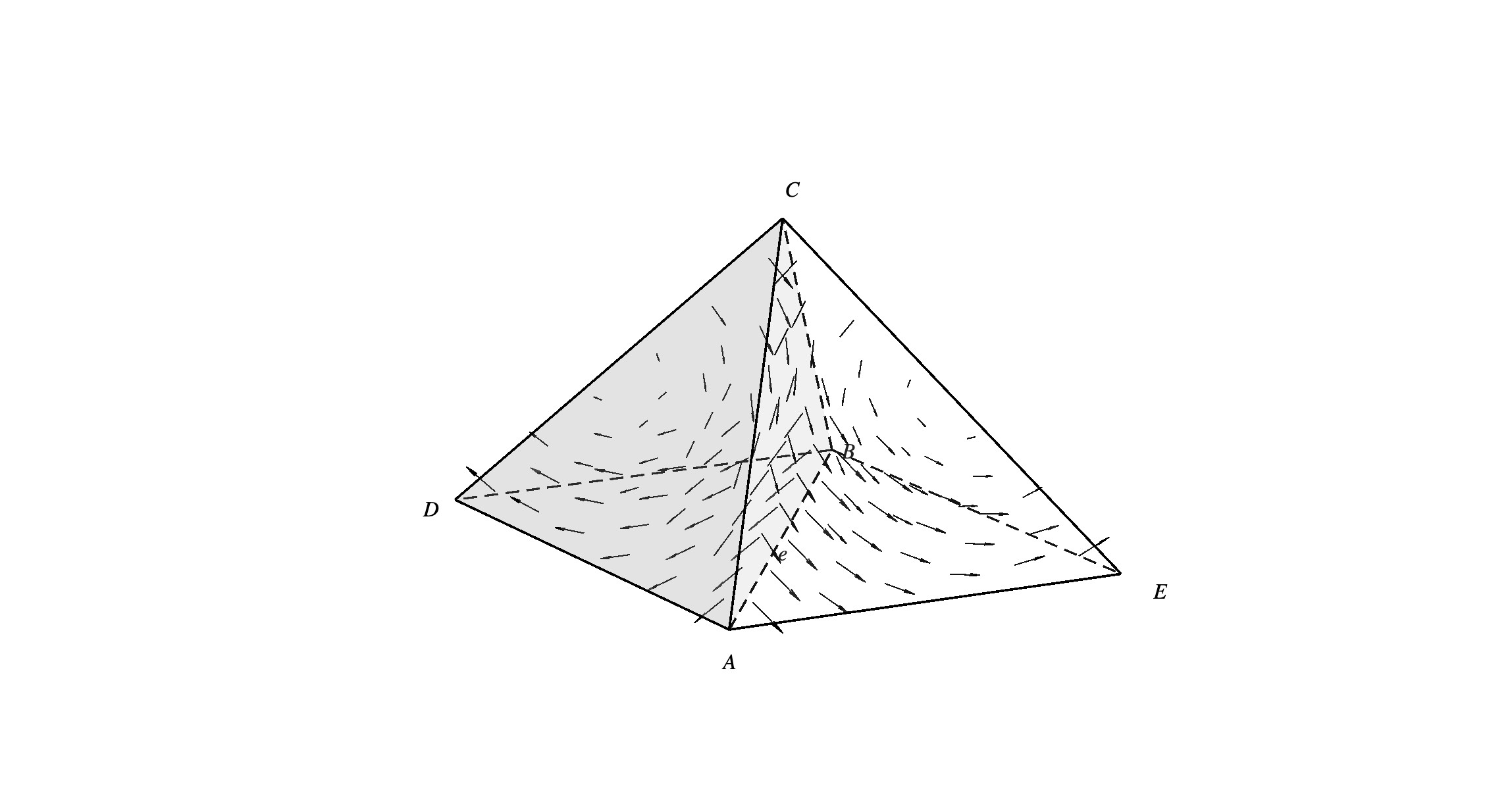}
	\caption{
		A global basis function of the three-dimensional dual nonconforming space
		$\mathbb{Ned}^{\rm nc}_h$. The support consists of two adjacent
		tetrahedra sharing the edge $e$.
	}
	\label{fig:nc3D}
\end{figure}

%

%
%
\section{Basis functions of the two-dimensional nonconforming space \eqref{eq:deffems2d} for $H(\rot,\Omega)\cap H_0(\dv,\Omega)$ }
\label{sec:app:basis-2d}

This section reports explicit cell-wise and patch-wise basis functions for $\fV^{\rm r\mathring{d}}_h=\fV^{\od\cap\mathring{\odelta}}_h\Lambda^1$ in two dimensions, supporting the proof of Theorem~\ref{thm:csbasis}.


Let
\[
A=\int_T \widetilde{x}^{\,2}\,d\boldsymbol{x},
\qquad
B=\int_T \widetilde{y}^{\,2}\,d\boldsymbol{x},
\qquad
C=\int_T \widetilde{x}\widetilde{y}\,d\boldsymbol{x}.
\]

Define
\[
\boldsymbol M_T(\widetilde x,\widetilde y)
=
\frac{1}{2(A+B)}
\begin{bmatrix}
	\widetilde x^{\,2}-\widetilde y^{\,2}
	+\dfrac{3B-A}{|T|}
	&
	-\dfrac{2C}{|T|}
	\\[2mm]
	-\dfrac{2C}{|T|}
	&
	-\widetilde x^{\,2}+\widetilde y^{\,2}
	+\dfrac{3A-B}{|T|}
\end{bmatrix},
\]
and
\[
\boldsymbol P_T(\widetilde x,\widetilde y)
=
\frac{1}{4(A+B)}
\begin{bmatrix}
	-\dfrac{2C}{|T|}
	&
	-\widetilde x^{\,2}+\widetilde y^{\,2}
	+\dfrac{3A-B}{|T|}
	\\[2mm]
	-\widetilde x^{\,2}+\widetilde y^{\,2}
	+\dfrac{A-3B}{|T|}
	&
	\dfrac{2C}{|T|}
\end{bmatrix}.
\]
Then the six local basis functions of \(\pddeL^1(T)\) are given by
\begin{equation}
	\begin{aligned}
		\fvarrho_i^T
		&=
		\frac{1}{2|T|}
		\binom{\widetilde x}{\widetilde y}
		+
		\boldsymbol M_T(\widetilde x,\widetilde y)
		\binom{\widetilde x_i}{\widetilde y_i},
		\\[2mm]
		\fvarsigma_i^T
		&=
		-\frac{1}{2|T|}
		\binom{\widetilde y}{-\widetilde x}
		+
		\boldsymbol P_T(\widetilde x,\widetilde y)
		\binom{\widetilde x_i}{\widetilde y_i},
		\qquad i=1,2,3.
	\end{aligned}
	\label{eq:local-basis-2d}
\end{equation}

Then, 
$$
P_{\rm rd}(T)={\rm span}\{\fvarrho^T_i,\fvarsigma^T_i\}_{i=1}^3
$$
and
\begin{equation}
	(\rot\fvarrho^T_i,1-2\lambda_j)_T-(\fvarrho^T_i,\curl (1-2\lambda_j))_T=0,\ \ \ (\dv\fvarrho^T_i,\lambda_j)_T+(\fvarrho^T_i,\nabla \lambda_j)_T=\delta_{ij},\ i,j=1:3
\end{equation}
and 
\begin{equation}
	(\rot\fvarsigma^T_i,1-2\lambda_j)_T-(\fvarsigma^T_i,\curl (1-2\lambda_j))_T=\delta_{ij},\ \ \ (\dv\fvarsigma^T_i,\lambda_j)_T+(\fvarsigma^T_i,\nabla \lambda_j)_T=0,\ i,j=1:3
\end{equation}

\begin{figure}[htbp]
	\centering
	
	\begin{subfigure}{0.48\textwidth}
		\centering
		\includegraphics[width=\textwidth]{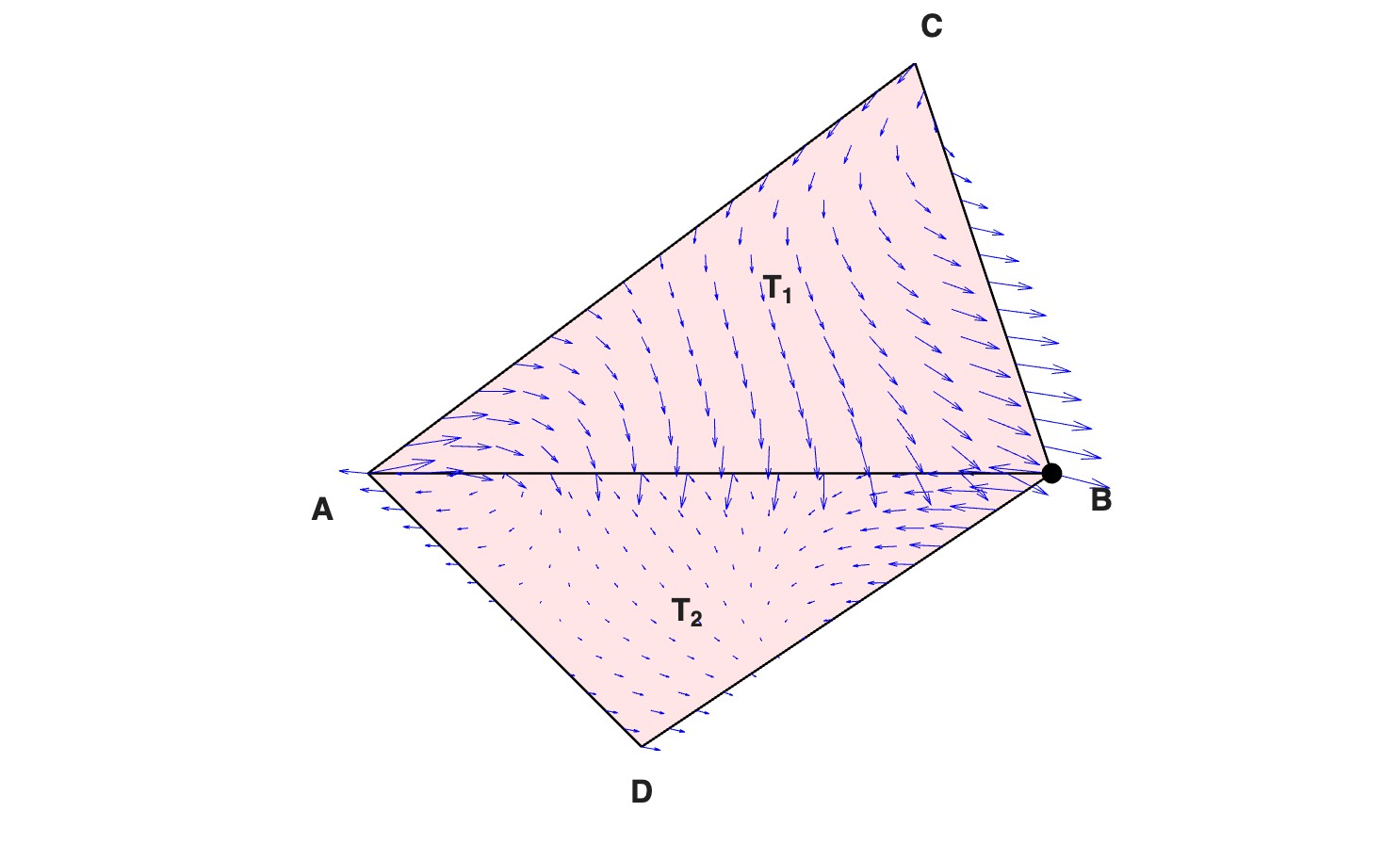}
		\caption{Point-associated global basis function $\fvarrho^T$}
		\label{fig:2Dpointbasis}
	\end{subfigure}
	\hfill
	\begin{subfigure}{0.48\textwidth}
		\centering
		\includegraphics[width=\textwidth]{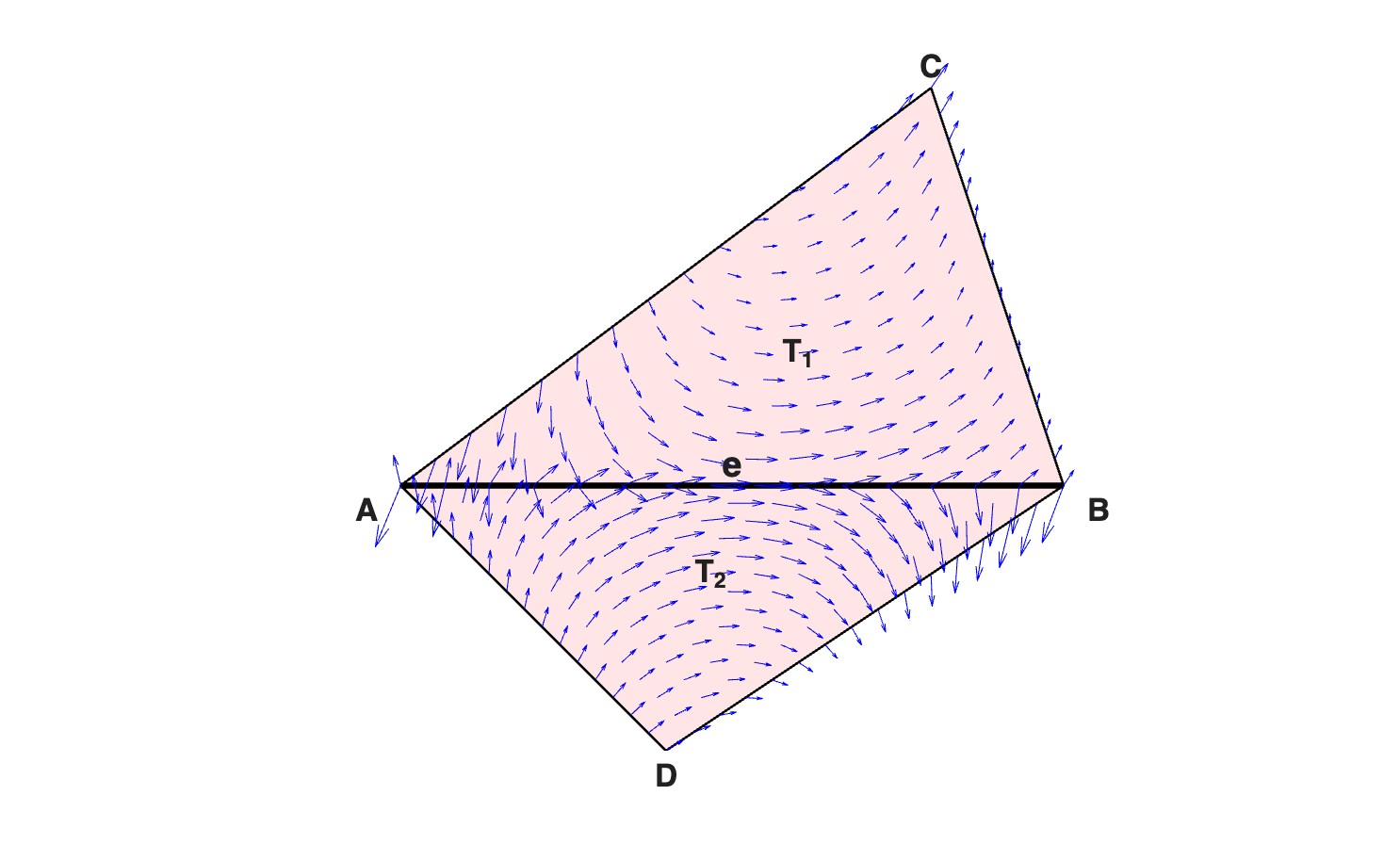}
		\caption{Edge-associated global basis functioin $\fvarsigma^T$}
		\label{fig:2Dedgebasis}
	\end{subfigure}
	
	\label{fig:2D-basis}
\end{figure}

\begin{figure}[htbp]
	\centering
	\includegraphics[width=\textwidth]{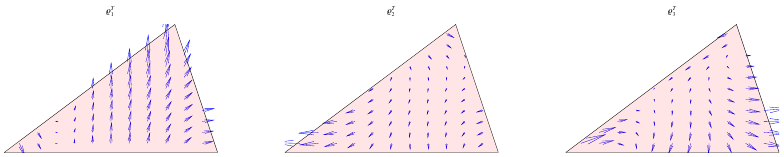}
	\caption{Local basis functions \(\fvarrho^T_1\)--\(\fvarrho^T_3\).}
	\label{fig:local-basis-varrho}
\end{figure}

\begin{figure}[htbp]
	\centering
	\includegraphics[width=\textwidth]{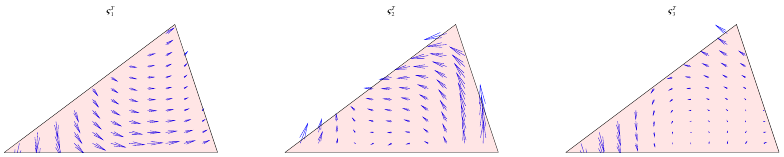}
	\caption{Local basis functions \(\fvarsigma^T_1\)--\(\fvarsigma^T_3\).}
	\label{fig:local-basis-varsigma}
\end{figure}

\begin{remark}
	The explicit basis functions $\fvarrho^T_i$ and $\fvarsigma^T_i$ can be written out by these 3 steps. 
	
	First, by Remark~\ref{rem:kernel=range}, we can write out
	$\omega_1 = \binom {\widetilde {x}} {\widetilde {y}}$,$\omega_2 = \binom {\widetilde {y}} {- \widetilde {x}}$, $\omega_3 = \binom {\widetilde {x} ^ {2} - \widetilde {y} ^ {2} + M} {0}$, $\omega_4=\binom {0} {\widetilde {x} ^ {2} - \widetilde {y} ^ {2} + M}$, $\omega_5=\binom {1} {0}$, $\omega_6=\binom {0} {1}$ and $(s_1,\eta_1)=\frac{1}{2|T|}\binom {1} {0}$, $(s_2,\eta_2)=-\frac{1}{2|T|}\binom {0} {1}$, $(s_3,\eta_3)=\frac{1}{2(A+B)}\binom {\widetilde {x}} {\widetilde {y}}$, $(s_4,\eta_4)=-\frac{1}{2(A+B)}\binom {\widetilde {y}} {-\widetilde {x}}$, $(s_5,\eta_5)=\frac{1}{|T|}(\frac {1}{2} \binom {\tilde {x}} {- \tilde {y}} + \frac {B - A}{2 (A + B)} \binom {\tilde {x}} {\tilde {y}} + \frac {C}{A + B} \binom {\tilde {y}} {- \tilde {x}})$, $(s_6,\eta_6)=\frac{1}{|T|}(\frac {1}{2} \binom {\tilde {y}} {\tilde {x}} + \frac {C}{A + B} \binom {\tilde {x}} {\tilde {y}} + \frac {A - B}{2 (A + B)} \binom {\tilde {y}} {- \tilde {x}})$, so that
	\begin{equation}
		(\dv\omega_i,s_j)_T+(\omega_i,\nabla s_j)_T+(\rot\omega_i,\eta_j)_T-(\omega_i,\curl\eta_j)_T=\delta_{i,j}.
	\end{equation}
	Set
	$$
	\left[
	\begin{array}{c}
		(s_1,\eta_1)
		\\
		(s_2,\eta_2)
		\\
		(s_3,\eta_3)
		\\
		(s_4,\eta_4)
		\\
		(s_5,\eta_5)
		\\
		(s_6,\eta_6)
	\end{array}
	\right]
	=
	\overline{M}
	\left[
	\begin{array}{c}
		(\lambda_1,0)
		\\
		(\lambda_2,0)
		\\
		(\lambda_3,0)
		\\
		(0,1-2\lambda_1)
		\\
		(0,1-2\lambda_2)
		\\
		(0,1-2\lambda_3)
	\end{array}
	\right]
	=
	\overline{M}\left[
	\begin{array}{c}
		(\tilde s_1,\tilde\eta_1)
		\\
		(\tilde s_2,\tilde\eta_2)
		\\
		(\tilde s_3,\tilde\eta_3)
		\\
		(\tilde s_4,\tilde\eta_4)
		\\
		(\tilde s_5,\tilde\eta_5)
		\\
		(\tilde s_6,\tilde\eta_6)
	\end{array}
	\right]
	$$
	and
	$$
	(\tilde\omega_1,\tilde\omega_2,\dots,\tilde\omega_6)=(\omega_1,\omega_2,\dots,\omega_6)\overline{M}.
	$$
	Then
	$$
	(\dv\tilde\omega_i,\tilde s_j)_T+(\tilde \omega_i,\nabla \tilde s_j)_T+(\rot\tilde \omega_i,\tilde\eta_j)_T-(\tilde\omega_i,\curl\tilde\eta_j)_T=\delta_{i,j}.
	$$ 
	
\end{remark}

%

%

\section{Basis functions of the three-dimensional nonconforming space \eqref{eq:deffems3d} for $H(\dv,\Omega)\cap H_0(\curl,\Omega)$}
\label{sec:app:basis-3d}

This section reports the explicit basis functions in three dimensions.

Let
\[
A = \int_T \widetilde{x}^{\,2}\,d\boldsymbol{x},
\qquad
B = \int_T \widetilde{y}^{\,2}\,d\boldsymbol{x},
\qquad
C = \int_T \widetilde{z}^{\,2}\,d\boldsymbol{x},
\]
and
\[
D = \int_T \widetilde{x}\widetilde{y}\,d\boldsymbol{x},
\qquad
E = \int_T \widetilde{x}\widetilde{z}\,d\boldsymbol{x},
\qquad
F = \int_T \widetilde{y}\widetilde{z}\,d\boldsymbol{x}.
\]
Set
\[
\mathsf{G}
=
\begin{pmatrix}
	4A+B+C & 3D & 3E\\
	3D & A+4B+C & 3F\\
	3E & 3F & A+B+4C
\end{pmatrix},
\qquad
\mathsf{R}
=
\mathsf{G}^{-1}.
\]
Let
\[
\widetilde{\boldsymbol{x}}
=
\begin{pmatrix}
	\widetilde{x}\\
	\widetilde{y}\\
	\widetilde{z}
\end{pmatrix},
\qquad
\widetilde{\boldsymbol{x}}_i
=
\begin{pmatrix}
	\widetilde{x}_i\\
	\widetilde{y}_i\\
	\widetilde{z}_i
\end{pmatrix},
\qquad i=1,\ldots,4.
\]
Define
\[
\mathsf{Q}_T(\widetilde{\boldsymbol{x}})
=
\begin{pmatrix}
	2\widetilde{x}^{\,2}-\widetilde{y}^{\,2}-\widetilde{z}^{\,2} & 0 & 0\\
	0 & 2\widetilde{y}^{\,2}-\widetilde{x}^{\,2}-\widetilde{z}^{\,2} & 0\\
	0 & 0 & 2\widetilde{z}^{\,2}-\widetilde{x}^{\,2}-\widetilde{y}^{\,2}
\end{pmatrix},
\]
and
\[
\mathsf{H}_T
=
\begin{pmatrix}
	7B+7C-2A & 0 & 0\\
	0 & 7A+7C-2B & 0\\
	0 & 0 & 7A+7B-2C
\end{pmatrix}.
\]
Define
\[
\mathsf{M}_T(\widetilde{\boldsymbol{x}})
=
-\frac{1}{3}\mathsf{Q}_T(\widetilde{\boldsymbol{x}})\mathsf{R}
-\frac{1}{9|T|}\mathsf{H}_T\mathsf{R}
+\frac{1}{9|T|}\mathsf{I}_3 .
\]
Then the four face-associated local basis functions are given by
\begin{equation}
	\fvarrho_i^T
	=
	\frac{1}{3|T|}
	\widetilde{\boldsymbol{x}}
	+
	\mathsf{M}_T(\widetilde{\boldsymbol{x}})
	\widetilde{\boldsymbol{x}}_i,
	\qquad i=1,\ldots,4 .
	\label{eq:rho-local-3d}
\end{equation}

For each edge \((i,j)\), \(1\leq i<j\leq4\), define
\[
\boldsymbol{d}_{ij}
=
\widetilde{\boldsymbol{x}}_j
-
\widetilde{\boldsymbol{x}}_i,
\qquad
\boldsymbol{\eta}_{ij}
=
\widetilde{\boldsymbol{x}}_i
\times
\widetilde{\boldsymbol{x}}_j .
\]
Define
\[
\mathsf{N}_T(\widetilde{\boldsymbol{x}})
=
\frac{1}{2}\mathsf{Q}_T(\widetilde{\boldsymbol{x}})\mathsf{R}
+
\frac{1}{6|T|}\mathsf{H}_T\mathsf{R}
-
\frac{2}{3|T|}\mathsf{I}_3 .
\]
The six edge-associated local basis functions are given by
\begin{equation}
	\fvarsigma_{ij}^T
	=
	\frac{1}{2|T|}
	\left(
	\boldsymbol{d}_{ij}
	\times
	\widetilde{\boldsymbol{x}}
	\right)
	+
	\mathsf{N}_T(\widetilde{\boldsymbol{x}})
	\boldsymbol{\eta}_{ij},
	\qquad 1\leq i<j\leq4 .
	\label{eq:sigma-local-3d}
\end{equation}

Then,
\begin{equation*}
	P_{\rm cd}(T)
	=
	\operatorname{span}
	\{
	\fvarrho_1^T,\fvarrho_2^T,\fvarrho_3^T,\fvarrho_4^T,
	\fvarsigma_{12}^T,\fvarsigma_{13}^T,\fvarsigma_{14}^T,
	\fvarsigma_{23}^T,\fvarsigma_{24}^T,\fvarsigma_{34}^T
	\}.
\end{equation*}

Let
\[
\mathcal E_T
=
\{(1,2),(1,3),(1,4),(2,3),(2,4),(3,4)\}.
\]
The local functions satisfy
\begin{equation}
	(\curl\fvarrho_i^T,N_{pq})
	-
	(\fvarrho_i^T,\curl N_{pq})
	=
	0,
	\qquad
	(\dv\fvarrho_i^T,1-2\lambda_k)
	+
	(\fvarrho_i^T,\nabla(1-2\lambda_k))
	=
	\delta_{ik},
	\label{eq:rho-dof-3d}
\end{equation}
for \(i,k=1,\ldots,4\) and \((p,q)\in\mathcal E_T\), and
\begin{equation}
	(\curl\fvarsigma_{ij}^T,N_{pq})
	-
	(\fvarsigma_{ij}^T,\curl N_{pq})
	=
	\delta_{(i,j),(p,q)},
	\qquad
	(\dv\fvarsigma_{ij}^T,1-2\lambda_k)
	+
	(\fvarsigma_{ij}^T,\nabla(1-2\lambda_k))
	=
	0,
	\label{eq:sigma-dof-3d}
\end{equation}
for \((i,j),(p,q)\in\mathcal E_T\) and \(k=1,\ldots,4\).

\begin{figure}[htbp]
	\centering
	\includegraphics[width=\textwidth]{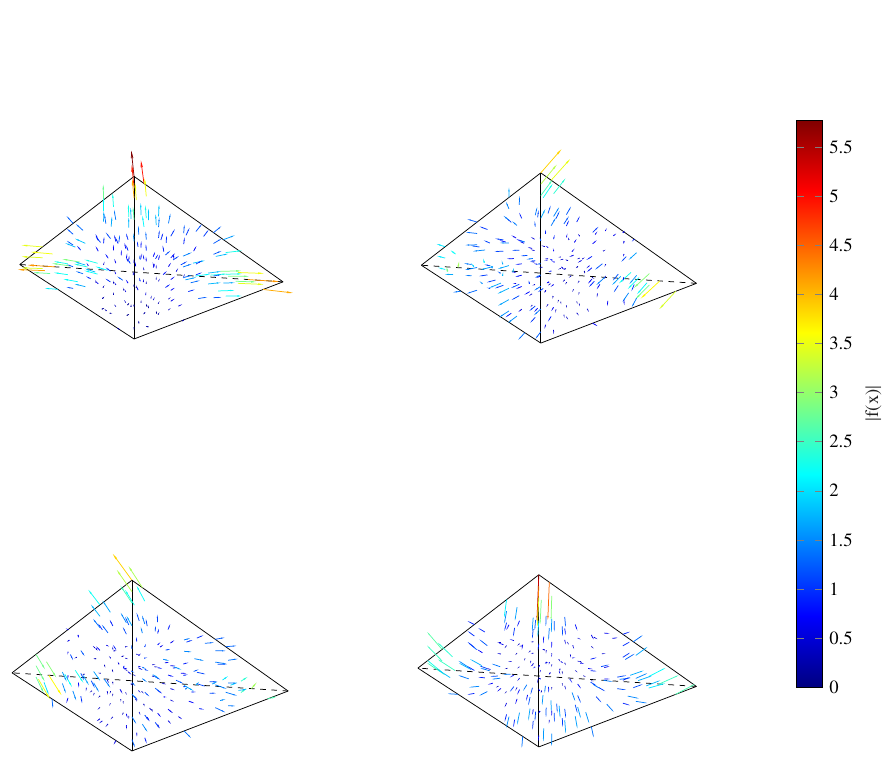}
	\caption{Local basis function \(\fvarrho_1^T\)--\(\fvarrho_4^T\)}
	\label{3Dbasis1}
\end{figure}

\begin{figure}[htbp]
	\centering
	\includegraphics[width=\textwidth]{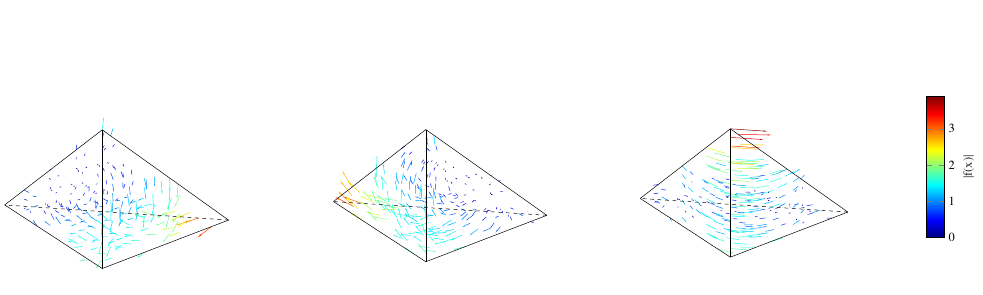}
	\caption{Local basis function \(\fvarsigma_{12}^T\), \(\fvarsigma_{13}^T\), and \(\fvarsigma_{14}^T\)}
	\label{3Dbasis2}
\end{figure}

\begin{figure}[htbp]
	\centering
	\includegraphics[width=\textwidth]{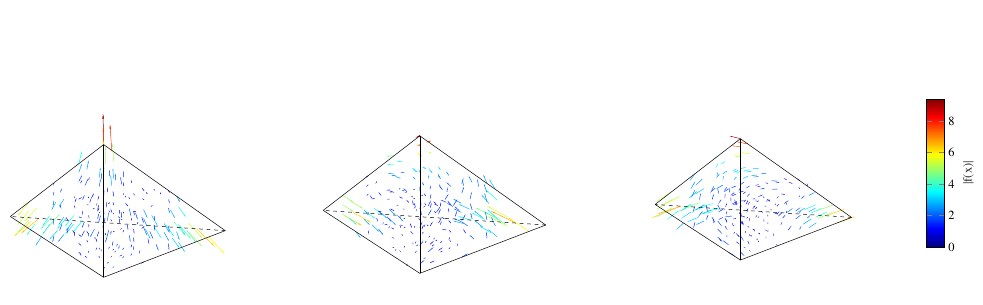}
	\caption{Local basis function \(\fvarsigma_{23}^T\), \(\fvarsigma_{24}^T\), and \(\fvarsigma_{34}^T\)}
	\label{f3Dbasis3}
\end{figure}

\begin{figure}[htbp]
	\centering
	
	\begin{minipage}[t]{0.48\textwidth}
		\centering
		\hspace*{-2cm}\includegraphics[width=1.8\linewidth]{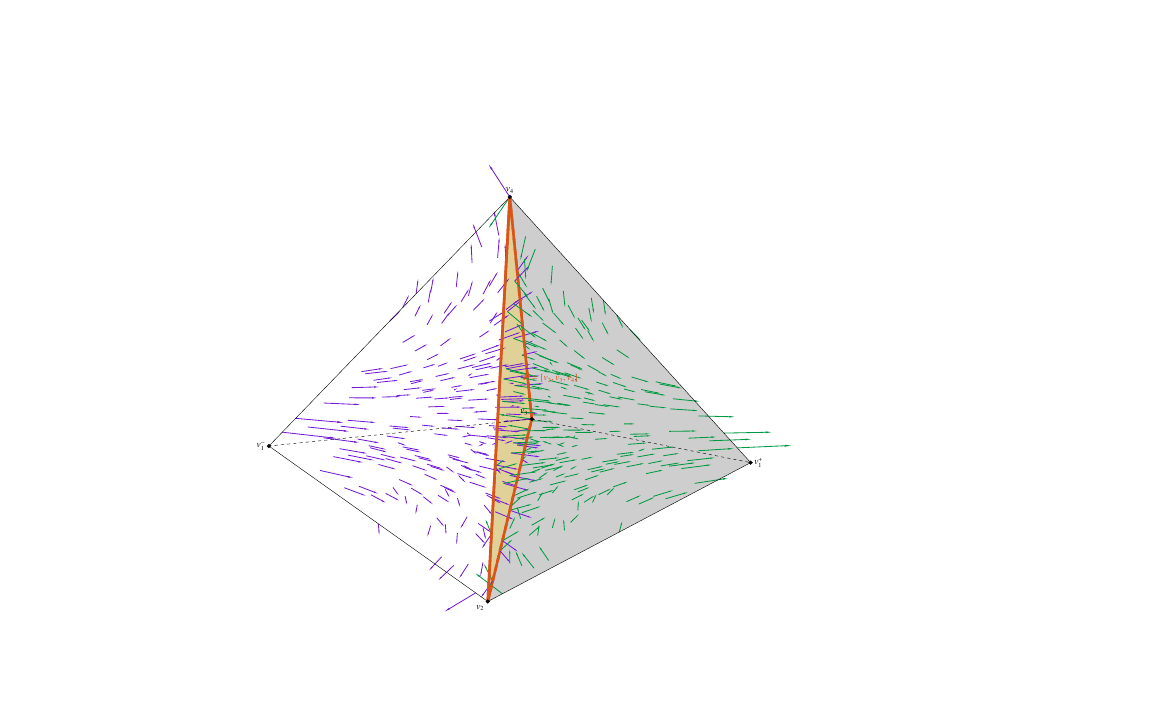}
		
		\vspace{-2mm}
		{\small (a) Face-associated global basis function $\fvarrho^T$}
	\end{minipage}
	\hspace{0.02\textwidth}
	\begin{minipage}[t]{0.48\textwidth}
		\centering
		\hspace*{-4.5cm}\includegraphics[width=2.5\linewidth]{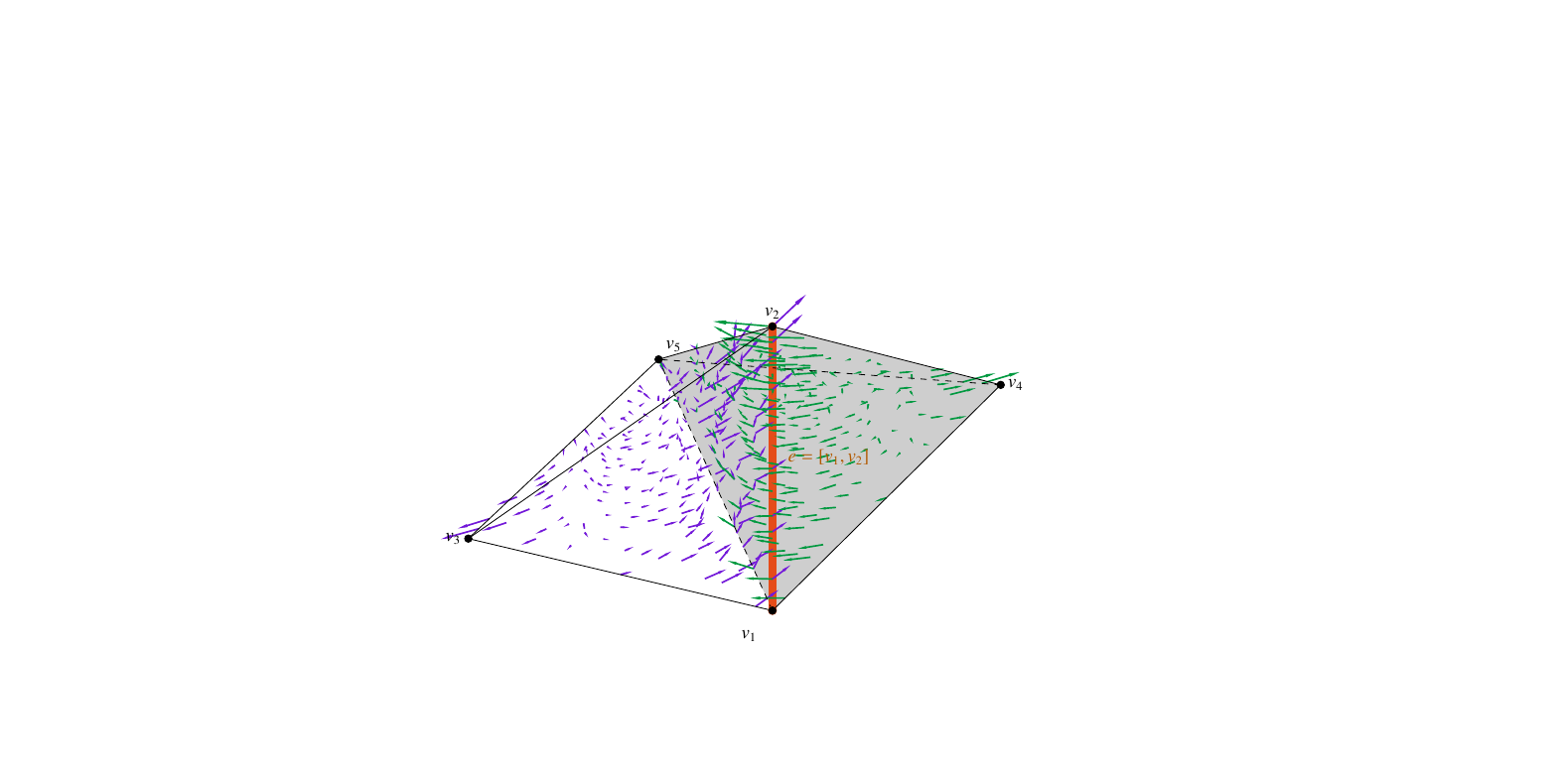}
		
		\vspace{-2mm}
		{\small (b) Edge-associated global basis function $\fvarsigma^T$}
	\end{minipage}
	
	\label{fig:3d_basis_functions}
\end{figure}

\begin{remark}
	The explicit basis functions $\fvarrho^T$ and $\fvarsigma^T$ can be constructed through the following three steps.
	
	Let $T$ be a tetrahedron. Denote by $\{\psi_j\}_{j=1}^4$ the Crouzeix--Raviart basis functions and by
	\[
	\{\tau_{12},\tau_{13},\tau_{14},\tau_{23},\tau_{24},\tau_{34}\}
	\]
	the Nedelec basis functions associated with the ordered edges
	\[
	\mathcal E_T=\{(12),(13),(14),(23),(24),(34)\}.
	\]
	We seek functions
	\[
	\widetilde{\omega}_1,\ldots,\widetilde{\omega}_4,\quad
	\widetilde{\omega}_{12},\widetilde{\omega}_{13},\widetilde{\omega}_{14},
	\widetilde{\omega}_{23},\widetilde{\omega}_{24},\widetilde{\omega}_{34}
	\in P_{\mathrm{rd}}(T),
	\]
	such that
	\begin{equation}
		\int_T \Big(
		\operatorname{div}\widetilde{\omega} \, \psi_j
		+ \widetilde{\omega} \cdot \nabla \psi_j
		\Big) dx
		=
		\begin{cases}
			\delta_{ij},
			& \widetilde{\omega}=\widetilde{\omega}_i,\quad i=1,\ldots,4, \\[1mm]
			0,
			& \widetilde{\omega}=\widetilde{\omega}_{mn},\quad (mn)\in\mathcal E_T,
		\end{cases}
		\qquad j=1,\ldots,4,
		\label{eq:dual-cr-condition}
	\end{equation}
	and
	\begin{equation}
		\int_T \Big(
		\operatorname{curl}\widetilde{\omega} \cdot \tau_{pq}
		- \widetilde{\omega} \cdot \operatorname{curl}\tau_{pq}
		\Big) dx
		=
		\begin{cases}
			0,
			& \widetilde{\omega}=\widetilde{\omega}_i,\quad i=1,\ldots,4, \\[1mm]
			\delta_{(mn),(pq)},
			& \widetilde{\omega}=\widetilde{\omega}_{mn},\quad (mn)\in\mathcal E_T,
		\end{cases}
		\qquad (pq)\in\mathcal E_T .
		\label{eq:dual-nedelec-condition}
	\end{equation}
	Here $\delta_{(mn),(pq)}$ denotes the Kronecker delta with respect to the ordered edge indices.
	
	We first introduce the following ordered basis of $P_{\mathrm{rd}}(T)$:
	\begin{align*}
		\omega_1 &=
		\begin{pmatrix}
			\widetilde{x} \\ \widetilde{y} \\ \widetilde{z}
		\end{pmatrix},
		\omega_2 =
		\begin{pmatrix}
			0 \\ -\widetilde{z} \\ \widetilde{y}
		\end{pmatrix},
		\omega_3 =
		\begin{pmatrix}
			\widetilde{z} \\ 0 \\ -\widetilde{x}
		\end{pmatrix},
		\omega_4 =
		\begin{pmatrix}
			-\widetilde{y} \\ \widetilde{x} \\ 0
		\end{pmatrix},
		\\[0.6em]
		\omega_5 &=
		\begin{pmatrix}
			2\widetilde{x}^2-(\widetilde{y}^2+\widetilde{z}^2)
			+\dfrac{B+C-2A}{V}
			\\ 0 \\ 0
		\end{pmatrix},
		\omega_6 =
		\begin{pmatrix}
			0 \\
			2\widetilde{y}^2-(\widetilde{x}^2+\widetilde{z}^2)
			+\dfrac{A+C-2B}{V}
			\\ 0
		\end{pmatrix},
		\\[0.6em]
		\omega_7 &=
		\begin{pmatrix}
			0 \\ 0 \\
			2\widetilde{z}^2-(\widetilde{x}^2+\widetilde{y}^2)
			+\dfrac{A+B-2C}{V}
		\end{pmatrix},
		\omega_8 =
		\begin{pmatrix}
			1 \\ 0 \\ 0
		\end{pmatrix},
		\omega_9 =
		\begin{pmatrix}
			0 \\ 1 \\ 0
		\end{pmatrix},
		\omega_{10} =
		\begin{pmatrix}
			0 \\ 0 \\ 1
		\end{pmatrix}.
	\end{align*}
	
	We also introduce the following ordered basis for the auxiliary product space:
	\begin{align*}
		(\psi,\eta)_1 &=
		\left(
		1,
		\begin{pmatrix}
			0 \\ 0 \\ 0
		\end{pmatrix}
		\right),
		&
		(\psi,\eta)_2 &=
		\left(
		0,
		\begin{pmatrix}
			1 \\ 0 \\ 0
		\end{pmatrix}
		\right),
		&
		(\psi,\eta)_3 &=
		\left(
		0,
		\begin{pmatrix}
			0 \\ 1 \\ 0
		\end{pmatrix}
		\right),
		&
		(\psi,\eta)_4 &=
		\left(
		0,
		\begin{pmatrix}
			0 \\ 0 \\ 1
		\end{pmatrix}
		\right),
		\\[0.6em]
		(\psi,\eta)_5 &=
		\left(
		2\widetilde{x},
		\begin{pmatrix}
			0 \\ -\widetilde{z} \\ \widetilde{y}
		\end{pmatrix}
		\right),
		&
		(\psi,\eta)_6 &=
		\left(
		2\widetilde{y},
		\begin{pmatrix}
			\widetilde{z} \\ 0 \\ -\widetilde{x}
		\end{pmatrix}
		\right),
		&
		(\psi,\eta)_7 &=
		\left(
		2\widetilde{z},
		\begin{pmatrix}
			-\widetilde{y} \\ \widetilde{x} \\ 0
		\end{pmatrix}
		\right),
		\\[0.6em]
		(\psi,\eta)_8 &=
		\left(
		2\widetilde{x},
		\begin{pmatrix}
			0 \\ \widetilde{z} \\ -\widetilde{y}
		\end{pmatrix}
		\right),
		&
		(\psi,\eta)_9 &=
		\left(
		2\widetilde{y},
		\begin{pmatrix}
			-\widetilde{z} \\ 0 \\ \widetilde{x}
		\end{pmatrix}
		\right),
		&
		(\psi,\eta)_{10} &=
		\left(
		2\widetilde{z},
		\begin{pmatrix}
			\widetilde{y} \\ -\widetilde{x} \\ 0
		\end{pmatrix}
		\right).
	\end{align*}
	
	For $\omega\in P_{\mathrm{rd}}(T)$ and $(\psi,\eta)$ in the above auxiliary product space, define
	\begin{equation}
		T\omega :=
		\bigl(
		\operatorname{div}\omega,
		\operatorname{curl}\omega
		\bigr),
		\qquad
		T^*(\psi,\eta)
		:=
		\nabla\psi-\operatorname{curl}\eta .
		\label{eq:T-and-adjoint}
	\end{equation}
	The associated bilinear form is given by
	\begin{equation}
		a\bigl(\omega,(\psi,\eta)\bigr)
		:=
		\bigl(\operatorname{div}\omega,\psi\bigr)_T
		+
		\bigl(\operatorname{curl}\omega,\eta\bigr)_T
		+
		\bigl(\omega,\nabla\psi-\operatorname{curl}\eta\bigr)_T .
		\label{eq:bilinear-form}
	\end{equation}
	
	The corresponding $L^2$ pairing matrix is
	\[
	\resizebox{\textwidth}{!}{$
		L_1 =
		\begin{bmatrix}
			3V&0&0&0&0&0&0&0&0&0 \\
			0&2V&0&0&0&0&0&0&0&0 \\
			0&0&2V&0&0&0&0&0&0&0 \\
			0&0&0&2V&0&0&0&0&0&0 \\
			0&0&0&0&8A+2B+2C&6D&6E&8A-2B-2C&10D&10E \\
			0&0&0&0&6D&8B+2A+2C&6F&10D&8B-2A-2C&10F \\
			0&0&0&0&6E&6F&8C+2A+2B&10E&10F&8C-2A-2B \\
			0&0&0&0&0&0&0&4V&0&0 \\
			0&0&0&0&0&0&0&0&4V&0 \\
			0&0&0&0&0&0&0&0&0&4V
		\end{bmatrix}.
		$}
	\]
	
	Its inverse is
	\[
	L_2 =
	\begin{bmatrix}
		\dfrac{1}{3V} & 0_{1\times 3} & 0_{1\times 3} & 0_{1\times 3} \\[2mm]
		0_{3\times 1} & \dfrac{1}{2V}I_3 & 0_{3\times 3} & 0_{3\times 3} \\[2mm]
		0_{3\times 1} & 0_{3\times 3} & \dfrac{1}{2}\mathsf{R}
		& \dfrac{M}{2}\mathsf{R}-\dfrac{5}{12V}I_3 \\[2mm]
		0_{3\times 1} & 0_{3\times 3} & 0_{3\times 3}
		& \dfrac{1}{4V}I_3
	\end{bmatrix}.
	\]
	Here \(I_3\) denotes the \(3\times3\) identity matrix,
	\[
	M=\frac{4(A+B+C)}{3V},
	\]
	and
	\[
	\mathsf{R}=(r_{ij})_{i,j=1}^{3}
	=
	\mathsf{G}^{-1},
	\]
	where
	\[
	\mathsf{G}
	=
	\begin{pmatrix}
		4A+B+C & 3D & 3E \\
		3D & A+4B+C & 3F \\
		3E & 3F & A+B+4C
	\end{pmatrix}.
	\]
	
	In the second step, the auxiliary product-space basis is represented in terms
	of the Crouzeix--Raviart and Nedelec basis functions. To write the corresponding
	representation matrix compactly, let
	\[
	\widetilde{\boldsymbol x}_i
	=
	(\widetilde{x}_i,\widetilde{y}_i,\widetilde{z}_i)^T,
	\qquad i=1,\ldots,4,
	\]
	and define
	\[
	\mathsf X
	=
	\begin{bmatrix}
		\widetilde{\boldsymbol x}_1&
		\widetilde{\boldsymbol x}_2&
		\widetilde{\boldsymbol x}_3&
		\widetilde{\boldsymbol x}_4
	\end{bmatrix}.
	\]
	For the ordered edges
	\[
	(12),(13),(14),(23),(24),(34),
	\]
	set
	\[
	\mathsf{\Delta}_{\rm e}
	=
	\begin{bmatrix}
		\widetilde{\boldsymbol x}_2-\widetilde{\boldsymbol x}_1&
		\widetilde{\boldsymbol x}_3-\widetilde{\boldsymbol x}_1&
		\widetilde{\boldsymbol x}_4-\widetilde{\boldsymbol x}_1&
		\widetilde{\boldsymbol x}_3-\widetilde{\boldsymbol x}_2&
		\widetilde{\boldsymbol x}_4-\widetilde{\boldsymbol x}_2&
		\widetilde{\boldsymbol x}_4-\widetilde{\boldsymbol x}_3
	\end{bmatrix},
	\]
	and
	\[
	\mathsf{\Xi}_{\rm e}
	=
	\begin{bmatrix}
		\widetilde{\boldsymbol x}_1\times\widetilde{\boldsymbol x}_2&
		\widetilde{\boldsymbol x}_1\times\widetilde{\boldsymbol x}_3&
		\widetilde{\boldsymbol x}_1\times\widetilde{\boldsymbol x}_4&
		\widetilde{\boldsymbol x}_2\times\widetilde{\boldsymbol x}_3&
		\widetilde{\boldsymbol x}_2\times\widetilde{\boldsymbol x}_4&
		\widetilde{\boldsymbol x}_3\times\widetilde{\boldsymbol x}_4
	\end{bmatrix}.
	\]
	Then the representation matrix is
	\[
	L_3
	=
	\begin{bmatrix}
		\boldsymbol 1_4^T & 0_{1\times 6} \\[1mm]
		0_{3\times 4} & \mathsf{\Delta}_{\rm e} \\[1mm]
		-\dfrac{2}{3}\mathsf X & \mathsf{\Xi}_{\rm e} \\[1mm]
		-\dfrac{2}{3}\mathsf X & -\mathsf{\Xi}_{\rm e}
	\end{bmatrix}.
	\]
	
	Finally, the desired basis functions are obtained by
	\begin{equation}
		(\widetilde{\omega}_1,\widetilde{\omega}_2,\widetilde{\omega}_3,\widetilde{\omega}_4,
		\widetilde{\omega}_{12},\widetilde{\omega}_{13},\widetilde{\omega}_{14},
		\widetilde{\omega}_{23},\widetilde{\omega}_{24},\widetilde{\omega}_{34})
		=
		(\omega_1,\omega_2,\ldots,\omega_{10})L_2^{\mathsf T}L_3 .
		\label{eq:constructed-basis}
	\end{equation}
\end{remark}

\section{Numerical results for boundary-value problems}
\label{app:bvp-results}

This appendix reports boundary-value experiments for the proposed element, supplementing the eigenvalue tests in Section~\ref{sec:numer}.
We consider manufactured solutions on $[0,1]^2$ and $[0,1]^3$ and monitor both their mesh-refinement errors in the log--log plots in Figure~\ref{fig:bvp-convergence}.

\subsection{Two-dimensional boundary-value problem}\label{app:bvp-2d}

We first consider the two-dimensional boundary-value problem in
\(H(\rot,\Omega)\cap H_0(\dv,\Omega)\). The continuous problem is to find
\(\omega \in H(\rot,\Omega)\cap H_0(\dv,\Omega)\) such that
\begin{equation}\label{bvp_primal2D}
	(\dv\omega,\dv\tau) + (\rot\omega,\rot\tau) = (f,\tau),
	\qquad
	\forall\,\tau \in H(\rot,\Omega)\cap H_0(\dv,\Omega).
\end{equation}
Here, in two dimensions, \(\rot\) denotes the scalar rotation of a vector
field.

The corresponding discrete problem is to find
\(\omega_h\in \fV^{\rm r\mathring d}_h\) such that
\begin{equation}\label{eq:2d-bvp-discrete}
	(\dv_h\omega_h,\dv_h\tau_h)
	+
	(\rot_h\omega_h,\rot_h\tau_h)
	=
	(f,\tau_h),
	\qquad
	\forall\,\tau_h\in \fV^{\rm r\mathring d}_h .
\end{equation}
Here \(\dv_h\) and \(\rot_h\) are understood elementwise.

To test the convergence behavior, we take \(\Omega=[0,1]^2\) and prescribe the
exact solution
\begin{equation*}
	\tilde u =
	\bigl(\sin(\pi x)\cos(\pi y),\,\cos(\pi x)\sin(\pi y)\bigr),
\end{equation*}
which gives
\begin{equation*}
	f =
	2\pi^2
	\bigl(\sin(\pi x)\cos(\pi y),\,\cos(\pi x)\sin(\pi y)\bigr).
\end{equation*}
\subsection{Three-dimensional boundary-value problem}\label{app:bvp-3d}

We next consider the three-dimensional boundary-value problem in
\(H(\dv,\Omega)\cap H_0(\curl,\Omega)\). The continuous problem is to find
\(\omega \in H(\dv,\Omega)\cap H_0(\curl,\Omega)\) such that
\begin{equation}\label{eq:3d-bvp}
	(\dv\omega,\dv\tau) + (\curl\omega,\curl\tau)
	= (f,\tau),
	\qquad
	\forall\,\tau\in H(\dv,\Omega)\cap H_0(\curl,\Omega).
\end{equation}

The corresponding discrete problem is to find
\(\omega_h\in \fV^{\mathring{\mathrm r}{\rm d}}_h\) such that
\begin{equation}\label{eq:3d-bvp-discrete}
	(\dv_h\omega_h,\dv_h\tau_h)
	+
	(\curl_h\omega_h,\curl_h\tau_h)
	=
	(f,\tau_h),
	\qquad
	\forall\,\tau_h\in \fV^{\mathring{\mathrm r}{\rm d}}_h .
\end{equation}
Here \(\dv_h\) and \(\curl_h\) are understood elementwise.

To test the convergence behavior, we take \(\Omega=[0,1]^3\) and prescribe the
exact solution
\[
\tilde u =
\bigl(
\cos(\pi x)\sin(\pi y)\sin(\pi z),\,
\cos(\pi y)\sin(\pi x)\sin(\pi z),\,
\cos(\pi z)\sin(\pi x)\sin(\pi y)
\bigr).
\]
Accordingly,
\[
f = 3\pi^2
\bigl(
\cos(\pi x)\sin(\pi y)\sin(\pi z),\,
\cos(\pi y)\sin(\pi x)\sin(\pi z),\,
\cos(\pi z)\sin(\pi x)\sin(\pi y)
\bigr).
\]
The corresponding log--log convergence plots are shown in
Figure~\ref{fig:bvp-convergence}.

\begin{figure}[htbp]
	\centering
	\begin{minipage}[t]{0.49\textwidth}
		\centering
		\includegraphics[width=\linewidth]{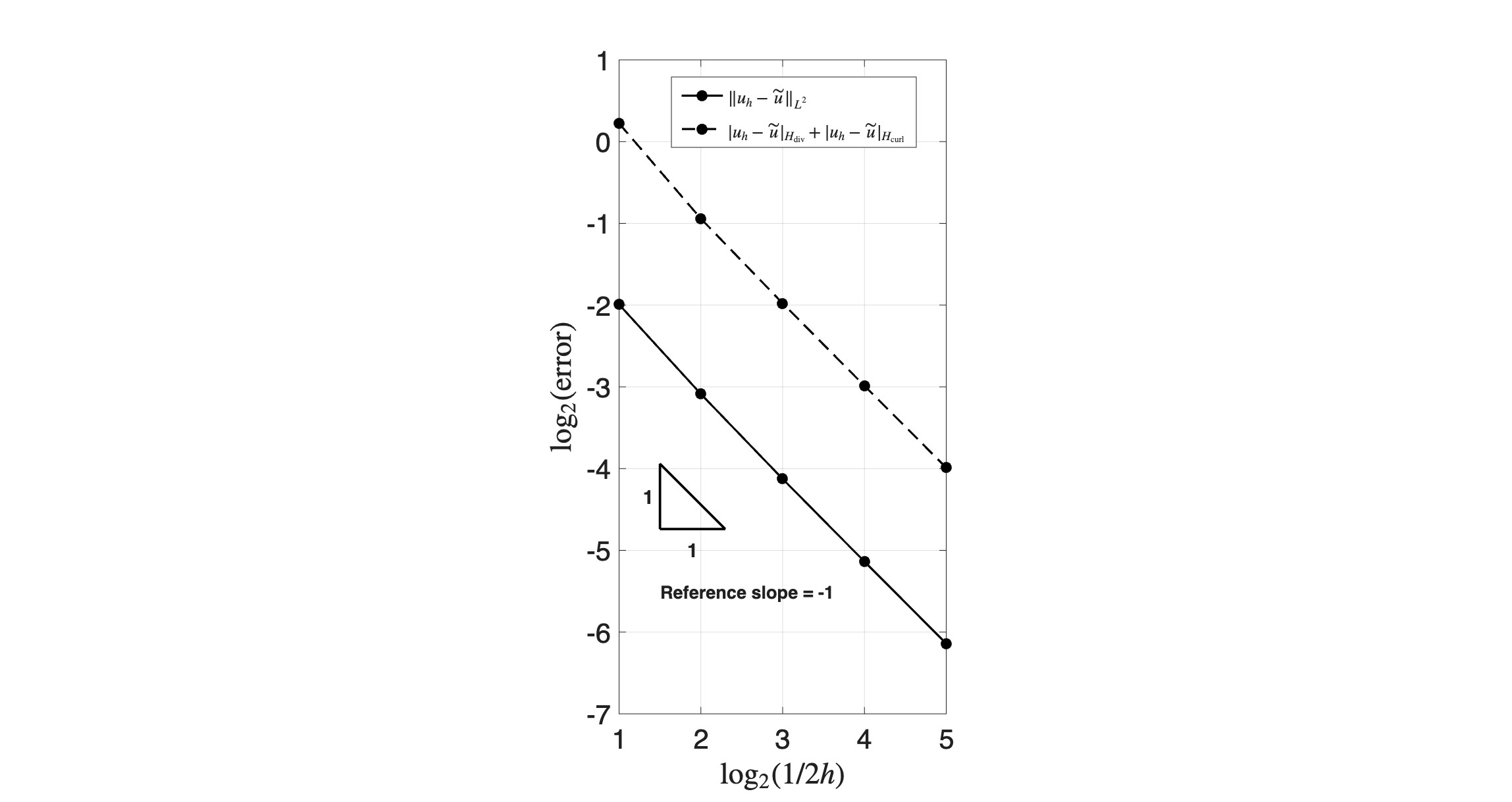}
	\end{minipage}%
	\hfill
	\begin{minipage}[t]{0.49\textwidth}
		\centering
		\includegraphics[width=\linewidth]{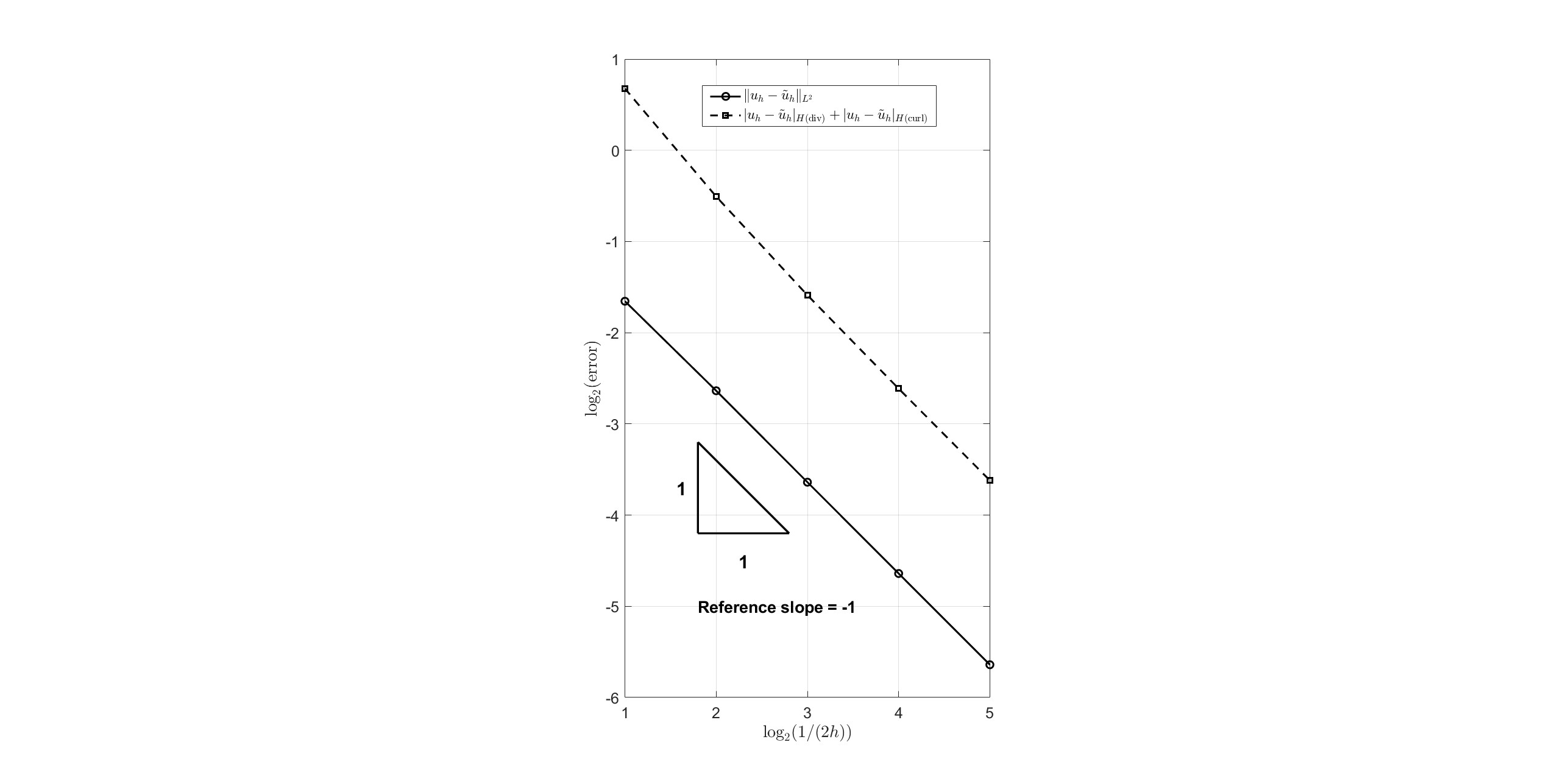}
	\end{minipage}
	\caption{Log--log mesh-refinement errors for manufactured boundary-value problems on $[0,1]^2$ \eqref{eq:2d-bvp-discrete} (left) and $[0,1]^3$ \eqref{eq:3d-bvp-discrete} (right). Solid curves: $\|u_h-\tilde u\|_{L^2}$; dashed curves: $|u_h-\tilde u|_{H_{\mathrm{div}}}+|u_h-\tilde u|_{H_{\mathrm{curl}}}$. Reference slope~$-1$.}
	\label{fig:bvp-convergence}
\end{figure}

\paragraph{\bf Summary}
In both dimensions, the error curves in Figure~\ref{fig:bvp-convergence} are parallel to the reference slope~$-1$, indicating first-order mesh convergence for the manufactured smooth solutions.

\end{document}